\documentclass[11pt, twoside]{article}

\usepackage[english]{babel}

\usepackage{amssymb}
\usepackage{mathrsfs}
\usepackage{amsmath}
\usepackage{amsthm}
\usepackage{amsfonts}
\usepackage{latexsym}
\usepackage{indentfirst}
\usepackage{color}
\usepackage{txfonts}
\usepackage{enumerate}

\usepackage{anysize}

\usepackage{tikz}

\usepackage[colorlinks=true,
linkcolor=blue,
citecolor=red,
urlcolor=magenta,
backref=page
]{hyperref}

\allowdisplaybreaks

\hypersetup{pdftitle={Sharp Nonlocal Poincar\'e Inequalities and Liminf Lower Bounds
for Sobolev and BV Energies on Metric Measure Spaces},
pdfauthor={Feng Dai, Tuomas P. Hyt\"onen, Dachun Yang, Wen Yuan and Yirui Zhao}}

\newtheorem{theorem}{Theorem}[section]
\newtheorem{lemma}[theorem]{Lemma}
\newtheorem{corollary}[theorem]{Corollary}
\newtheorem{proposition}[theorem]{Proposition}
\theoremstyle{definition}
\newtheorem{remark}[theorem]{Remark}

\numberwithin{equation}{section}

\begin{document}
\title{\bf\Large Sharp Nonlocal Poincar\'e Inequalities and Liminf Lower Bounds
for Sobolev and BV Energies on Metric Measure Spaces
\footnotetext{\hspace{-0.26cm}2020
\emph{Mathematics Subject Classification}. Primary 46E36;
Secondary 30L15, 26D10, 26B30.
\endgraf\emph{Key words and phrases}.
nonlocal Poincar\'e inequality, nonlocal nonconvex functional,
Sobolev function, function of bounded variation,
metric measure space.
\endgraf
This project is partially supported by the
National Natural Science Foundation of China
(Grant Nos. 12431006, 12371093, and 124B2004), the Beijing Natural
Science Foundation (Grant No. 1262011), the Fundamental Research
Funds for the Central Universities (Grant No. 2253200028), and
the Research Council of Finland (Grant Nos. 364208 and 371637).
Feng Dai is also supported by the NSERC of Canada Discovery Grant
(Grant No. RGPIN 2026-05667).}}
\author{Feng Dai, Tuomas P. Hyt\"onen, Dachun Yang,
%\thanks{Corresponding author, E-mail: \texttt{dcyang@bnu.edu.cn}.}
Wen Yuan and Yirui Zhao}
%\date{}
\maketitle
\vspace{-0.7cm}
\begin{center}
\begin{minipage}{13cm}
{\small {\bf Abstract.}\quad
Let $p\in[1,\infty)$, and let
$(X,\rho,\mu)$ be a complete doubling $q$-Poincar\'e space for some
$q\in[1,p]$ with a lower dimension $d$ and an upper dimension $D$.
In this article, we establish Poincar\'e inequalities for two families
of nonlocal nonconvex functionals defined using distances or
volumes.
More precisely, for any ball $B$ with its radius
$r_B\in(0,\operatorname{diam} X)$, any $\lambda\in(0,\infty)$, and any
measurable function $f:X\to\mathbb R$,
\begin{align*}
 &\fint_B\fint_B|f(x)-f(y)|^p\,d\mu(x)\,d\mu(y)\\
 &\quad\lesssim\lambda^pr_B^p\phi(B)^{\gamma p}
 +\frac{\lambda^pr_B^p}{\mu(B)}
 \mathop{\int_{\Lambda B}\int_{\Lambda B}}\limits_{
 \genfrac{}{}{}{2}{|f(x)-f(y)|}{\rho(x,y)\phi(x,y)^\gamma}>\lambda}
 \frac{\phi(x,y)^{\gamma p}}{V(x,y)}\,d\mu(x)\,d\mu(y),
\end{align*}
where $V(x,y):=\mu(B(x,\rho(x,y)))$ and either $\phi=\rho$ with
$\phi(B)=r_B$ and $\gamma\in(-\infty,-\frac{\min\{q,D\}}{p}]$
or $\phi=V$ with $\phi(B)=\mu(B)$ and
$\gamma\in(-\infty,-\frac{q}{p\max\{q,d\}}]$.
Here both ranges of $\gamma$ are sharp over the class of $q$-Poincar\'e spaces
with the prescribed lower and upper dimensions, and the constant $\Lambda\in[1,\infty)$
and the implicit positive constant depend only on $p$, $\gamma$,
and the geometric constants of $X$.
We also obtain lower bounds for Sobolev energy and total variation
in terms of the limit inferior of these functionals, for $\gamma$ in either the above negative ranges or
$(0,\infty)$.
Our proofs use the curve characterization of Poincar\'e spaces
and an adaptive dyadic refinement argument based on medians.
}
\end{minipage}
\end{center}
\vspace{0.1cm}

\tableofcontents

\section{Introduction}\label{sec:main-results}

The Poincar\'e inequality is a basic tool in analysis on metric measure
spaces. It controls the mean oscillation of a function by its upper
gradients and connects the geometry of a space with the first-order
regularity of functions on it; see, for example,
\cite{HeinonenKoskela98,bb11,HKST}.
Beyond the gradient, some variants of Poincar\'e inequalities
involving nonlocal approximations of Sobolev seminorms and total
variation have been extensively studied in Euclidean spaces and
extended to non-Euclidean settings;
see, for example,
\cite{Barbieri2011,BBM2001,BBM2002,BrezisNguyenImages,LinEtAlCompactness2026,MallickNguyen2020,
Nguyen2011,NguyenSquassina2019,Picenni2026,Ponce2004}.
In this article, we study Poincar\'e inequalities for a family of
nonlocal nonconvex functionals in complete doubling Poincar\'e
spaces and derive the corresponding liminf lower bounds for Sobolev energy and total
variation.

\subsection{Background}
We first recall the Euclidean functionals introduced by Nguyen
\cite{Nguyen2006} and further studied in
\cite{BourgainNguyen2006,BrezisNguyenImages}.
For $p\in[1,\infty)$,
$\lambda\in(0,\infty)$, a measurable set $E\subset\mathbb R^N$,
and a measurable function $u:E\to\mathbb R$, consider
\begin{equation}\label{eq:nguyen-functional}
 I_{\lambda,p}(u;E)
 :=\lambda^p
 \iint_{\genfrac{}{}{0pt}{}{x,y\in E}{|u(x)-u(y)|>\lambda}}
 \frac{dx\,dy}{|x-y|^{N+p}}.
\end{equation}
For $p\in(1,\infty)$, the behavior of $I_{\lambda,p}$
on $\mathbb R^N$ as
$\lambda\to0^+$ characterizes $W^{1,p}(\mathbb R^N)$ among
functions in $L^p(\mathbb R^N)$; see
\cite{BourgainNguyen2006,Nguyen2006}.
Nguyen \cite[Theorem~1(a)]{Nguyen2011} further established a
Poincar\'e inequality with $I_{\lambda,p}$ in place of the gradient
energy. More precisely, for
$p\in[1,\infty)$, $\lambda\in(0,\infty)$, a ball $B\subset\mathbb R^N$,
and a measurable function $u:B\to\mathbb R$,
\begin{equation}\label{eq:nguyen-poincare}
 \fint_B\fint_B|u(x)-u(y)|^p\,dx\,dy
 \lesssim\lambda^p+r_B^{p-N}I_{\lambda,p}(u;B),
\end{equation}
where $r_B$ denotes the radius of $B$, $\fint_B$ denotes the integral
average over $B$, and the implicit positive constant depends only on
$N$ and $p$.
This inequality requires no a priori smoothness of $u$ and recovers
the classical Poincar\'e inequality for $u\in C_{\mathrm c}^\infty(B)$
by letting $\lambda\to0^+$. Nguyen
\cite[p.\,489]{Nguyen2011} noted that similar results in
Carnot--Carath\'eodory spaces or in metric spaces with appropriate
properties would be desirable.
For further results on asymptotic formulae and characterizations
associated with $I_{\lambda,p}$ and related functionals, including
extensions to non-Euclidean settings, we refer to
\cite{BrezisNguyenImages,BrezisNguyen2020,CuiLamLu2014,
DiMarinoSquassina19,LahtiPinamontiZhou2024,Nguyen2008,
NguyenEtAlMagnetic2018,NguyenSquassinaAnisotropic2019};
for applications of Nguyen's Poincar\'e inequality, we refer to
\cite{BrezisNguyenVMO,BrezisNguyenImages,MallickNguyen2020,
Nguyen2011,NguyenDegree2017,NguyenSquassina2019}.

By introducing a real parameter $\gamma$, Brezis et al.\
\cite{BSVY} placed the functionals $I_{\lambda,p}$ in
\eqref{eq:nguyen-functional} within a broader
framework.
More precisely, for $N\in\mathbb N$, $p\in[1,\infty)$,
$\lambda\in(0,\infty)$, $\gamma\in\mathbb R$, a measurable set
$E\subset\mathbb R^N$, and a measurable function $u:E\to\mathbb R$,
their functionals are defined by
\begin{equation}\label{eq:euclidean-functional}
 \Phi_{\lambda,p}^{\gamma}(u;E)
 :=\lambda^p
 \iint_{\genfrac{}{}{0pt}{}{x,y\in E}{
 |u(x)-u(y)|>\lambda|x-y|^{1+\gamma}}}
 |x-y|^{p\gamma-N}\,dx\,dy.
\end{equation}
The case $\gamma=-1$ gives $I_{\lambda,p}$, while
$\gamma=\frac{N}{p}$ corresponds to the  weak-type
representation of Sobolev seminorms obtained in
\cite{BVY2021}.
For $\gamma\in\mathbb R\setminus\{0\}$ and
$u\in C_{\mathrm c}^{\infty}(\mathbb R^N)$,
Brezis et al.\ \cite[Theorem~1.1]{BSVY} proved the limit formula
\[
 \lim_{\lambda\to L_\gamma}
 \Phi_{\lambda,p}^{\gamma}(u;\mathbb R^N)
 =\frac{K_{N,p}}{p|\gamma|}
 \int_{\mathbb R^N}|\nabla u(x)|^p\,dx,
\]
where $\lambda\to L_\gamma$ means that $\lambda\to0^+$ when
$\gamma\in(-\infty,0)$ and $\lambda\to\infty$ when
$\gamma\in(0,\infty)$.
Here $K_{N,p}:=\int_{\mathbb S^{N-1}}|\omega_1|^p\,d\sigma(\omega)$
with $\sigma$ denoting surface measure.
These formulae have stimulated extensive research; see, for example,
\cite{BVYLorentz2021,DaiEtAlOperators2026,DaiEtAlBall2023,
DominguezMilman,DominguezEtAlBesov2023,Frank2024,
HuEtAlHigherOrder2025,LiEtAlWeighted2026,Poliakovsky2022,
ZhaoEtAlFrank2024}
for extensions beyond difference quotients and classical Sobolev
spaces and
\cite{dlyyz26b,DaiEtAlMetric2022,HanXuZhu2025,HK:W1p,
HytonenYangYuanZhao2026} for developments on metric measure spaces.

Two questions concerning the nonlocal nonconvex functionals
\eqref{eq:euclidean-functional} are particularly relevant here.
Brezis et al.\ \cite[Section~7B]{BSVY} asked whether the corresponding
limit inferior controls the Sobolev seminorm or the total variation
without an a priori Sobolev or bounded variation assumption; see also
\cite[Open Question~14]{NguyenSurvey}.
Nguyen \cite[Open Question~33]{NguyenSurvey} also asked whether
\eqref{eq:nguyen-poincare} extends to the functionals in
\eqref{eq:euclidean-functional}.
In Euclidean spaces, Gobbino and Picenni
\cite[Theorem~1.1]{gp25} resolved the liminf question for
$\gamma\in(0,\infty)$.
Nguyen \cite[Theorem~16]{NguyenSurvey} proved the liminf lower bounds
for $\gamma\in[-1,-\frac1p]$ and $u\in L^p(\mathbb R^N)$,
and the remaining range $\gamma\in(-\infty,-1)$ was treated by
Lin et al.\ \cite[Theorem~1.4(ii)]{LinEtAlFullRange2026}
and Picenni \cite[Corollary~1.2]{Picenni2026}.
When $\gamma\in(-\frac1p,0)$, however,
$\Phi_{\lambda,p}^{\gamma}(\mathbf1_{(0,1)^N};\mathbb R^N)\to0$
as $\lambda\to0^+$, so the lower bound fails; see
\cite[Proposition~17]{NguyenSurvey}.
For the generalized Poincar\'e inequality involving
$\Phi_{\lambda,p}^{\gamma}$, Dai et al.\ \cite{DYYZTruncated26}
showed that it holds
for $\gamma\in(-\infty,-\frac1p]$ and that this range of $\gamma$
is optimal, using sharp truncated difference estimates
which also give a unified proof of the aforementioned
liminf lower bounds.
Picenni \cite[Theorem~1.3]{Picenni2026} also proves a
Poincar\'e inequality on bounded convex Euclidean domains for
$\gamma<-1$ in the normalization used in this article.
For $\gamma>0$, the corresponding Poincar\'e inequality with an
additive $\lambda^p$ term fails even on a fixed Euclidean domain;
\cite[Proposition~4.4 and Theorem~1.4]{Picenni2026} gives the
obstruction and an oscillation-dependent replacement.
Our complementary truncated estimate in
Lemma~\ref{lem:curve-local-estimates}(ii) supplies the lower bound
for positive parameters without asserting \eqref{eq:NPI} in that range.

However, the picture for these two questions on metric measure
spaces remains incomplete.
Di Marino and Squassina
\cite[Theorem~1.5]{DiMarinoSquassina19} showed that, on complete
doubling Poincar\'e spaces, the limit superior of the
Nguyen-type functionals characterizes Sobolev spaces with
$p\in(1,\infty)$.
Di Marino and Squassina \cite[Open Problem~1.8]{DiMarinoSquassina19}
asked whether, for $f\in L^1(\mu)$, the finiteness of the limit inferior
of the Nguyen-type energies along a sequence $\lambda_j\to0^+$
implies $f\in\mathrm{BV}(\mu)$ with total variation bounded by that limit
inferior. Their formulation allows a doubling kernel comparable to
$V(x,y)$.
On complete connected Ahlfors $Q$-regular spaces with $Q>1$,
Lahti et al.\ \cite[Theorem~1.4]{LahtiPinamontiZhou2024} bounded
Sobolev energy and total variation by the limit superior of more
general nonconvex functionals on bounded open sets.

In this article, we complete the aforementioned picture by establishing
Poincar\'e inequalities with sharp parameter ranges and the
corresponding liminf lower bounds on complete
doubling Poincar\'e spaces.
As a special case, we answer the problem posed by Di Marino and
Squassina \cite[Open Problem~1.8]{DiMarinoSquassina19} affirmatively.
In Euclidean spaces, our results recover the known sharp results.
The new contribution is the extension to general doubling Poincar\'e
spaces and the sharp dependence of the negative parameter ranges on
the Poincar\'e and volume-growth exponents.

\subsection{Main Results}

Let $(X,\rho,\mu)$ be a complete metric measure space, where
$\mu$ is a Borel measure that is finite and positive on balls and satisfies
the doubling condition; that is,
there exists $C_D\in[1,\infty)$ such that
\begin{equation}\label{eq:doubling}
0<\mu(B(x,2r))\leq C_D\mu(B(x,r))<\infty
\end{equation}
for any $x\in X$ and $r\in(0,\infty)$,
where the \emph{ball $B(x,r):=\{y\in X: \rho(x,y)<r\}$}. For any $x,y\in X$ and
$r\in(0,\infty)$, write
$V(x,r):=\mu(B(x,r))\ \text{and}\ V(x,y):=V(x,\rho(x,y)).$

The doubling condition implies that there exist
$D\in(0,\infty)$ and $C_{\mathrm u}\in[1,\infty)$ such that
\begin{equation}\label{eq:volume-upper-growth}
 V(x,R)\leq C_{\mathrm u}
 \left(\frac{R}{r}\right)^D V(x,r)
\end{equation}
for any $x\in X$ and $0<r\leq R<\infty$.
We fix one such exponent $D$ and call it an \emph{upper dimension}
of $\mu$.
Throughout this article,  we assume  that $X$ contains  at least two distinct points
and $\mu(\{x\}) = 0$ for any $x\in X$.

A nonnegative Borel function $g$ is an \emph{upper gradient} of a
function $u:X\to\mathbb R$ if, for any nonconstant rectifiable curve
$\eta:[0,\ell_\eta]\to X$ parametrized by arclength,
\[
 |u(\eta(0))-u(\eta(\ell_\eta))|
 \leq\int_0^{\ell_\eta}g(\eta(t))\,d t.
\]
For $q\in[1,\infty)$, the space $(X,\rho,\mu)$ is said to support a
\emph{weak $(1,q)$-Poincar\'e inequality} if there exist constants
$C_P\in(0,\infty)$ and $\Lambda_P\in[1,\infty)$ such that,
for any ball $B$, any $u\in L^1_{\mathrm{loc}}(\mu)$,
and any upper gradient $g$ of $u$,
\begin{equation}\label{PS}
 \fint_B|u-u_B|\,d\mu
 \leq C_Pr_B
 \left(\fint_{\Lambda_PB}g^q\,d\mu\right)^{\frac1q}.
\end{equation}
Here, and thereafter, $r_B$ denotes the radius of the ball $B$ and
$$u_B:=\fint_Bu\,d\mu:=\frac{1}{\mu(B)}\int_B u\,d\mu.$$
A complete doubling metric measure
space satisfying inequality \eqref{PS} is called a
\emph{$q$-Poincar\'e space}; cf. \cite{bb11,HKST}.
For any nonempty set $E\subset X$, let
$\operatorname{diam} E:=\sup\{\rho(x,y):x,y\in E\}$ denote its diameter.
We say that $\mu$ is \emph{reverse doubling} if there exist constants
$a,C_0\in(1,\infty)$ such that
\begin{equation*}
 V(x,ar)\geq C_0V(x,r)
\end{equation*}
for any $x\in X$ and
$r\in(0,\frac{\operatorname{diam} X}{a})$.  The reverse doubling condition implies
that there exist $d\in(0,\infty)$ and $c_d\in(0,1]$ such that
\begin{equation}\label{eq:volume-reverse-growth}
 V(x,\lambda r)\geq c_d\lambda^dV(x,r)
\end{equation}
for any $x\in X$, $r\in(0,\frac{\operatorname{diam} X}{3})$, and
$\lambda\in[1,\frac{\operatorname{diam} X}{3r})$. We fix one exponent
$d\in(0,\infty)$ for which \eqref{eq:volume-reverse-growth} holds,
and call $d$ a \emph{lower dimension} of $\mu$.
Comparing \eqref{eq:volume-upper-growth} and \eqref{eq:volume-reverse-growth} as $r\to0^+$ gives $d\leq D$.

It is known that any Poincar\'e space is
connected; see \cite[Proposition~4.2]{bb11}.  Consequently, any
doubling Poincar\'e space is reverse doubling by
\cite[Corollary~3.8]{bb11}.
For convenience, we refer to the doubling constant $C_D$ in
\eqref{eq:doubling},
the exponent $q$ and the constants $\Lambda_P$ and $C_P$ in \eqref{PS},
the lower dimension $d$ and the constant $c_d$ in
\eqref{eq:volume-reverse-growth}, and the upper dimension $D$
and the constant $C_{\mathrm u}$ in \eqref{eq:volume-upper-growth}
as the \emph{geometric constants} of the space $X$.

Let $p\in[1,\infty)$ and $f:X\to \mathbb R$ be measurable.
A nonnegative measurable function $g$ is called a
\emph{Haj\l asz upper gradient} of $f$ if,
for almost every $x,y\in X$,
\begin{align}\label{eq:H-gradient}
|f(x)-f(y)|\leq\rho(x,y)[g(x)+g(y)].
\end{align}
The \emph{homogeneous Haj\l asz--Sobolev space $\dot M^{1,p}(\mu)$}
consists of all measurable
functions $f$ on $X$ admitting a Haj\l asz upper gradient in $L^p(\mu)$, and its
seminorm is given by
\begin{equation*}
 \|f\|_{\dot M^{1,p}(\mu)}
 :=\inf_g\|g\|_{L^p(\mu)},
\end{equation*}
where the infimum is taken over all such Haj\l asz gradients; see, for
example, \cite{HKST}.

For a locally Lipschitz function $u$, write
\[
 \operatorname{lip} u(x):=\liminf_{r\to0^+}\frac1r
 \sup_{y\in B(x,r)}|u(x)-u(y)|, \quad x\in X.
\]
For $f\in L^1_{\mathrm{loc}}(\mu)$,
its \emph{total variation} is defined by
\begin{equation}\label{eq:variation-definition}
 |Df|(X):=\inf_{\{f_j\}_{j\in\mathbb N}}
 \liminf_{j\to\infty}\int_X\operatorname{lip} f_j\,d\mu,
\end{equation}
where the infimum is taken over all sequences
$\{f_j\}_{j\in\mathbb N}\subset\operatorname{Lip}_{\mathrm{loc}}(X)$ converging to
$f$ in $L^1_{\mathrm{loc}}(\mu)$ as $j\to\infty$.  We define
$\dot{\mathrm{BV}}(\mu):=\{f\in L^1_{\mathrm{loc}}(\mu):|Df|(X)<\infty\}$; see,
for example, \cite{Miranda03}.
For $p\in[1,\infty)$ and $f\in L^1_{\mathrm{loc}}(\mu)$, define
\[
 \mathcal E_p(f;X):=
 \begin{cases}
 \|f\|_{\dot M^{1,p}(\mu)}^p
 &\text{if }p\in(1,\infty)\ \text{and}\ f\in\dot M^{1,p}(\mu),\\
 |Df|(X)&\text{if }p=1\ \text{and}\ f\in\dot{\mathrm{BV}}(\mu),\\
 \infty&\text{otherwise}.
 \end{cases}
\]

In this article, we study the following metric and volume
functionals.
More precisely, for $p\in[1,\infty)$, $\gamma\in\mathbb R$,
$\phi\in\{\rho,V\}$, $\lambda\in(0,\infty)$,
a measurable set $E\subset X$, and a measurable function
$f:X\to\mathbb R$, define
\begin{equation}\label{eq:energy}
 \mathcal E_{\lambda,\gamma,p}^\phi(f;E)
 :=\lambda^p
 \iint_{\genfrac{}{}{0pt}{}{x,y\in E}{
 |f(x)-f(y)|>\lambda\rho(x,y)\phi(x,y)^{\gamma}}}
 \frac{\phi(x,y)^{\gamma p}}{V(x,y)}
 \,d\mu(x)\,d\mu(y),
\end{equation}
where the integrand is understood to vanish on the diagonal and
the integral may be infinite.
For a ball $B$, we also write
\[
 \phi(B):=
 \begin{cases}
  r_B&\text{if }\phi=\rho,\\
  \mu(B)&\text{if }\phi=V.
 \end{cases}
\]
The two formulations are equivalent up to constants and a change of
parameters on Ahlfors regular spaces, but need not be so in general.
Thus, we retain these two formulations in this article.

Our first main result gives new Poincar\'e inequalities for nonlocal nonconvex functionals
in \eqref{eq:energy}.

\begin{theorem}\label{thm:NPI}
Let $p\in[1,\infty)$, $(X,\rho,\mu)$ be a complete doubling $q$-Poincar\'e space for some
$q\in[1,p]$ with a lower dimension $d$ and an upper dimension $D$,
and let $\phi\in\{\rho,V\}$. If either $\phi=\rho$ and
$\gamma\in(-\infty,-\frac{\min\{q,D\}}{p}]$ or $\phi=V$ and
$\gamma\in(-\infty,-\frac{q}{p\max\{q,d\}}]$,
then there exist constants $C\in(0,\infty)$ and
$\Lambda\in[1,\infty)$, depending only on $p,\gamma$, and the
geometric constants of $X$, such that
\begin{align}
 \fint_B\fint_B|f(x)-f(y)|^p
 \,d\mu(x)\,d\mu(y)\leq C\lambda^p r^p\phi(B)^{\gamma p}
 +C\frac{r^p}{\mu(B)}
 \mathcal E^\phi_{\lambda,\gamma,p}(f;\Lambda B)
 \label{eq:NPI}
\end{align}
for any ball $B$ with its radius
$r\in(0,\operatorname{diam} X)$, any $\lambda\in(0,\infty)$, and any
measurable function $f:X\to\mathbb R$.
\end{theorem}

\begin{remark}
\begin{enumerate}[(i)]
\item
On $(\mathbb R^N,|\cdot-\cdot|,dx)$, we may take $q=1$
and $d=D=N$.
Thus, Theorem~\ref{thm:NPI} gives the ranges $\gamma\leq-1/p$
for $\phi=\rho$ and $\gamma\leq-1/(Np)$ for $\phi=V$, with energy
on $\Lambda B$. These agree with the ranges in
\cite[Corollary~1.3(ii)]{DYYZTruncated26}; the direct application
here retains the dilation $\Lambda$.
Writing $\omega_N:=|B(0,1)|$, the exact normalizations on $\mathbb R^N$ are
$\mathcal E^\rho_{\lambda,\gamma,p}=\omega_N^{-1}\Phi^\gamma_{\lambda,p}$ and
$\mathcal E^V_{\lambda,\gamma,p}
 =\omega_N^{-1}\Phi^{N\gamma}_{\lambda\omega_N^\gamma,p}.$
For the parameter used in \cite{DYYZTruncated26,Picenni2026}, the
corresponding substitutions are $p\gamma$ and $Np\gamma$,
respectively. We do not assert that $\Lambda=1$ in general.
In particular, when $\phi=\rho$, the case $\gamma=-1$ corresponds to
Nguyen's Poincar\'e inequality
\cite[Theorem~1(a)]{Nguyen2011}.

\item
The negative ranges of $\gamma$ in Theorem~\ref{thm:NPI} are sharp
for any fixed $p\in[1,\infty)$, $q\in[1,p]$,
$D\in[1,\infty)$, and $d\in(0,D]$.
For any negative $\gamma$ outside
the corresponding range, \eqref{eq:NPI} fails on some compact
$q$-Poincar\'e space with lower dimension $d$ and upper dimension
$D$; see Corollary~\ref{cor:NPI-sharpness}.

\end{enumerate}
\end{remark}

As an application, we obtain the following
liminf lower bounds.
For $\gamma\in\mathbb R\setminus\{0\}$, we write
$\lambda\to L_\gamma$ for $\lambda\to0^+$ when
$\gamma\in(-\infty,0)$ and for $\lambda\to\infty$ when
$\gamma\in(0,\infty)$.

\begin{theorem}\label{thm:asymptotic}
Let $p\in[1,\infty)$, $(X,\rho,\mu)$
be a complete doubling $q$-Poincar\'e space for some $q\in[1,p]$
with a lower dimension $d$ and an upper dimension $D$,
and let $\phi\in\{\rho,V\}$. Assume either $\phi=\rho$ and
$\gamma\in(-\infty,-\frac{\min\{q,D\}}{p}]\cup(0,\infty)$ or $\phi=V$ and
$\gamma\in(-\infty,-\frac{q}{p\max\{q,d\}}]\cup(0,\infty)$.
Let $f\in L^1_{\mathrm{loc}}(\mu)$ satisfy
\[
 \liminf_{\lambda\to L_\gamma}
 \mathcal E^\phi_{\lambda,\gamma,p}(f;X)<\infty.
\]
Then $f\in \dot{M}^{1,p}(\mu)$ when $p\in(1,\infty)$
and $f\in \dot{\rm BV}(\mu)$ when $p=1$. In either case,
\begin{equation}\label{eq:asymptotic-lower}
 \mathcal E_p(f;X)
 \le C
 \liminf_{\lambda\to L_\gamma}
 \mathcal E^\phi_{\lambda,\gamma,p}(f;X),
\end{equation}
where $C$ is a positive constant depending only on $p$, $\gamma$,
and the geometric constants of $X$.
\end{theorem}
\begin{remark}
\begin{enumerate}[(i)]
\item
On $(\mathbb R^N,|\cdot-\cdot|,dx)$, we may take $q=1$
and $d=D=N$.
In this case, Theorem~\ref{thm:asymptotic} reduces to the lower bound in
\cite[Corollary~4.1]{DYYZTruncated26}
(see also \cite{Picenni2026,LinYangYangYuanZhang,LinEtAlFullRange2026}),
with $\gamma$ therein replaced
by $p\gamma$ when $\phi=\rho$ and by $Np\gamma$ when $\phi=V$.

\item
The negative ranges of $\gamma$ in Theorem~\ref{thm:asymptotic} are sharp
for any fixed $p\in[1,\infty)$, $q\in[1,p]$, $D\in[1,\infty)$,
and $d\in(0,D]$.
For any negative $\gamma$ outside
the corresponding range, \eqref{eq:asymptotic-lower} fails on some compact
$q$-Poincar\'e space with lower dimension $d$ and upper dimension
$D$; see Proposition~\ref{prop:counterexample}.
In Euclidean spaces, the sharp negative ranges are recorded in
\cite[p.~6]{Picenni2026}.

\item
Theorem~\ref{thm:asymptotic} is a one-sided regularity criterion;
in particular, it does not assert a converse for BV functions.
Indeed, on $X=[-1,1]$ with Lebesgue measure, take
$p=1$, $\phi=\rho$, $\gamma=-1$, and $f=\mathbf 1_{(0,1]}$.
Then $|Df|(X)=1$, whereas, for $\lambda\in(0,1)$,
\[
 \mathcal E^\rho_{\lambda,-1,1}(f;X)
 \geq\frac\lambda2\int_0^{\frac14}\int_0^{\frac14}
 \frac{du\,dv}{(u+v)^2}=\infty.
\]
For $p=1$, $\phi=\rho$, and $\gamma=-1$,
\eqref{eq:asymptotic-lower} gives precisely the BV lower bound
stated in \cite[Open Problem~1.8]{DiMarinoSquassina19}.

\end{enumerate}
\end{remark}

The Euclidean proofs of Theorems~\ref{thm:NPI} and
\ref{thm:asymptotic} for $\phi=\rho$ and $\gamma=-1$
(see \cite{BourgainNguyen2006,Nguyen2011}) and of
Theorem~\ref{thm:asymptotic} for $\gamma\in(0,\infty)$
(see \cite{gp25}) use ingenious one-dimensional interval selection
and then pass to higher dimensions
by slicing or rotational averaging. However,
these arguments strongly rely on one-dimensional geometry
and the symmetries of Euclidean spaces and hence they cannot be extended
directly to general metric measure spaces.
Moreover, the interaction between the Poincar\'e exponent and
volume growth leads to sharp ranges of $\gamma$ different from
those in Euclidean spaces.
To overcome these difficulties, we use two proof strategies adapted
to the different endpoints.
The ranges furnished by the two methods are
\[
\begin{array}{c|cc}
 &\phi=\rho&\phi=V\\ \hline
 \text{curve method}&\gamma\leq-q/p&\gamma\leq-q/(dp)\\
 \text{adaptive dyadic method}&\gamma\leq-D/p&\gamma\leq-1/p
\end{array}
\]
Their unions give the ranges in Theorem~\ref{thm:NPI}. In particular,
on an Ahlfors $N$-regular space these become
$\gamma\leq-\min\{q,N\}/p$ and
$\gamma\leq-\min\{q,N\}/(Np)$, consistently with the change
of parameter between the two functionals.

For $\gamma\le-\frac qp$ when $\phi=\rho$ and
$\gamma\le-\frac q{dp}$ when $\phi=V$, we use the quantitative curve
characterization of Poincar\'e spaces to replace integration over
straight lines. Pointwise bounds for the nonlocal energy densities
allow us to use one-dimensional truncated difference estimates
along the resulting joining curves to obtain the desired inequalities.

For $\gamma\le-\frac Dp$ when $\phi=\rho$ and
$\gamma\le-\frac1p$ when $\phi=V$, we first construct dyadic
refinements which keep neighboring cubes at comparable scales and
bound the weighted dyadic sum over the final partition by the sums over
the initial and marked cubes
(Lemma~\ref{lem:volume-weighted-linear-closure}).
We then use median differences to select marked cubes that yield
sufficiently many pairs with separated function values, so that the
nonlocal energy bounds the weighted dyadic sum over the marked cubes.
This argument using low and high level sets to bound the nonlocal
energy from below borrows some ideas from the proofs of
\cite[Lemma~2]{BourgainNguyen2006} and
\cite[Lemma~5]{Nguyen2011}.
We then combine the discrete Poincar\'e estimate for the cube medians
with the nonlocal energy bound to obtain the desired inequality.
The adaptive construction may be of independent interest.

The remainder of this article is organized as follows.
In Section~\ref{sec:one-dimensional}, we recall a curve estimate,
prove one-dimensional truncated difference estimates, and establish
pointwise bounds for the nonlocal energy densities.
In Section~\ref{sec:metric-proofs}, we use the curve method to prove
Poincar\'e inequalities for negative parameters and obtain asymptotic lower bounds by discrete convolutions.
Section~\ref{sec:volume-proofs} develops the adaptive dyadic argument
and completes the proofs of Theorems~\ref{thm:NPI} and
\ref{thm:asymptotic}.
The sharpness of the negative parameter ranges is shown in
Section~\ref{sec:optimality}.

We end this introduction by making some notational conventions.
Let $\mathbb{N}:=\{1,2,\ldots\}$.
We also write $\mathbb Z_+:=\{0,1,2,\ldots\}$.
Given a set $E\subset X$, we denote its characteristic
function by $\mathbf{1}_E$.
The cardinality of a set $E$ is denoted by $\sharp E$.
For a ball $B=B(x,r)$ and $\tau\in(0,\infty)$, let
$\tau B:=B(x,\tau r)$.
The notation $f\lesssim g$ means that $f\leq Cg$ for some
constant $C\in(0,\infty)$. If both $f\lesssim g$ and $g\lesssim f$
hold, we write $f\sim g$.
We write $f\gtrsim g$ if $g\lesssim f$.
% Referee revision: use the conventional notation for one-sided limits.

\section{Curve and Truncated Difference Estimates}\label{sec:one-dimensional}

In Subsection~\ref{sec:poincare-tools}, we recall a curve estimate
that controls line integrals by spatial integrals.
We then prove complementary truncated difference estimates on intervals
in Subsection~\ref{subsec:one-dimensional-endpoints} and establish
the required pointwise bounds for the nonlocal energy densities
in Subsection~\ref{subsec:pointwise-estimates}.

\subsection{Curve Estimates on Poincar\'e Spaces}
\label{sec:poincare-tools}

In this subsection, we recall a curve estimate on Poincar\'e spaces.

Let $(X,\rho,\mu)$ be a doubling metric measure space.
For $x,y\in X$, let
\begin{equation*}
 K(x,y):=
 \begin{cases}
  \displaystyle\frac{\rho(x,y)}{V(x,y)}&\text{if }x\neq y,\\
  0&\text{if }x=y.
 \end{cases}
\end{equation*}
The doubling condition gives $K(x,y)\sim K(y,x)$. Integration over
annuli centered at $y$ then yields, for any $\tau\in[1,\infty)$
and any ball $B$,
\begin{equation}\label{eq:kernel-average}
 \sup_{y\in\tau B}\int_B K(x,y)\,d\mu(x)\lesssim r_B,
\end{equation}
where the implicit positive constant depends only on $\tau$ and
the doubling constant.

The following curve estimate follows from the quantitative equivalences in
\cite[Proposition~2.4 and Theorem~A.3]{CaputoCavallucci25}.

\begin{lemma}\label{lem:riesz-curve}
Let $q\in[1,\infty)$, and let $(X,\rho,\mu)$ be a complete doubling
$q$-Poincar\'e space. Then there exists $L\in[1,\infty)$ such that,
for any distinct $x,y\in X$ and any nonnegative Borel function
$F:X\to[0,\infty]$, there exists an arclength-parametrized curve
$\eta:[0,\ell_\eta]\to X$ joining $x$ to $y$ such that $\ell_\eta\leq L\rho(x,y)$
and
\begin{equation}\label{eq:riesz-curve}
 \left(\int_\eta F\,d s\right)^q
 \lesssim\rho(x,y)^{q-1}
 \int_{B(x,L\rho(x,y))}F(z)^q
 \left[K(x,z)+K(y,z)\right]\,d\mu(z),
\end{equation}
where $L$ and the implicit positive constant depend only on $q$
and the geometric constants of $X$.
\end{lemma}

Indeed, Theorem~A.3(ii) of \cite{CaputoCavallucci25} asserts
the existence of such a curve for every nonnegative Borel integrand.
Since $B(x,2L\rho(x,y))\cup B(y,2L\rho(x,y))
\subset B(x,(2L+1)\rho(x,y))$, it follows that
\eqref{eq:riesz-curve} holds after increasing $L$.

We will average the right-hand side of \eqref{eq:riesz-curve} over
$x$ and $y$. Note that, by \eqref{eq:kernel-average} and Tonelli's theorem,
we obtain, for any nonnegative Borel function $F$ and $\tau\in[1,\infty)$,
\begin{equation}\label{eq:riesz-average}
 \fint_B\fint_B\int_{\tau B}F(z)
 \left[K(x,z)+K(y,z)\right]\,d\mu(z)\,d\mu(x)\,d\mu(y)
 \lesssim\frac{r_B}{\mu(B)}\int_{\tau B}F\,d\mu,
\end{equation}
where the implicit positive constant depends only on $\tau$ and
the doubling constant.

\subsection{One-Dimensional Truncated Difference Estimates}
\label{subsec:one-dimensional-endpoints}

In this subsection, we prove complementary one-dimensional truncated
difference estimates for the metric and volume functionals.
We will use these estimates along curves in
Section~\ref{sec:metric-proofs}.

\begin{proposition}\label{prop:volume-one-dimensional}
Let $(X,\rho,\mu)$ be a doubling and reverse doubling metric measure
space.
Let $C_\ast,\sigma\in[1,\infty)$, $A\in(0,C_\ast\operatorname{diam} X)$, and
$\gamma\in\mathbb R\setminus\{0\}$.
Let $\eta:[0,A]\to X$ be
$1$-Lipschitz, and let $v:[0,A]\to\mathbb R$ and
$h:[0,A]\to[0,\infty]$ be measurable. Assume that, for any
distinct $u,w\in[0,A]$ and any $R\in[|u-w|,4A]$,
\begin{align}
 h(u)+h(w)\leq V(\eta(u),R)^\gamma\
 \Longrightarrow\
 |v(u)-v(w)|\leq R V(\eta(u),R)^\gamma.
 \label{eq:volume-one-dimensional-coarse}
\end{align}
Then there exists a constant $c\in(0,\infty)$ such that the following
statements hold.
\begin{enumerate}[{\rm (i)}]
\item If $\gamma\in(-\infty,-\frac{1}{d\sigma}]$, where $d$
is a lower dimension of $X$, then
\begin{align}\label{eq:v-p-conclusion}
\left(|v(0)-v(A)|-cAV(\eta(0),A)^\gamma\right)_+
\lesssim A
 \left[h(0)^\sigma+h(A)^\sigma+
 \fint_0^Ah(u)^\sigma\,d u\right]^{\frac{1}{\sigma}}.
\end{align}
\item If $\gamma\in(0,\infty)$, then
\begin{align}\label{eq:volume-positive-conclusion}
\min\left\{|v(0)-v(A)|,cAV(\eta(0),A)^\gamma\right\}
\lesssim A
 \left[h(0)^\sigma+h(A)^\sigma+
 \fint_0^Ah(u)^\sigma\,d u\right]^{\frac{1}{\sigma}}.
\end{align}
\end{enumerate}
The constant $c$ and the implicit positive constants depend only on
$\gamma$, $\sigma$, $C_\ast$, and the geometric constants of $X$.
\end{proposition}

The left-hand sides of \eqref{eq:v-p-conclusion} and
\eqref{eq:volume-positive-conclusion} correspond to the two complementary
truncations in the identity
\[
 (s-a)_++\min\{s,a\}=s,\quad s,a\in[0,\infty),
\]
where $(s-a)_+:=\max\{s-a,0\}$, with $s=|v(0)-v(A)|$ and
$a=cAV(\eta(0),A)^\gamma$.

The proof of Proposition~\ref{prop:volume-one-dimensional} uses the
following two lemmas.
We first record a dyadic endpoint estimate. The proof of part~(ii)
of the following lemma borrows some ideas from the proof of
\cite[Proposition~2.2]{gp25}. For any interval
$Q:=[a_Q,b_Q]\subset\mathbb R$ with $a_Q<b_Q$ and for any $k\in\mathbb Z_+$,
let $\mathcal D_k(Q)$ denote the family of its $2^k$ dyadic subintervals
of length $2^{-k}|Q|$, and let
$\mathcal D(Q):=\bigcup_{k\geq0}\mathcal D_k(Q)$.

\begin{lemma}\label{lem:dyadic-endpoint}
Let $Q:=[a_Q,b_Q]\subset\mathbb R$ with $a_Q<b_Q$.
Let $\sigma\in[1,\infty)$, $v:Q\to\mathbb R$, and
$\{\theta_I\}_{I\in\mathcal D(Q)}$ be a family of positive numbers.
Assume that there exist $c,C\in(0,\infty)$ such that
$c\theta_I\leq\theta_{I'}\leq C\theta_I$
for any $I\in\mathcal D(Q)$ and any dyadic child $I'$ of $I$. Let
\begin{equation}\label{eq:dyadic-bad-family}
 K_I:=|I|\theta_I^\sigma\ \text{and}\
 \mathcal B:=\{I=[a,b]\in\mathcal D(Q):|v(b)-v(a)|>|I|\theta_I\}.
\end{equation}
Then the following assertions hold.
\begin{enumerate}[{\rm (i)}]
\item If there exists $c_\ast\in(0,1]$ such that
$K_I\geq c_\ast K_Q$ for any $I\in\mathcal D(Q)$, then
\[
 |v(b_Q)-v(a_Q)|^\sigma
 \lesssim |Q|^{\sigma-1}
 \left[K_Q+\sum_{I\in\mathcal B}K_I\right].
\]
\item If $Q\notin\mathcal B$ and
$\max_{I\in\mathcal D_k(Q)}\theta_I\to0$ as $k\to\infty$, then
\[
 |v(b_Q)-v(a_Q)|^\sigma
 \lesssim |Q|^{\sigma-1}\sum_{I\in\mathcal B}K_I.
\]
\end{enumerate}
The implicit positive constants in (i) and (ii) depend only on $\sigma$, $c$, and $C$.
\end{lemma}

\begin{proof}
We first show (i). We may assume that
$\sum_{I\in\mathcal B}K_I<\infty$. Starting from $Q$, recursively subdivide any interval in
$\mathcal B$ into its two dyadic children, while retaining intervals
outside $\mathcal B$ without further subdivision.
Since $K_I\geq c_\ast K_Q>0$ for any $I\in\mathcal B$, it follows that
\[
 \sharp \mathcal B\leq\frac{1}{c_\ast K_Q}
 \sum_{I\in\mathcal B}K_I<\infty.
\]
Thus, the procedure terminates after at most
$\sharp \mathcal B$ subdivisions, yielding a finite partition
$\mathcal C\subset\mathcal D(Q)\setminus\mathcal B$ of $Q$.
The triangle inequality and H\"older's inequality give
\[
 |v(b_Q)-v(a_Q)|^\sigma
 \leq |Q|^{\sigma-1}\sum_{I\in\mathcal C}K_I.
\]
Any member of $\mathcal C$ other than $Q$ has a parent in
$\mathcal B$. Since $K_{I'}\leq\frac{C^\sigma}{2}K_I$ for a child
$I'$ of $I$, it follows that
$\sum_{I\in\mathcal C}K_I\lesssim K_Q+\sum_{I\in\mathcal B}K_I$,
which proves (i).

We next prove (ii). Let $\mathcal P$ consist of the maximal dyadic
intervals outside $\mathcal B$ that have a child in $\mathcal B$.
The intervals in $\mathcal P$ have pairwise disjoint interiors by
maximality. From $Q\notin\mathcal B$, we deduce that the first interval in
$\mathcal B$ along the dyadic chain from $Q$ to any $I\in\mathcal B$
has a parent outside $\mathcal B$. Thus, $I$ is contained in an
interval of $\mathcal P$. For any $m\in\mathbb Z_+$, the members of
$\mathcal P$ of generation at most $m$, together with the intervals
in $\mathcal D_m(Q)$ not contained in any such member, form a finite
partition of $Q$. Since any interval in $\mathcal B$ of generation
$m$ is contained in a member of $\mathcal P$ of strictly smaller
generation, it follows that none of the intervals in this partition
belongs to $\mathcal B$. Using the triangle inequality, we find that
\[
 |v(b_Q)-v(a_Q)|
 \leq \sum_{P\in\mathcal P}|P|\theta_P
       +|Q|\max_{I\in\mathcal D_m(Q)}\theta_I.
\]
Letting $m\to\infty$ and applying H\"older's inequality, we obtain
\[
 |v(b_Q)-v(a_Q)|^\sigma
 \leq |Q|^{\sigma-1}\sum_{P\in\mathcal P}K_P.
\]
Choose a child $P'\in\mathcal B$ of any $P\in\mathcal P$.
Then the chosen children are distinct and
$K_P\leq2c^{-\sigma}K_{P'}$, which shows (ii).
This completes the proof of Lemma~\ref{lem:dyadic-endpoint}.
\end{proof}

The following radial estimate follows from
\cite[Lemma~3.2]{dlyyz26b} by a standard annular argument using
the doubling condition and the reverse doubling condition.

\begin{lemma}\label{lem:volume-radial}
Let $(X,\rho,\mu)$ be a doubling and reverse doubling metric measure
space. Let $C_\ast\in(0,\infty)$ and
$\beta\in\mathbb R\setminus\{0\}$. Then, for any $x\in X$,
$t\in(0,\infty)$, and $R_0\in(0,C_\ast\operatorname{diam} X)$,
\begin{equation*}
 \int_0^{R_0}
 \mathbf 1_{\{t>V(x,r)^\beta\}}
 V(x,r)^\beta\frac{\,d r}{r}
 \lesssim t,
\end{equation*}
where the implicit positive constant depends only on $\beta$,
$C_\ast$, and the geometric constants of $X$.
\end{lemma}

\begin{proof}[Proof of Proposition~\ref{prop:volume-one-dimensional}]
Since $A\in(0,C_\ast\operatorname{diam}X)$,
it follows from \eqref{eq:volume-reverse-growth} and
the proof of $(\mathrm{RD1})\Rightarrow(\mathrm{RD2})$
in \cite[Lemma~2.2]{yz11} that
\begin{align}\label{eq:rd-d-ex}
V(x,R)\gtrsim\left(\frac{R}{r}\right)^dV(x,r)
\end{align}
for any $x\in X$ and $0<r\le R\le 8A$,
where the implicit positive constant depends only on $C_\ast$
and the geometric constants of $X$.

Define the projection $\pi:\mathbb R\to[0,A]$ by
$\pi(t):=\min\{A,\max\{0,t\}\},$
and let
\begin{align}\label{eq:proj}
\widetilde\eta:=\eta\circ\pi,
 \quad \widetilde v:=v\circ\pi,
 \quad \widetilde h:=h\circ\pi.
\end{align}
Since $\pi$ is $1$-Lipschitz, the assumption
\eqref{eq:volume-one-dimensional-coarse} remains valid for
$\widetilde{v}$ and $\widetilde{h}$; that is,
\begin{align}
 |\widetilde v(u)-\widetilde v(w)|>
 R V(\widetilde\eta(u),R)^\gamma
\ \Longrightarrow\
 \widetilde h(u)+\widetilde h(w)>
 V(\widetilde\eta(u),R)^\gamma
 \label{eq:volume-extended-contrapositive}
\end{align}
for any distinct $u,w\in\mathbb R$ and any
$R\in[|u-w|,4A]$.

Suppose that $\gamma\in(-\infty,-\frac{1}{d\sigma}]$.  Fix
$\tau\in[2A,4A)$ and $a\in[-A,0)$, and let $Q=[a,a+\tau]$.
We use $\mathcal D_k(Q)$ and $\mathcal D(Q)$ as defined before
Lemma~\ref{lem:dyadic-endpoint}.
For $I=[a_I,b_I]\in\mathcal D(Q)$, write
\begin{equation}\label{eq:volume-dyadic-weights}
 r_I:=b_I-a_I,\quad W_I:=V(\widetilde\eta(a_I),r_I),
 \quad K_I:=r_IW_I^{\gamma\sigma}.
\end{equation}
Let $\mathcal B$ be the family in \eqref{eq:dyadic-bad-family}
with $v:=\widetilde v|_Q$ and $\theta_I:=W_I^\gamma$.
If $I\in\mathcal D_j(Q)$, $L\in\mathcal D_{j+k}(Q)$, and
$L\subset I$, then $r_L=2^{-k}r_I$ and
$\rho(\widetilde\eta(a_L),\widetilde\eta(a_I))\leq r_I$.
This implies $B(\widetilde\eta(a_L),2r_I)\subset B(\widetilde\eta(a_I),3r_I)$,
which, together with \eqref{eq:rd-d-ex} and the doubling condition, gives
\[
 2^{kd}W_L\lesssim
 V(\widetilde\eta(a_L),2r_I)
 \lesssim W_I.
\]
It follows that
\begin{equation}\label{eq:volume-K-monotonicity}
 \frac{K_L}{K_I}\gtrsim
 2^{-k(1+d\gamma\sigma)}\gtrsim1,
\end{equation}
where the implicit positive constants depend only on $\gamma$, $\sigma$,
$C_\ast$, and the geometric constants of $X$.
For a dyadic child $I'$ of $I$, the $1$-Lipschitz property of
$\widetilde\eta$ and the doubling condition give $W_{I'}\sim W_I$.
Applying Lemma~\ref{lem:dyadic-endpoint}(i) and using
\eqref{eq:volume-K-monotonicity}, we obtain
\begin{equation}\label{eq:volume-negative-partition-bound}
 |v(0)-v(A)|^\sigma
 \lesssim\tau^{\sigma-1}
 \left[K_Q+
 \sum_{k\geq1}\sum_{I\in\mathcal D_k(Q)\cap\mathcal B}K_I\right].
\end{equation}
We integrate \eqref{eq:volume-negative-partition-bound} against
\begin{equation}\label{eq:volume-grid-averaging-measure}
 \frac{1}{A\log2}\frac{\,d\tau}{\tau}\,d a
 \quad\text{on }[2A,4A)\times[-A,0).
\end{equation}
For fixed $k\in\mathbb N$, let $r:=2^{-k}\tau$.  Define
\begin{equation}\label{eq:volume-difference-density}
 \Psi(u,r):=
 \mathbf 1_{\{|\widetilde v(u+r)-\widetilde v(u)|>
 rV(\widetilde\eta(u),r)^\gamma\}}
 V(\widetilde\eta(u),r)^{\gamma\sigma}.
\end{equation}
The intervals in $\mathcal D_k(Q)$ are
\[
 I_{k,i}=[a+ir,a+(i+1)r],
 \quad i\in\{0,\ldots,2^k-1\}.
\]
It follows from the definition of $K_I$ that
\begin{align*}
 \int_{-A}^0
 \sum_{I\in\mathcal D_k(Q)\cap\mathcal B}K_I\,d a
 &\leq r
 \sum_{i=0}^{2^k-1}\int_{-A}^0\Psi(a+ir,r)\,d a\\
 &= r
 \int_{-A}^{4A}\Psi(u,r)
 \sum_{i=0}^{2^k-1}\mathbf 1_{[-A+ir,ir]}(u)\,d u.
\end{align*}
The last identity follows from $u=a+ir$.  Moreover,
\begin{align*}
 \sum_{i=0}^{2^k-1}\mathbf 1_{[-A+ir,ir]}(u)
 &=\sharp \{i\in\{0,\ldots,2^k-1\}:ir\in[u,u+A]\}\leq1+\frac{A}{r}.
\end{align*}
Since $k\in\mathbb N$ and $\tau\in[2A,4A)$, we infer that
$r=2^{-k}\tau\in(0,2A)$.
Combining this and the preceding estimates gives
\[
 \frac{1}{A}\int_{-A}^0
 \sum_{I\in\mathcal D_k(Q)\cap\mathcal B}K_I\,d a
 \lesssim\int_{-A}^{4A}\Psi(u,r)\,d u.
\]
Therefore, summing the fixed-scale estimate in $k\in\mathbb N$ yields
\begin{align}
 &\frac{1}{A\log2}\int_{2A}^{4A}\frac{\,d\tau}{\tau}
 \int_{-A}^0
 \sum_{k\geq1}\sum_{I\in\mathcal D_k(Q)\cap\mathcal B}K_I\,d a
 \notag\\
 &\quad\lesssim
 \sum_{k\geq1}
 \int_{2^{1-k}A}^{2^{2-k}A}
 \int_{-A}^{4A}\Psi(u,r)\,d u\frac{\,d r}{r}\lesssim
 \int_{-A}^{4A}\int_0^{2A}\Psi(u,r)
 \frac{\,d r}{r}\,d u.
 \label{eq:volume-negative-grid-average}
\end{align}
Since $V(\widetilde\eta(u),r)\sim V(\widetilde\eta(u+r),r)$,
\eqref{eq:volume-extended-contrapositive} shows that the right-hand side of
\eqref{eq:volume-negative-grid-average} is bounded by a constant
multiple of
\begin{equation}\label{eq:volume-negative-radial-reduction}
 \int_{-A}^{6A}\int_0^{2A}
 \mathbf 1_{\{\widetilde h(u)>c_2
 V(\widetilde\eta(u),r)^\gamma\}}
 V(\widetilde\eta(u),r)^{\gamma\sigma}
 \frac{\,d r}{r}\,d u,
\end{equation}
where $c_2\in(0,\infty)$ depends only on $\gamma$ and the
doubling constant.
Applying Lemma~\ref{lem:volume-radial} with
$C_\ast$ replaced by $3C_\ast$, $x=\widetilde\eta(u)$,
$\beta=\gamma\sigma$, $t=(\frac{\widetilde h(u)}{c_2})^\sigma$,
and $R_0=2A$ when $\widetilde h(u)\in(0,\infty)$, we conclude that
\eqref{eq:volume-negative-radial-reduction} is bounded by a constant
multiple of
\[
 Ah(0)^\sigma+Ah(A)^\sigma+\int_0^Ah(u)^\sigma\,d u.
\]
Using $K_Q\lesssim A V(\eta(0),A)^{\gamma\sigma}$
and $\tau\sim A$ in the average of
\eqref{eq:volume-negative-partition-bound}, we obtain
\[
 |v(0)-v(A)|\le c A V(\eta(0),A)^\gamma
 +cA\left[h(0)^\sigma+h(A)^\sigma+
 \fint_0^Ah(u)^\sigma\,d u\right]^{\frac1\sigma},
\]
where $c$ is a positive constant depending only on
$\gamma$, $\sigma$, $C_\ast$, and the geometric constants of $X$.
This implies that \eqref{eq:v-p-conclusion} holds.

Suppose now that $\gamma\in(0,\infty)$. After replacing $v$ by
$-v$ if necessary, we may assume that $v(0)\leq v(A)$. The
$1$-Lipschitz truncation
\begin{equation}\label{eq:volume-positive-truncation}
 v(u)\mapsto v(0)+\min\left\{(v(u)-v(0))_+,v(A)-v(0),
 AV(\eta(0),A)^\gamma\right\}
\end{equation}
preserves \eqref{eq:volume-one-dimensional-coarse} and replaces the
endpoint difference by
$\min\{|v(A)-v(0)|,AV(\eta(0),A)^\gamma\}$.
We continue to write $v$ for the truncated function and use its
constant extension $\widetilde v$ from \eqref{eq:proj}. Then
\begin{equation}\label{eq:volume-positive-large}
 v([0,A])\subset[v(0),v(A)],\quad
 v(A)-v(0)\leq AV(\eta(0),A)^\gamma.
\end{equation}
We use the same intervals $Q$ and the notation in
\eqref{eq:volume-dyadic-weights}, \eqref{eq:dyadic-bad-family}, and
\eqref{eq:volume-difference-density}, now for the truncated function in
\eqref{eq:volume-positive-truncation}. Since $\widetilde v(a)=v(0)$,
$\widetilde v(a+\tau)=v(A)$, and
$W_Q=V(\eta(0),\tau)\geq V(\eta(0),A)$, it follows from
\eqref{eq:volume-positive-large} and $\tau\geq2A$ that
$Q\notin\mathcal B$. From the reverse doubling condition and the doubling condition,
we infer that
\begin{equation*}
 \max_{I\in\mathcal D_k(Q)}\theta_I
 \lesssim 2^{-kd\gamma}V(\eta(0),A)^\gamma\to0
 \quad\text{as }k\to\infty.
\end{equation*}
Applying Lemma~\ref{lem:dyadic-endpoint}(ii), we obtain
\begin{equation}\label{eq:volume-positive-grid-bound}
 [v(A)-v(0)]^\sigma
 \lesssim\tau^{\sigma-1}
 \sum_{k\geq1}\sum_{I\in\mathcal D_k(Q)\cap\mathcal B}K_I.
\end{equation}
The averaging argument leading to
\eqref{eq:volume-negative-grid-average} and the reduction to
\eqref{eq:volume-negative-radial-reduction} apply for either sign of
$\gamma$. Integrating
\eqref{eq:volume-positive-grid-bound} against the normalized measure
in \eqref{eq:volume-grid-averaging-measure}, and using $\tau\sim A$ and
Lemma~\ref{lem:volume-radial}, we conclude that
\begin{align*}
 [v(A)-v(0)]^\sigma
 &\lesssim A^{\sigma-1}
 \int_{-A}^{4A}\int_0^{2A}\Psi(u,r)\frac{\,d r}{r}\,d u\\
 &\lesssim A^{\sigma}
 \left[h(0)^\sigma+h(A)^\sigma+\fint_0^Ah(u)^\sigma\,d u\right].
\end{align*}
Taking the $\sigma$-th root and recalling the endpoint difference
after the truncation in \eqref{eq:volume-positive-truncation} gives
\eqref{eq:volume-positive-conclusion} with
$c=1$.
This completes the proof of Proposition~\ref{prop:volume-one-dimensional}.
\end{proof}

On the real line, Proposition~\ref{prop:volume-one-dimensional}
gives the following corollary.

\begin{corollary}
\label{prop:one-dimensional-endpoint}
Let $p\in[1,\infty)$, $\gamma\in\mathbb R\setminus\{0\}$,
and $A\in(0,\infty)$, and let $v:[0,A]\to\mathbb R$ and
$h:[0,A]\to[0,\infty]$ be measurable. Assume that, for any
$s,t\in[0,A]$ and $R\in[|s-t|,4A]\cap(0,\infty)$,
\begin{equation}\label{eq:coarse-1d}
 h(s)+h(t)\leq R^\gamma
 \quad\Longrightarrow\quad
 |v(s)-v(t)|\leq R^{1+\gamma}.
\end{equation}
Then there exists a constant $c\in(0,\infty)$ such that the following
statements hold.
\begin{enumerate}[{\rm (i)}]
\item If $\gamma\in(-\infty,-\frac1p]$, then
\begin{equation}\label{eq:endpoint-chain}
 \left(|v(0)-v(A)|-cA^{1+\gamma}\right)_+
 \lesssim A\left[h(0)^p+h(A)^p+
 \fint_0^Ah(s)^p\,d s\right]^{\frac1p}.
\end{equation}
\item If $\gamma\in(0,\infty)$, then
\begin{equation}\label{eq:endpoint-conclusion}
 \min\left\{|v(0)-v(A)|,cA^{1+\gamma}\right\}
 \lesssim A\left[h(0)^p+h(A)^p+
 \fint_0^Ah(s)^p\,d s\right]^{\frac1p}.
\end{equation}
\end{enumerate}
The constant $c$ and the implicit positive constants depend only on
$\gamma$ and $p$.
\end{corollary}

\begin{proof}
Since $V(x,R)=2R$ on $X:=\mathbb R$,
it follows from \eqref{eq:coarse-1d} that $2^\gamma v$ and
$2^\gamma h$ satisfy \eqref{eq:volume-one-dimensional-coarse}
with $\eta(t)=t$. Applying Proposition~\ref{prop:volume-one-dimensional}
with $C_\ast=d=1$ and $\sigma=p$, and multiplying the resulting
estimates by $2^{-\gamma}$, we obtain both conclusions with the
same constant $c$.
This completes the proof of Corollary~\ref{prop:one-dimensional-endpoint}.
\end{proof}

\begin{remark}
The hypotheses \eqref{eq:volume-one-dimensional-coarse} and
\eqref{eq:coarse-1d} are weaker than the Haj\l asz pointwise
inequality in the following sense. Recall that
this inequality for a function $v$ and a nonnegative function $h$
on $[0,A]$ is
\begin{equation}\label{eq:standard-hajlasz-1d}
 |v(s)-v(t)|\leq |s-t|[h(s)+h(t)].
\end{equation}
Here, we assume that \eqref{eq:standard-hajlasz-1d} holds
for any $s,t\in[0,A]$. If
$R\in[|s-t|,4A]\cap(0,\infty)$ and the left-hand side condition in
\eqref{eq:coarse-1d} holds, then
\begin{align*}
 |v(s)-v(t)|
 &\leq |s-t|[h(s)+h(t)]
 \leq R[h(s)+h(t)]
 \leq R^{1+\gamma}.
\end{align*}
Thus, \eqref{eq:standard-hajlasz-1d} implies
\eqref{eq:coarse-1d}. The same argument, with $R^\gamma$ replaced by
$V(\eta(s),R)^\gamma$, also gives \eqref{eq:volume-one-dimensional-coarse}
for any $1$-Lipschitz map $\eta:[0,A]\to X$.
\end{remark}

\subsection{Pointwise Estimates for the Nonlocal Energy Densities}
\label{subsec:pointwise-estimates}

In this subsection, we establish pointwise bounds for the nonlocal
energy densities. These bounds verify the hypotheses of the
one-dimensional truncated difference estimates along Lipschitz curves.

Let $(X,\rho,\mu)$ be a doubling metric measure space,
$p\in[1,\infty)$, $\gamma\in\mathbb R\setminus\{0\}$,
$\lambda\in(0,\infty)$, and $E\subset X$ be measurable.
For $\phi\in\{\rho,V\}$,
a measurable function $f:X\to\mathbb R$, and
$x\in X$, define
\begin{equation}\label{eq:G}
 G^\phi(x):=\lambda\left[
 \int_E
 \mathbf 1_{\{|f(x)-f(z)|>\lambda\rho(x,z)\phi(x,z)^\gamma\}}
 \frac{\phi(x,z)^{\gamma p}}{V(x,z)}\,d\mu(z)
 \right]^{\frac1p}.
\end{equation}
By \eqref{eq:energy}, for any measurable
$E_0\subset E$,
\begin{equation}\label{eq:G-mass}
 \int_{E_0}[G^\phi(x)]^p\,d\mu(x)
 \leq\int_E[G^\phi(x)]^p\,d\mu(x)
 =\mathcal E^\phi_{\lambda,\gamma,p}(f;E).
\end{equation}

The following lemma gives a common pointwise estimate for the two
energy densities. We use the notation $\phi(B)$ introduced in
Section~\ref{sec:main-results}.

\begin{lemma}\label{lem:spatial-coarse}
Let $(X,\rho,\mu)$ be a doubling and reverse doubling metric measure space,
$p\in[1,\infty)$, $\gamma\in\mathbb R\setminus\{0\}$,
$\phi\in\{\rho,V\}$, and $A\in[1,\infty)$.
Let $f:X\to\mathbb R$ be measurable, and let $B$ be a ball with the radius
$r\in(0,\operatorname{diam} X)$.
Then there exists $c_0\in(0,1)$, depending only on $p$, $\gamma$,
$A$, and the geometric constants of $X$, such that
\begin{align}\label{eq:upper-scale-coarse}
 &G^\phi(u)+G^\phi(v)\leq c_0\phi(B(u,R))^\gamma\ \Longrightarrow\
 |f(u)-f(v)|\leq c_0^{-1}R\phi(B(u,R))^\gamma
\end{align}
for any $u,v\in AB$ and
$R\in[\rho(u,v),Ar]\cap(0,\infty)$, where $G^\phi$ is defined as in
\eqref{eq:G} with $\lambda=1$ and $E:=(2A)B$.
\end{lemma}

The proof of Lemma \ref{lem:spatial-coarse} uses the following elementary geometric estimate.

\begin{lemma}\label{lem:witness-ball}
Let $(X,\rho,\mu)$ be a doubling and reverse doubling metric measure space.
Then, for any $u,v\in X$ and $R\in(0,\infty)$ satisfying
$\rho(u,v)\leq R\leq\operatorname{diam} X$, there exists a ball
$S\subset B(u,\frac R2)$ such that, for any $z\in S$,
\begin{equation}\label{eq:witness-distance}
 \rho(u,z)\sim R,
 \quad
 \rho(v,z)\sim R,
\end{equation}
and
\begin{equation}\label{eq:witness-measure}
 \mu(S)\gtrsim\mu(B(u,R)),
\end{equation}
where the implicit positive constants depend only on the geometric constants of $X$.
\end{lemma}

\begin{proof}
Fix $u,v\in X$ and $R\in(0,\infty)$ with
$\rho(u,v)\leq R\leq\operatorname{diam} X$, and let $U:=B(u,\frac R4)$.
By \eqref{eq:volume-reverse-growth} and the doubling condition,
we obtain, for any $\varepsilon\in(0,\frac18)$,
\[
 \mu\left(B(u,\varepsilon R)\cup B(v,\varepsilon R)\right)
 \lesssim\varepsilon^d
 \left[V(u,\tfrac R4)+V(v,\tfrac R4)\right]
 \lesssim\varepsilon^d\mu(U),
\]
where the implicit positive constants depend only on the geometric
constants of $X$.
For $\varepsilon$ sufficiently small, depending only on these
constants, the preceding estimate gives a point
\[
 z_0\in U\setminus
 \left[B(u,\varepsilon R)\cup B(v,\varepsilon R)\right].
\]
Let $S:=B(z_0,\frac{\varepsilon R}{4})$.
By the triangle inequality, we obtain $S\subset B(u,\frac R2)$
and \eqref{eq:witness-distance}.
Since $B(u,R)\subset B(z_0,\frac{5R}{4})$, it follows from
the doubling condition that $\mu(B(u,R))\lesssim\mu(S)$,
which shows \eqref{eq:witness-measure} and completes the proof of
Lemma~\ref{lem:witness-ball}.
\end{proof}

\begin{proof}[Proof of Lemma~\ref{lem:spatial-coarse}]
Fix $u,v\in AB$ and $R\in[\rho(u,v),Ar]\cap(0,\infty)$.
Let $R_0:=\min\{R,\operatorname{diam} X\}$.
Note that $\rho(u,v)\leq R_0$.
Since $\frac{R}{A}\leq r<\operatorname{diam} X$ and $A\in[1,\infty)$,
it follows that $\frac{R}{A}\leq R_0\leq R$.
Applying Lemma~\ref{lem:witness-ball} with radius $R_0$, we obtain
a ball $S\subset B(u,\frac{R_0}{2})\subset E$.
From \eqref{eq:witness-distance}, \eqref{eq:witness-measure}, and
the doubling condition, we infer that $\mu(S)\gtrsim V(u,R)$,
$\rho(u,z)\sim\rho(v,z)\sim R$, and
$V(u,z)\sim V(v,z)\sim V(u,R)$ for any $z\in S$. Thus,
\begin{equation}\label{eq:spatial-phi-comparison}
 \phi(u,z)\sim\phi(v,z)\sim\phi(B(u,R)).
\end{equation}
Let $c_0\in(0,1)$ be sufficiently small, depending only on $p$,
$\gamma$, $A$, and the geometric constants of $X$. Using these comparisons,
we find that, for any $w\in\{u,v\}$ and $z\in S$,
\begin{equation}\label{eq:spatial-threshold}
 2\rho(w,z)\phi(w,z)^\gamma
 \leq c_0^{-1}R\phi(B(u,R))^\gamma.
\end{equation}
Suppose that $G^\phi(u)+G^\phi(v)\leq c_0\phi(B(u,R))^\gamma$
but $|f(u)-f(v)|>c_0^{-1}R\phi(B(u,R))^\gamma$.
For $w\in\{u,v\}$, let
\[
 S_w:=\{z\in S:|f(w)-f(z)|>
 \rho(w,z)\phi(w,z)^\gamma\}.
\]
By the triangle inequality, \eqref{eq:spatial-threshold}, and
the assumed lower bound for $|f(u)-f(v)|$, we obtain $S=S_u\cup S_v$.
From \eqref{eq:G} with $\lambda=1$, \eqref{eq:spatial-phi-comparison},
\eqref{eq:witness-measure}, and the doubling condition, we deduce that
\[
 [G^\phi(u)]^p+[G^\phi(v)]^p
 \gtrsim\frac{\phi(B(u,R))^{\gamma p}}{V(u,R)}
 [\mu(S_u)+\mu(S_v)]
 \gtrsim\phi(B(u,R))^{\gamma p}.
\]
Thus, $G^\phi(u)+G^\phi(v)\gtrsim\phi(B(u,R))^\gamma$, where the
implicit positive constant depends only on $p$, $\gamma$, $A$,
and the geometric constants of $X$. Since $c_0$ was chosen sufficiently small,
the last estimate contradicts the assumption in
\eqref{eq:upper-scale-coarse}.
This completes the proof of Lemma~\ref{lem:spatial-coarse}.
\end{proof}

\section{The Curve Method}\label{sec:metric-proofs}

In Subsection~\ref{subsec:curve-poincare}, we combine the estimates
from Section~\ref{sec:one-dimensional} to obtain truncated difference
estimates and Poincar\'e inequalities.
We then derive asymptotic lower bounds by discrete
convolutions in Subsection~\ref{subsec:curve-asymptotic}.

Throughout the proofs in Section~\ref{sec:metric-proofs},
we work with a finite-valued Borel
representative of $f$. Such a representative exists because $f$ is
finite almost everywhere and measurable with respect to the completion
of the Borel measure $\mu$.
With this choice, the densities $G^\phi$ in \eqref{eq:G} are
Borel measurable. Thus, for any Lipschitz curve $\eta$, the functions
$f\circ\eta$ and $G^\phi\circ\eta$ are Lebesgue measurable on its
parameter interval and hence satisfy the measurability assumptions in
Proposition~\ref{prop:volume-one-dimensional} and
Corollary~\ref{prop:one-dimensional-endpoint}.

\subsection{Poincar\'e Inequalities}
\label{subsec:curve-poincare}

In this subsection, we use Lemma~\ref{lem:riesz-curve} to pass
from the one-dimensional estimates in
Subsection~\ref{subsec:one-dimensional-endpoints} to truncated
difference estimates on balls. We then deduce Poincar\'e
inequalities for negative parameters.

\begin{lemma}\label{lem:curve-local-estimates}
Let $p\in[1,\infty)$, and let $(X,\rho,\mu)$ be a complete
doubling $q$-Poincar\'e space for some $q\in[1,p]$
with a lower dimension $d$.
Let $\phi\in\{\rho,V\}$ and $\gamma\in\mathbb R\setminus\{0\}$.
Then there exist constants $c\in(0,\infty)$ and $\Lambda\in[1,\infty)$
such that, for any ball $B$ with its radius $r\in(0,\operatorname{diam} X)$
and for any measurable function $f:X\to\mathbb R$,
the following statements hold.
\begin{enumerate}[{\rm (i)}]
\item If $\phi=\rho$ and $\gamma\in(-\infty,-\frac qp]$ or
$\phi=V$ and $\gamma\in(-\infty,-\frac q{dp}]$, then
\begin{equation}\label{eq:curve-negative-local}
 \fint_B\fint_B
 \left(|f(x)-f(y)|-cr\phi(B)^\gamma\right)_+^p
 \,d\mu(x)\,d\mu(y)
 \lesssim\frac{r^p}{\mu(B)}
 \mathcal E^\phi_{1,\gamma,p}(f;\Lambda B).
\end{equation}
\item If $\gamma\in(0,\infty)$, then
\begin{equation}\label{eq:curve-positive-local}
 \fint_B\fint_B
 \min\left\{|f(x)-f(y)|,cr\phi(B)^\gamma\right\}^p
 \,d\mu(x)\,d\mu(y)
 \lesssim\frac{r^p}{\mu(B)}
 \mathcal E^\phi_{1,\gamma,p}(f;\Lambda B).
\end{equation}
\end{enumerate}
The constants $c$, $\Lambda$, and the implicit positive constants
depend only on $p$, $\gamma$, and the geometric constants of $X$.
\end{lemma}

\begin{proof}
Fix $B$ with its radius $r$, and let $L$ be as in
Lemma~\ref{lem:riesz-curve}. Let $\sigma:=\frac{p}{q}$,
$T:=2Lr$, and $\Lambda:=16L$.
We use $G^\phi$ defined as in \eqref{eq:G} with $\lambda=1$ and
$E=\Lambda B$, and may assume that
$\mathcal E^\phi_{1,\gamma,p}(f;\Lambda B)<\infty$.
Fix any distinct $x,y\in B$.
Applying Lemma~\ref{lem:riesz-curve} with such $x,y$ and
$F=(G^\phi)^\sigma$, we obtain an
arclength-parametrized curve $\eta:[0,\ell_\eta]\to X$ joining
$x$ to $y$, with $\ell_\eta\leq L\rho(x,y)\leq T$, which
implies that $\eta([0,\ell_\eta])\subset(8L)B$.
Since $\sigma q=p$ and $B(x,L\rho(x,y))\subset\Lambda B$,
it follows from \eqref{eq:riesz-curve} that
\begin{equation}\label{eq:curve-cost}
 \left[\int_\eta(G^\phi)^\sigma\,d s\right]^q
 \lesssim r^{q-1}\int_{\Lambda B}[G^\phi(z)]^p
 \left[K(x,z)+K(y,z)\right]\,d\mu(z).
\end{equation}
The constant extension of $\eta$ to $[0,T]$ given by
\eqref{eq:proj} with $A=\ell_\eta$ is $1$-Lipschitz and is
still denoted by $\eta$.
For $v:=f\circ\eta$ and $h:=G^\phi\circ\eta$, we have
\[
 \int_0^Th(s)^\sigma\,d s
 \leq\int_\eta(G^\phi)^\sigma\,d s+T[G^\phi(y)]^\sigma.
\]
Since $\eta([0,T])\subset(8L)B$ and $4T=8Lr$, it follows from
\eqref{eq:upper-scale-coarse} with $A=8L$ that
$c_0v$ and $\frac{h}{c_0}$ satisfy
\eqref{eq:coarse-1d} when $\phi=\rho$ and
\eqref{eq:volume-one-dimensional-coarse} when $\phi=V$.
We apply \eqref{eq:endpoint-chain} or \eqref{eq:endpoint-conclusion}
with exponent $\sigma$ when $\phi=\rho$, and \eqref{eq:v-p-conclusion}
or \eqref{eq:volume-positive-conclusion} with $C_\ast=2L$ when
$\phi=V$, to $c_0v$ and $\frac{h}{c_0}$ in place of $v$ and $h$,
respectively.
Note that $T<2L\operatorname{diam} X$.
From these, we deduce that
\begin{align}\label{eq:curve-endpoint-bound}
 &\begin{cases}
 \left(|v(0)-v(T)|-c_1T\phi(B(x,T))^\gamma\right)_+
 &\text{if }\gamma\in(-\infty,0),\\
 \min\left\{|v(0)-v(T)|,c_1T\phi(B(x,T))^\gamma\right\}
 &\text{if }\gamma\in(0,\infty)
 \end{cases}\notag\\
 &\quad\lesssim T\left[h(0)^\sigma+h(T)^\sigma+
 \fint_0^Th(s)^\sigma\,d s\right]^{\frac1\sigma},
\end{align}
where $c_1\in(0,\infty)$ depends only on $p$, $\gamma$, and the
geometric constants of $X$.
Since $T\sim r$ and $V(x,T)\sim\mu(B)$, it follows that
$T\phi(B(x,T))^\gamma\sim r\phi(B)^\gamma$.
By taking the $p$-th power in \eqref{eq:curve-endpoint-bound} and using
$v(0)=f(x)$, $v(T)=f(y)$, and \eqref{eq:curve-cost}, we obtain
\begin{align}\label{eq:pointwise-lift}
 &\begin{cases}
 \left(|f(x)-f(y)|-cr\phi(B)^\gamma\right)_+^p
 &\text{if }\gamma\in(-\infty,0),\\
 \min\left\{|f(x)-f(y)|,cr\phi(B)^\gamma\right\}^p
 &\text{if }\gamma\in(0,\infty)
 \end{cases}\notag\\
 &\quad\lesssim r^p\left\{[G^\phi(x)]^p+[G^\phi(y)]^p\right\}
 +r^{p-1}\int_{\Lambda B}[G^\phi(z)]^p
 \left[K(x,z)+K(y,z)\right]\,d\mu(z)
\end{align}
for a suitable constant $c\in(0,\infty)$ depending only on $p$,
$\gamma$, and the geometric constants of $X$.
By averaging \eqref{eq:pointwise-lift} over $B\times B$ and applying
\eqref{eq:riesz-average} and \eqref{eq:G-mass}, we conclude that
\eqref{eq:curve-negative-local} and \eqref{eq:curve-positive-local} hold.
This completes the proof of Lemma~\ref{lem:curve-local-estimates}.
\end{proof}

Using Lemma~\ref{lem:curve-local-estimates}(i) with $f$ replaced by $\frac f\lambda$
and using $s\leq a+(s-a)_+$ for $s,a\in[0,\infty)$, we obtain
\[
 \fint_B\fint_B|f(x)-f(y)|^p\,d\mu(x)\,d\mu(y)
 \lesssim\lambda^pr_B^p\phi(B)^{\gamma p}
 +\frac{r_B^p}{\mu(B)}\mathcal E^\phi_{\lambda,\gamma,p}(f;\Lambda B).
\]
Thus, \eqref{eq:NPI} holds in the parameter ranges of
Lemma~\ref{lem:curve-local-estimates}(i).

\subsection{Asymptotic Lower Bounds}
\label{subsec:curve-asymptotic}

In this subsection, we use discrete convolutions to derive asymptotic lower bounds from the Poincar\'e inequality
\eqref{eq:NPI} for negative parameters and from
Lemma~\ref{lem:curve-local-estimates}(ii) for positive parameters.

\begin{proposition}\label{prop:curve-asymptotic}
Let $p\in[1,\infty)$, and let $(X,\rho,\mu)$ be a complete doubling
$q$-Poincar\'e space for some $q\in[1,p]$ with a lower dimension $d$.
Let $\phi\in\{\rho,V\}$ and $\gamma\in\mathbb R\setminus\{0\}$.
Then \eqref{eq:asymptotic-lower} holds for any
$f\in L^1_{\mathrm{loc}}(\mu)$ in either of the following cases:
\begin{enumerate}[{\rm (i)}]
\item $\gamma\in(-\infty,0)$ and \eqref{eq:NPI} holds with fixed constants
$C\in(0,\infty)$ and $\Lambda\in[1,\infty)$,
\item $\gamma\in(0,\infty)$.
\end{enumerate}
The positive constant in \eqref{eq:asymptotic-lower} depends, in the case (i),
only on $p$, $\gamma$, $C$, $\Lambda$, and the geometric constants of $X$
and, in the case (ii), only on $p$, $\gamma$, and the geometric constants of $X$.
\end{proposition}

We first recall the discrete approximations
used in the proof.
Given $f\in L^1_{\mathrm{loc}}(\mu)$ and $r\in(0,\infty)$ and
$r\in(0,\frac14\operatorname{diam} X]$, we consider the discrete convolution
\begin{equation}\label{eq:approximate-identity}
 f_r:=\sum_{i\in I_r} f_{B_i}\varphi_i,
\end{equation}
where $I_r\subset\mathbb N$ and the centers $\{x_i\}_{i\in I_r}$ of balls
$\{B_i\}_{i\in I_r}$ form a
maximal $r$-separated set. Thus, the balls $\{B_i:=B(x_i,\frac12 r)\}_{i\in I_r}$
are pairwise disjoint and the union of $\{2B_i\}_{i\in I_r}$ covers $X$. For any
$i\in I_r$, the function
$\varphi_i$ satisfies
\begin{equation}\label{eq:partition}
 0\leq\varphi_i\leq1,\
 \operatorname{supp}\varphi_i\subset3B_i,\
 \sum_{i\in I_r}\varphi_i=1,\ \text{and}\
 \operatorname{Lip}(\varphi_i)\lesssim r^{-1}.
\end{equation}
Here, the implicit positive constant depends only on the doubling constant.
For example, let
$\psi_i(x):=\min\{1,(5-4\rho(x,x_i)/r)_+\}$ for any
$x\in X$ and
$\varphi_i:=\psi_i/\sum_j\psi_j$.
For any $i\in I_r$, the
function $\psi_i$ equals one on $2B_i$ and
has support in $3B_i$, and $\{\psi_i\}_{i\in I_r}$
have uniformly bounded overlap and their sum is at least one.

Clearly, $f_r\in\operatorname{Lip}_{\mathrm{loc}}(X)$ and
$f_r\to f$ in $L^1_{\mathrm{loc}}(\mu)$ as $r\to0^+$
(see also the proof of \cite[Proposition~8.2]{HK:W1p}).
For an open set $\Omega\subset X$, write
\[
 I_{\Omega,r}:=\{i\in I_r:2B_i\cap\Omega\neq\emptyset\}.
\]

We record the following estimate for these discrete convolutions that will be
used below.

\begin{lemma}\label{lem:discrete-convolution}
Let $p\in[1,\infty)$, $(X,\rho,\mu)$ be a complete doubling
metric measure space, and $f\in L^1_{\mathrm{loc}}(\mu)$.
Let $f_r$ be as in \eqref{eq:approximate-identity}.
Then, for any open set $\Omega\subset X$,
\begin{equation*}
 \int_\Omega(\operatorname{lip} f_r)^p\,d\mu
 \lesssim\frac1{r^p}\sum_{i\in I_{\Omega,r}}
 \int_{8B_i}|f-f_{8B_i}|^p\,d\mu,
\end{equation*}
where the implicit positive constant depends only on $p$
and the doubling constant.
\end{lemma}

\begin{proof}
This proof borrows some ideas from the proof of
\cite[Lemma~3.3]{HK:W1p}.
Fix $i\in I_r$, $x\in2B_i$, and $y\in B(x,r)\setminus\{x\}$.
For any $j\in I_r$ with $\varphi_j(x)+\varphi_j(y)\neq0$, the support
condition in \eqref{eq:partition} gives $B_j\subset8B_i$.
Indeed, for $z\in B_j$, $\rho(x_i,z)<r+r+3r/2+r/2=4r$, which implies
$B_j\subset 8B_i$.
The doubling condition gives $\mu(B_j)\sim\mu(8B_i)$.
Moreover, the doubling condition and the disjointness of the balls give
$\sharp\{j\in I_r:\varphi_j(x)+\varphi_j(y)\neq0\}\lesssim1$.
From these estimates and \eqref{eq:partition}, we deduce that
\begin{align*}
 |f_r(x)-f_r(y)|=\left|\sum_{j\in I_r}(f_{B_j}-f_{8B_i})
       [\varphi_j(x)-\varphi_j(y)]\right|
       \lesssim\frac{\rho(x,y)}{r}
       \fint_{8B_i}|f-f_{8B_i}|\,d\mu.
\end{align*}
By dividing by $\rho(x,y)$, letting $y\to x$, and taking the $p$-th power,
we obtain
\[
 [\operatorname{lip} f_r(x)]^p
 \lesssim\frac1{r^p}\fint_{8B_i}|f-f_{8B_i}|^p\,d\mu,
 \quad x\in2B_i,\quad i\in I_r.
\]
Since $\{2B_i\}_{i\in I_r}$ covers $X$ and
$\mu(\Omega\cap2B_i)\leq\mu(8B_i)$ for each $i$, it follows that
 \begin{align*}
 \int_\Omega(\operatorname{lip} f_r)^p\,d\mu
 \leq\sum_{i\in I_{\Omega,r}}
 \int_{\Omega\cap2B_i}(\operatorname{lip} f_r)^p\,d\mu
 \lesssim\frac1{r^p}\sum_{i\in I_{\Omega,r}}
 \int_{8B_i}|f-f_{8B_i}|^p\,d\mu.
 \end{align*}
This completes the proof of Lemma~\ref{lem:discrete-convolution}.
\end{proof}

The following lemma is a consequence of \eqref{eq:variation-definition}
when $p=1$ and of the compactness and maximal function arguments used in
the proof of \cite[Proposition~8.2]{HK:W1p} when $p\in(1,\infty)$.

\begin{lemma}\label{lem:discrete-sobolev-closure}
Let $p\in[1,\infty)$, and let $(X,\rho,\mu)$ be a complete doubling
$p$-Poincar\'e space. If a sequence
$\{u_j\}_{j\in\mathbb N}\subset\operatorname{Lip}_{\mathrm{loc}}(X)$
converges to $f$ in $L^1_{\mathrm{loc}}(\mu)$, then
\[
 \mathcal E_p(f;X)
 \lesssim\liminf_{j\to\infty}\int_X(\operatorname{lip} u_j)^p\,d\mu,
\]
where the implicit positive constant depends only on $p$ and the
geometric constants of $X$.
\end{lemma}

\begin{proof}
The case $p=1$ follows from \eqref{eq:variation-definition}.
Let $p\in(1,\infty)$. We may assume that the limit inferior on the
right-hand side is finite and denote it by $K$. We pass to a subsequence
along which the integrals converge to $K$.
By \cite[Lemma~2.6 and Theorem~2.8]{HK:W1p}, we obtain
$a\in(1,p)$ such that the weak $(1,a)$-Poincar\'e inequality holds
with $\operatorname{lip} u$ for any Lipschitz function $u$.
Let $\Lambda_a\in[1,\infty)$ denote its dilation constant, which
depends only on $p$ and the geometric constants of $X$.
Passing to a further subsequence, we may assume that
$(\operatorname{lip} u_j)^a\rightharpoonup h\geq0$ in $L^{p/a}(\mu)$, with
$\|h\|_{L^{p/a}(\mu)}^{p/a}\leq K$.
Fix a ball $B$. Since $X$ is complete and satisfies the doubling
condition, it follows that $X$ is proper. Thus, the restrictions of
$u_j$ to a sufficiently large closed ball admit Lipschitz extensions
to $X$. Using the weak $(1,a)$-Poincar\'e inequality on $B$,
the convergence $u_j\to f$ in $L^1_{\mathrm{loc}}(\mu)$, and the weak
convergence tested against $\mathbf 1_{\Lambda_a B}$, we obtain
\[
 \fint_B|f-f_B|\,d\mu
 \lesssim r_B\left(\fint_{\Lambda_a B}h\,d\mu\right)^{1/a}.
\]
Let $M$ be the Hardy--Littlewood maximal operator.
For any distinct Lebesgue points $x,y$ of $f$, let $r:=\rho(x,y)$.
From the preceding estimate and the doubling condition, we infer that
 \begin{align*}
 |f(x)-f(y)|
 &\leq \sum_{z\in\{x,y\}}\sum_{k\in\mathbb Z_+}
 \left|f_{B(z,2^{-k-1}r)}-f_{B(z,2^{-k}r)}\right|
 +\left|f_{B(x,r)}-f_{B(y,r)}\right|\\
 &\lesssim r\sum_{z\in\{x,y\}}\sum_{k\in\mathbb Z_+}
 2^{-k}[Mh(z)]^{\frac{1}{a}}\lesssim\rho(x,y)\left\{[Mh(x)]^{\frac{1}{a}}
 +[Mh(y)]^{\frac{1}{a}}\right\}.
 \end{align*}
Thus, $C(Mh)^{\frac{1}{a}}$ is a Haj\l asz upper gradient of $f$;
see also \cite[Proposition~8.2]{HK:W1p}.
Since $a\in(1,p)$, the strong maximal inequality yields
\[
 \mathcal E_p(f;X)\lesssim\|Mh\|_{L^{p/a}(\mu)}^{p/a}
 \lesssim\|h\|_{L^{p/a}(\mu)}^{p/a}\leq K.
\]
This completes the proof of Lemma~\ref{lem:discrete-sobolev-closure}.
\end{proof}

\begin{proof}[Proof of Proposition~\ref{prop:curve-asymptotic}]
We may assume that
\begin{equation*}
 K:=\liminf_{\lambda\to L_\gamma}
 \mathcal E^\phi_{\lambda,\gamma,p}(f;X)<\infty.
\end{equation*}
Choose $\lambda_j\to L_\gamma$ such that
\begin{equation}\label{eq:energy-sequence}
 \mathcal E^\phi_{\lambda_j,\gamma,p}(f;X)\longrightarrow K.
\end{equation}
Let $\Lambda$ be the dilation in \eqref{eq:NPI} when
$\gamma\in(-\infty,0)$ and in
Lemma~\ref{lem:curve-local-estimates}(ii) when $\gamma\in(0,\infty)$.
For any fixed ball $B$ with $r_B\in(0,\operatorname{diam} X)$, we first show
\begin{equation}\label{eq:local-p}
 \int_B|f-f_B|^p\,d\mu
 \lesssim r_B^p\liminf_{j\to\infty}
 \mathcal E^\phi_{\lambda_j,\gamma,p}(f;\Lambda B).
\end{equation}
When $\gamma\in(-\infty,0)$, \eqref{eq:NPI} with $\lambda=\lambda_j$
gives
\[
 \int_B|f-f_B|^p\,d\mu
 \lesssim\lambda_j^p r_B^p\mu(B)\phi(B)^{\gamma p}
 +r_B^p\mathcal E^\phi_{\lambda_j,\gamma,p}(f;\Lambda B).
\]
Since $\lambda_j\to0^+$, it follows that \eqref{eq:local-p} holds.
When $\gamma\in(0,\infty)$, using
Lemma~\ref{lem:curve-local-estimates}(ii) with $f$ replaced by $\frac f{\lambda_j}$,
multiplying by $\lambda_j^p$, and using Fatou's lemma, we obtain
\eqref{eq:local-p} because $\lambda_j\to\infty$.

Fix a sufficiently small $r\in(0,\infty)$, and let $f_r$ be as in
\eqref{eq:approximate-identity}. By the doubling condition and the
disjointness of $\{B_i\}_{i\in I_r}$, the balls
$\{8\Lambda B_i\}_{i\in I_r}$ have
bounded overlap. In particular,
\[
 \sum_{i\in I_r}\mathbf 1_{8\Lambda B_i}(x)\mathbf 1_{8\Lambda B_i}(y)
 \leq\sum_{i\in I_r}\mathbf 1_{8\Lambda B_i}(x)\lesssim1.
\]
Tonelli's theorem therefore gives, for any $\lambda\in(0,\infty)$,
\begin{equation}\label{eq:two-point-overlap}
 \sum_{i\in I_r}
 \mathcal E^\phi_{\lambda,\gamma,p}(f;8\Lambda B_i)
 \lesssim\mathcal E^\phi_{\lambda,\gamma,p}(f;X).
\end{equation}
Using Lemma~\ref{lem:discrete-convolution} with $\Omega=X$,
\eqref{eq:local-p} for the balls $8B_i$, Fatou's lemma,
\eqref{eq:two-point-overlap}, and \eqref{eq:energy-sequence}, we obtain
\begin{align}\label{eq:global-gradient-bound}
 \int_X(\operatorname{lip} f_r)^p\,d\mu
 &\lesssim\sum_{i\in I_r}\liminf_{j\to\infty}
 \mathcal E^\phi_{\lambda_j,\gamma,p}(f;8\Lambda B_i)\notag\\
 &\leq\liminf_{j\to\infty}\sum_{i\in I_r}
 \mathcal E^\phi_{\lambda_j,\gamma,p}(f;8\Lambda B_i)
 \lesssim K,
\end{align}
where the implicit positive constants depend only on $p$, $\gamma$,
$C$, $\Lambda$, and the geometric constants of $X$ when
$\gamma\in(-\infty,0)$, and only on $p$, $\gamma$, and the geometric
constants of $X$ when $\gamma\in(0,\infty)$.
The assumption $q\in[1,p]$ also gives a weak $(1,p)$-Poincar\'e inequality
on $X$. Since $f_r\to f$ in $L^1_{\mathrm{loc}}(\mu)$ as $r\to0^+$,
it follows from \eqref{eq:global-gradient-bound} and
Lemma~\ref{lem:discrete-sobolev-closure} that
$\mathcal E_p(f;X)\lesssim K$.
This completes the proof of Proposition~\ref{prop:curve-asymptotic}.
\end{proof}

\section{The Adaptive Dyadic Method}\label{sec:volume-proofs}

The main goal of this section is to prove
Proposition~\ref{prop:volume-negative-poincare}.
Subsection~\ref{subsec:dyadic-refinement} establishes estimates for
balanced dyadic partitions, and
Subsection~\ref{subsec:volume-negative} uses them to prove the
Poincar\'e inequalities. We then combine this proposition with the
results of Section~\ref{sec:metric-proofs} to complete the proofs of
Theorems~\ref{thm:NPI} and~\ref{thm:asymptotic}.

\begin{proposition}\label{prop:volume-negative-poincare}
Let $p\in[1,\infty)$, and let $(X,\rho,\mu)$ be a complete doubling
$p$-Poincar\'e space with an upper dimension $D$ as in
\eqref{eq:volume-upper-growth}. Let $\phi\in\{\rho,V\}$.
If either $\phi=\rho$ and $\gamma\in(-\infty,-\frac{D}{p}]$ or
$\phi=V$ and $\gamma\in(-\infty,-\frac1p]$, then there exists
$\Lambda\in[1,\infty)$ such that, for any ball $B$
with its radius $r\in(0,\operatorname{diam} X)$ and for any
measurable function $f:X\to\mathbb R$,
\begin{equation*}
 \fint_B\fint_B|f(x)-f(y)|^p\,d\mu(x)\,d\mu(y)
 \lesssim r^p\phi(B)^{p\gamma}
 +\frac{r^p}{\mu(B)}\mathcal E^\phi_{1,\gamma,p}(f;\Lambda B),
\end{equation*}
where $\Lambda$ and the
implicit positive constant depend only on $p$, $\gamma$, and the
geometric constants of $X$.
\end{proposition}

\subsection{Dyadic Refinement Estimates}
\label{subsec:dyadic-refinement}

In this subsection, we first prove a Poincar\'e estimate for balanced
dyadic partitions. We then bound the weighted sum over a refined
partition by the sums over the initial and marked cubes.

We use the dyadic cube construction in
\cite[Theorem~2.2]{HytKairema}. Let
$\delta\in(0,1)$ and $c_1,C_1\in(0,\infty)$ satisfy $c_1<C_1$.  A
collection
\[
 \mathcal D:=\bigcup_{k\in\mathbb Z}\mathcal D_k,
 \quad
 \mathcal D_k:=\{Q_\alpha^k\subset X:\alpha\in\mathcal A_k\},
\]
where $\mathcal A_k$ is at most countable, is called a dyadic cube
system with parameters $(\delta,c_1,C_1)$ if the following properties
hold:
\begin{enumerate}[(i)]
 \item for any $k\in\mathbb Z$, the cubes in $\mathcal D_k$ are
 pairwise disjoint and
 $X=\bigcup_{\alpha\in\mathcal A_k}Q_\alpha^k$,
 \item if $\ell,k\in\mathbb Z$, $\ell\geq k$,
 $\alpha\in\mathcal A_k$, and $\beta\in\mathcal A_\ell$, then either
 $Q_\beta^\ell\subset Q_\alpha^k$ or
 $Q_\beta^\ell\cap Q_\alpha^k=\emptyset$,
 \item for any $k\in\mathbb Z$ and $\alpha\in\mathcal A_k$, there exists
 $x_\alpha^k\in X$ such that
 \begin{align}\label{eq:dyadic-ball}
  B(x_\alpha^k,c_1\delta^k)\subset Q_\alpha^k
  \subset B(Q_\alpha^k):=B(x_\alpha^k,C_1\delta^k),
 \end{align}
 \item if $P,Q\in\mathcal D$ and $P\subset Q$, then
 $B(P)\subset B(Q)$.
\end{enumerate}
For $Q=Q_\alpha^k\in\mathcal D$, the radius and center of $B(Q)$
are denoted by
\begin{align}\label{eq:dyadic-radii}
r_Q:=C_1\delta^k\ \text{and}\ x_Q:=x_\alpha^k.
\end{align}

For the choice of $\phi$ in Proposition~\ref{prop:volume-negative-poincare},
we also write $\phi(Q):=r_Q$ if $\phi=\rho$ and
$\phi(Q):=\mu(Q)$ if $\phi=V$.
A cube $J\in\mathcal D$ is called a \emph{child} of $I\in\mathcal D$ if
$J\subset I$ and $r_J=\delta r_I$; equivalently, $I$ is the \emph{parent} of
$J$.  For $j\in\mathbb Z_+$ and $Q\in\mathcal D$, write
\begin{equation}\label{eq:dyadic-descendants}
 \mathcal D_j(Q):=\left\{P\in\mathcal D:P\subset Q
 \text{ and }r_P=\delta^jr_Q\right\}
 \ \text{and}\ \mathcal D(Q):=\bigcup_{j\in\mathbb Z_+}\mathcal D_j(Q).
\end{equation}
Any cube has a unique parent and finitely many children, and it is the
disjoint union of its children. We fix such a system throughout this section.
The fixed dyadic parameters $\delta,c_1,C_1$ are also included among
the geometric constants of $X$ in this section.

Fix a sufficiently large constant $A_{\rm n}\in[4,\infty)$,
depending only on the geometric constants of $X$.
Distinct dyadic cubes $P,P'\in\mathcal{D}$ are called \emph{neighbors}, written $P\sim P'$, if
\begin{equation}\label{eq:dyadic-neighbors}
 A_{\rm n}B(P)\cap A_{\rm n}B(P')\neq\emptyset,
\end{equation}
where $B(P)$ is as in \eqref{eq:dyadic-ball}.
Following the convention in \cite[Section~2.4]{HKT17}, for
$\theta\in(0,1)$, a measurable function $f:X\to\mathbb R$, and
a measurable set $E\subset X$ with $\mu(E)\in(0,\infty)$,
define the \emph{$\theta$-median} of $f$ on $E$ by
\begin{equation}\label{eq:volume-quantiles}
 m_E^\theta(f):=\inf\left\{a\in\mathbb R:
 \mu\left(E\cap\left\{f>a\right\}\right)<\theta\mu(E)\right\}.
\end{equation}
This definition gives
\begin{equation}\label{eq:volume-median-convention}
 \mu\left(E\cap\left\{f<m_E^\theta(f)\right\}\right)
 \leq(1-\theta)\mu(E)
 \ \text{and}\
 \mu\left(E\cap\left\{f>m_E^\theta(f)\right\}\right)
 \leq\theta\mu(E).
\end{equation}
For $\theta=\frac12$, we write $a_E(f):=m_E^{\frac12}(f)$.
For $\varepsilon\in(0,\frac12)$, by
\eqref{eq:volume-quantiles} and
\eqref{eq:volume-median-convention}, we obtain
$m_P^{1-\varepsilon}(g)\leq a_P(g)\leq m_P^\varepsilon(g)$,
\begin{equation}\label{eq:volume-quantile-mass}
 \min\left\{\mu\left(P\cap\left\{g\leq m_P^{1-\varepsilon}(g)\right\}\right),
       \mu\left(P\cap\left\{g\geq m_P^\varepsilon(g)\right\}\right)\right\}
 \geq\varepsilon\mu(P),
\end{equation}
and
\begin{equation}\label{eq:volume-quantile-concentration}
 \mu\left(P\setminus\left\{m_P^{1-\varepsilon}(g)\leq g\leq m_P^\varepsilon(g)\right\}\right)
 \leq2\varepsilon\mu(P)
\end{equation}
for any $P\in\mathcal D$ and any measurable function $g:X\to\mathbb R$.

We consider a set $U\subset X$ of the form
\[
 U=\bigcup_{Q\in\mathcal T_0}Q,
\]
where $\mathcal T_0$ is a subcollection of some
$\mathcal D_{k_0}$ with $k_0\in\mathbb Z$. Fix such a representation of $U$.
A \emph{dyadic leaf partition} of $U$ is a pairwise disjoint collection $\mathcal T$ of
descendants of cubes in $\mathcal T_0$ such that
\begin{align}\label{eq:leaf-union}
 U=\bigcup_{P\in\mathcal T}P.
\end{align}
The cubes in $\mathcal T$ are called its \emph{leaves}.
A leaf partition $\mathcal T$ is said to be
\emph{balanced} if
\begin{align}\label{eq:neighbor-r}
 \frac{r_P}{r_{P'}}\in\left\{1,\delta,\delta^{-1}\right\}
 \quad\text{whenever }P,P'\in\mathcal T\text{ with }P\sim P'.
\end{align}
The terminology is motivated by the $2{:}1$ balance condition for
adaptive quadtrees and octrees; see
\cite{SundarSampathBiros08,BursteddeWilcoxGhattas11}.
Since the leaves are pairwise disjoint, from
\eqref{eq:dyadic-ball}, \eqref{eq:neighbor-r}, and the geometric
doubling condition, it follows that
\begin{align}\label{eq:overlap-bound}
 \sup_{x\in X}\sum_{P\in\mathcal T}\mathbf 1_{A_{\rm n}B(P)}(x)\lesssim1
 \ \text{and}\
 \sup_{P\in\mathcal T}\sharp \{P'\in\mathcal T:P'\sim P\}\lesssim1,
\end{align}
where the implicit positive constants depend only on the geometric
constants of $X$.

\begin{lemma}\label{lem:volume-balanced-cover}
Let $p\in[1,\infty)$, and let $(X,\rho,\mu)$ be a complete doubling
$p$-Poincar\'e space. Then there exists a constant $C\in[1,\infty)$,
depending only on the geometric constants of $X$, such that, for any
ball $B\subset X$ and any finite balanced leaf partition
$\mathcal T$ of $U\supset CB$ satisfying
\[
 P\in\mathcal T,\  A_{\rm n}B(P)\cap CB\neq\emptyset
 \ \Longrightarrow\
 r_P\leq r_{B}\ \text{and}\  A_{\rm n}B(P)\subset U,
\]
one has
\begin{equation}\label{eq:volume-balanced-cover}
 \inf_{c\in\mathbb R}\sum_{\genfrac{}{}{0pt}{}{P\in\mathcal T}{P\cap B\neq\emptyset}}
 \mu(P)|a_P-c|^p
 \lesssim r_B^p
 \sum_{\genfrac{}{}{0pt}{}{P\in\mathcal T}{
 A_{\rm n}B(P)\cap CB\neq\emptyset}}
 \sum_{\genfrac{}{}{0pt}{}{P'\in\mathcal T}{P'\sim P}}
 \frac{\mu(P)}{r_P^p}|a_P-a_{P'}|^p
\end{equation}
for any real coefficients $\{a_P\}_{P\in\mathcal T}$, where the
implicit positive constant depends only on $p$ and the geometric
constants of $X$.
\end{lemma}

\begin{proof}
By the quasiconvexity of complete doubling Poincar\'e spaces and
\cite[Theorem~8.3.2 and Corollary~9.1.36]{HKST}, $X$ supports a weak
$(p,p)$-Poincar\'e inequality with a dilation $\Lambda\in[1,\infty)$,
where $\Lambda$ and the Poincar\'e constant depend only on $p$
and the geometric constants of $X$.
Choose $C:=4\Lambda$ and let $B_*:=3B$.
Let $U$ and its finite balanced leaf partition
$\mathcal T$ satisfy the assumptions of the present lemma.
We use standard normalized distance cutoffs, as in the Lipschitz
partition of unity construction in \cite[Section~4.1]{HKST}.
For any $P\in\mathcal T$, define
\[
 \psi_P:=\max\left\{0,
 1-\frac{\operatorname{dist}(\cdot,P)}{r_P}\right\}
 \ \text{and}\
 \Psi:=\sum_{Q\in\mathcal T}\psi_Q.
\]
By \eqref{eq:dyadic-ball}, we obtain, for any $P\in\mathcal T$,
\begin{align}\label{eq:supp-lip}
 \operatorname{supp}\psi_P\subset\overline{2B(P)}
 \subset A_{\rm n}B(P)\ \text{and}\
 \operatorname{Lip}\psi_P\lesssim r_P^{-1}.
\end{align}
From \eqref{eq:leaf-union}, it follows that $\Psi\geq1$ on $U$.
Since $\Lambda B_*\subset CB$, we infer from the assumption and
\eqref{eq:supp-lip} that $\operatorname{supp}\psi_P\subset U$
whenever $A_{\rm n}B(P)\cap\Lambda B_*\neq\emptyset$.
For any such $P$, define
\[
 \varphi_P(x):=
 \begin{cases}
  \displaystyle\frac{\psi_P(x)}{\Psi(x)}&
  \text{if }\Psi(x)>0,\\
  0&\text{if }\Psi(x)=0.
 \end{cases}
\]
We claim that, for any $P\in\mathcal T$ with
$A_{\rm n}B(P)\cap\Lambda B_*\neq\emptyset$,
$\operatorname{supp}\varphi_P\subset A_{\rm n}B(P)$
and $\operatorname{Lip}\varphi_P\lesssim r_P^{-1}$ and that
\begin{equation}\label{eq:volume-leaf-partition}
 \sum_{\genfrac{}{}{0pt}{}{P\in\mathcal T}{
 A_{\rm n}B(P)\cap\Lambda B_*\neq\emptyset}}\varphi_P=1
 \quad\text{on }\Lambda B_*.
\end{equation}
Using \eqref{eq:neighbor-r}, \eqref{eq:overlap-bound},
\eqref{eq:supp-lip}, and $\Psi\geq1$ on $\operatorname{supp}\psi_P$,
we obtain the support and the Lipschitz bounds.
For any $x\in\Lambda B_*$, only cubes $P\in\mathcal T$ with
$A_{\rm n}B(P)\cap\Lambda B_*\neq\emptyset$ contribute to $\Psi(x)$.
Since $\Psi(x)>0$, it follows that the sum in
\eqref{eq:volume-leaf-partition} equals $1$.

Fix $\{a_P\}_{P\in\mathcal T}\subset \mathbb{R}$. Define
\[
 u:=\sum_{\genfrac{}{}{0pt}{}{P\in\mathcal T}{
 A_{\rm n}B(P)\cap\Lambda B_*\neq\emptyset}}a_P\varphi_P
\]
and
\[
 g:=L\sum_{\genfrac{}{}{0pt}{}{P\in\mathcal T}{
 A_{\rm n}B(P)\cap\Lambda B_*\neq\emptyset}}\mathbf 1_{A_{\rm n}B(P)}
 \sum_{\genfrac{}{}{0pt}{}{P'\in\mathcal T}{P'\sim P}}
 \frac{|a_P-a_{P'}|}{r_P},
\]
where $L$ is a sufficiently large constant depending only
on the geometric constants of $X$.
Define the upper pointwise Lipschitz constant of $u$ by
\[
 h_u(x):=\limsup_{\genfrac{}{}{0pt}{}{y\to x}{y\neq x}}
 \frac{|u(y)-u(x)|}{\rho(x,y)}.
\]
Since $u$ is a finite sum of Lipschitz functions by
\eqref{eq:volume-leaf-partition}, it follows from
\cite[Lemma~6.2.6]{HKST} that $h_u$ is an upper gradient of $u$.
We next show that $h_u\leq g$ on $\Lambda B_*$.
By \eqref{eq:dyadic-neighbors} and \eqref{eq:volume-leaf-partition},
we find that, for any $P\in\mathcal T$ with
$A_{\rm n}B(P)\cap\Lambda B_*\neq\emptyset$,
\begin{align}\label{eq:u-aP}
 u-a_P=\sum_{\genfrac{}{}{0pt}{}{P'\in\mathcal T,\ P'\sim P}{
 A_{\rm n}B(P')\cap\Lambda B_*\neq\emptyset}}
 (a_{P'}-a_P)\varphi_{P'}
 \quad\text{on }\Lambda B_*\cap A_{\rm n}B(P).
\end{align}
Now, fix $x\in\Lambda B_*\subset U$. Using \eqref{eq:leaf-union}, we find that $x$
belongs to a unique $P\in\mathcal T$. Since $P\subset A_{\rm n}B(P)$,
it follows that $\Lambda B_*\cap A_{\rm n}B(P)$ is an open neighborhood of $x$.
By \eqref{eq:u-aP}, \eqref{eq:neighbor-r}, and
\eqref{eq:volume-leaf-partition}, we obtain
\begin{align}\label{eq:h-g}
 h_u(x)
 \leq \sum_{\genfrac{}{}{0pt}{}{P'\in\mathcal T,\ P'\sim P}{
 A_{\rm n}B(P')\cap\Lambda B_*\neq\emptyset}}
 |a_{P'}-a_P|\operatorname{Lip}\varphi_{P'}\lesssim \frac1{r_P}\sum_{\genfrac{}{}{0pt}{}{P'\in\mathcal T}{P'\sim P}}
 |a_{P'}-a_P|
 \leq g(x).
\end{align}
From \eqref{eq:dyadic-ball}, \eqref{eq:overlap-bound}, and the
doubling condition, we deduce that
\begin{equation}\label{eq:volume-leaf-upper-gradient-integral}
 \int_X g^p\,d\mu
 \lesssim
 \sum_{\genfrac{}{}{0pt}{}{P\in\mathcal T}{
 A_{\rm n}B(P)\cap\Lambda B_*\neq\emptyset}}
 \sum_{\genfrac{}{}{0pt}{}{P'\in\mathcal T}{P'\sim P}}
 \frac{\mu(P)}{r_P^p}|a_P-a_{P'}|^p.
\end{equation}
By the assumption and \eqref{eq:dyadic-ball}, we conclude that
$r_P\leq r_B$ and $P\subset B_*$ whenever $P\cap B\neq\emptyset$.
Using \eqref{eq:u-aP} and \eqref{eq:overlap-bound}, we obtain,
for any such $P$ and any $c\in\mathbb R$,
\[
 \mu(P)|a_P-c|^p
 \lesssim\int_P|u-c|^p\,d\mu
 +\frac{r_B^p}{r_P^p}\mu(P)
 \sum_{\genfrac{}{}{0pt}{}{P'\in\mathcal T}{P'\sim P}}
 |a_P-a_{P'}|^p.
\]
By summing over the pairwise disjoint leaves meeting $B$, we find that
\begin{align}\label{eq:volume-leaf-oscillation-integral}
 \sum_{\genfrac{}{}{0pt}{}{P\in\mathcal T}{P\cap B\neq\emptyset}}\mu(P)|a_P-c|^p
 &\lesssim \int_{B_*}|u-c|^p\,d\mu+r_B^p
 \sum_{\genfrac{}{}{0pt}{}{P\in\mathcal T}{P\cap B\neq\emptyset}}
 \sum_{\genfrac{}{}{0pt}{}{P'\in\mathcal T}{P'\sim P}}
 \frac{\mu(P)}{r_P^p}|a_P-a_{P'}|^p.
\end{align}
Finally, taking $c=u_{B_*}$ and applying the weak
$(p,p)$-Poincar\'e inequality with upper gradient $h_u$ to the first
term on the right-hand side of \eqref{eq:volume-leaf-oscillation-integral},
we obtain \eqref{eq:volume-balanced-cover} by \eqref{eq:h-g} and
\eqref{eq:volume-leaf-upper-gradient-integral}.
This completes the proof of Lemma~\ref{lem:volume-balanced-cover}.
\end{proof}

We next recall the refinement construction used in the second estimate.
For a finite leaf partition $\mathcal P$ and $Q\in\mathcal P$, write
\[
 \mathsf{Ref}_Q(\mathcal P)
 :=(\mathcal P\setminus\{Q\})\cup\mathcal D_1(Q).
\]
The refinement
$\mathcal P\mapsto \mathsf{Ref}_Q(\mathcal P)$
is called a {\em balancing refinement} if there exists $R\in\mathcal P$ such that
\begin{align}\label{eq:volume-balancing-operation}
Q\sim R
\ \text{and}\
r_R\leq\delta^2r_Q.
\end{align}
Fix $k_0\in\mathbb Z$, $J\in\mathbb Z_+$, and a finite nonempty
$\mathcal T_0\subset\mathcal D_{k_0}$.
For $j\in\{1,\ldots,J\}$, let $P_j$ be any leaf of
$\mathcal T_{j-1}$, called the \emph{marked leaf}.
We obtain $\mathcal T_j$ from
$\mathsf{Ref}_{P_j}(\mathcal T_{j-1})$ by repeatedly applying any
balancing refinement \eqref{eq:volume-balancing-operation} available
in the current partition until none is possible.
Note that, by definition, balancing refinements do not decrease the
minimum leaf radius. Thus, balancing at any step terminates after
finitely many refinements, with $\mathcal T_j$ balanced.
Figure~\ref{fig:volume-refinement-process} illustrates this construction.

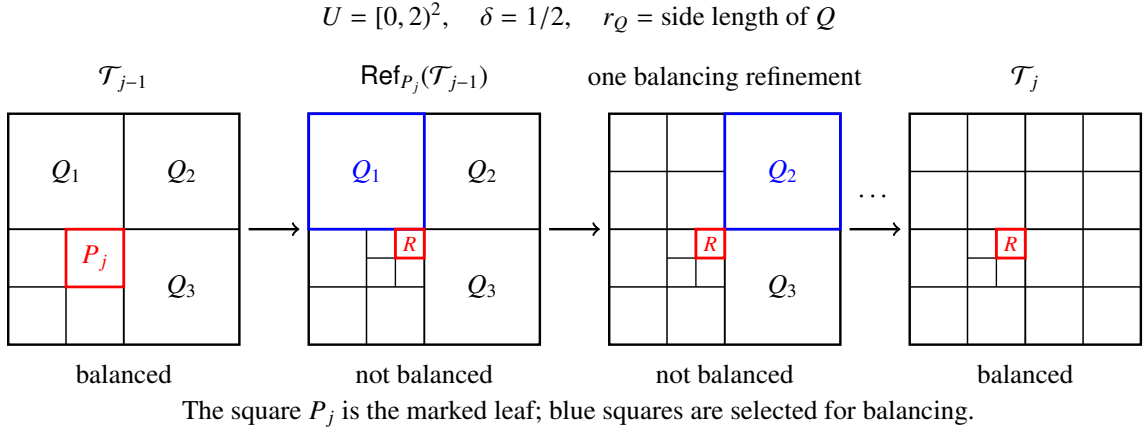
\begin{figure}[!htbp]
\centering
\definecolor{markedleaf}{HTML}{FF0000}
\resizebox{\textwidth}{!}{%
\begin{tikzpicture}[x=1cm,y=1cm,font=\small]
\node[inner sep=0pt,font=\small] at (7.5,5.65) {$U=[0,2)^2,\quad\delta=1/2,\quad r_Q=\text{side length of }Q$};
\draw[draw=black,line width=0.9pt] (0.1,1.4) rectangle (3.1,4.4);
\draw[line width=0.7pt] (1.6,1.4)--(1.6,4.4);
\draw[line width=0.7pt] (0.1,2.9)--(3.1,2.9);
\draw[line width=0.55pt] (0.85,1.4)--(0.85,2.9);
\draw[line width=0.55pt] (0.1,2.15)--(1.6,2.15);
\draw[draw=markedleaf,line width=1pt] (0.85,2.15) rectangle (1.6,2.9);
\node[inner sep=0pt,font=\small,text=markedleaf] at (1.225,2.525) {$P_j$};
\node[inner sep=0pt,font=\small] at (0.85,3.65) {$Q_1$};
\node[inner sep=0pt,font=\small] at (2.35,3.65) {$Q_2$};
\node[inner sep=0pt,font=\small] at (2.35,2.15) {$Q_3$};
\node[inner sep=0pt,font=\small] at (1.6,4.85) {$\mathcal T_{j-1}$};
\node[inner sep=0pt,font=\small] at (1.6,1.02) {balanced};
\draw[->,black,line width=0.9pt] (3.22,2.9)--(3.88,2.9);
\draw[draw=black,line width=0.9pt] (4,1.4) rectangle (7,4.4);
\draw[line width=0.7pt] (5.5,1.4)--(5.5,4.4);
\draw[line width=0.7pt] (4,2.9)--(7,2.9);
\draw[line width=0.55pt] (4.75,1.4)--(4.75,2.9);
\draw[line width=0.55pt] (4,2.15)--(5.5,2.15);
\draw[line width=0.5pt] (5.125,2.15)--(5.125,2.9);
\draw[line width=0.5pt] (4.75,2.525)--(5.5,2.525);
\draw[draw=blue,line width=1pt] (4,2.9) rectangle (5.5,4.4);\draw[draw=markedleaf,line width=1.05pt] (5.125,2.525) rectangle (5.5,2.9);
\node[inner sep=0pt,font=\scriptsize,text=markedleaf] at (5.3125,2.7125) {$R$};
\node[inner sep=0pt,font=\small,text=blue] at (4.75,3.65) {$Q_1$};
\node[inner sep=0pt,font=\small] at (6.25,3.65) {$Q_2$};
\node[inner sep=0pt,font=\small] at (6.25,2.15) {$Q_3$};
\node[inner sep=0pt,font=\small] at (5.5,4.85) {$\mathsf{Ref}_{P_j}(\mathcal T_{j-1})$};
\node[inner sep=0pt,font=\small] at (5.5,1.02) {not balanced};
\draw[->,black,line width=0.9pt] (7.12,2.9)--(7.78,2.9);
\draw[draw=black,line width=0.9pt] (7.9,1.4) rectangle (10.9,4.4);
\draw[line width=0.7pt] (9.4,1.4)--(9.4,4.4);
\draw[line width=0.7pt] (7.9,2.9)--(10.9,2.9);
\draw[line width=0.55pt] (8.65,1.4)--(8.65,2.9);
\draw[line width=0.55pt] (7.9,2.15)--(9.4,2.15);
\draw[line width=0.55pt] (8.65,2.9)--(8.65,4.4);
\draw[line width=0.55pt] (7.9,3.65)--(9.4,3.65);
\draw[line width=0.5pt] (9.025,2.15)--(9.025,2.9);
\draw[line width=0.5pt] (8.65,2.525)--(9.4,2.525);
\draw[draw=blue,line width=1pt] (9.4,2.9) rectangle (10.9,4.4);\draw[draw=markedleaf,line width=1.05pt] (9.025,2.525) rectangle (9.4,2.9);
\node[inner sep=0pt,font=\scriptsize,text=markedleaf] at (9.2125,2.7125) {$R$};
\node[inner sep=0pt,font=\small,text=blue] at (10.15,3.65) {$Q_2$};
\node[inner sep=0pt,font=\small] at (10.15,2.15) {$Q_3$};
\node[inner sep=0pt,font=\small] at (9.4,4.85) {one balancing refinement};
\node[inner sep=0pt,font=\small] at (9.4,1.02) {not balanced};
\draw[->,black,line width=0.9pt] (11.02,2.9)--(11.68,2.9);
\node[inner sep=0pt,font=\small] at (11.35,3.3) {$\cdots$};
\draw[draw=black,line width=0.9pt] (11.8,1.4) rectangle (14.8,4.4);
\draw[line width=0.7pt] (13.3,1.4)--(13.3,4.4);
\draw[line width=0.7pt] (11.8,2.9)--(14.8,2.9);
\draw[line width=0.55pt] (12.55,1.4)--(12.55,2.9);
\draw[line width=0.55pt] (11.8,2.15)--(13.3,2.15);
\draw[line width=0.55pt] (12.55,2.9)--(12.55,4.4);
\draw[line width=0.55pt] (11.8,3.65)--(13.3,3.65);
\draw[line width=0.55pt] (14.05,2.9)--(14.05,4.4);
\draw[line width=0.55pt] (13.3,3.65)--(14.8,3.65);
\draw[line width=0.55pt] (14.05,1.4)--(14.05,2.9);
\draw[line width=0.55pt] (13.3,2.15)--(14.8,2.15);
\draw[line width=0.5pt] (12.925,2.15)--(12.925,2.9);
\draw[line width=0.5pt] (12.55,2.525)--(13.3,2.525);
\draw[draw=markedleaf,line width=1.05pt] (12.925,2.525) rectangle (13.3,2.9);
\node[inner sep=0pt,font=\scriptsize,text=markedleaf] at (13.1125,2.7125) {$R$};
\node[inner sep=0pt,font=\small] at (13.3,4.85) {$\mathcal T_j$};
\node[inner sep=0pt,font=\small] at (13.3,1.02) {balanced};
\node[inner sep=0pt,font=\small] at (7.5,0.5) {The square $P_j$ is the marked leaf; blue squares are selected for balancing.};
\node[inner sep=0pt,font=\small] at (7.5,-0.25) {$r_{P_j}=1/2,\quad r_R=1/4=\delta^2r_{Q_i},\quad Q_i\sim R,\quad i\in\{1,2,3\}$};
\end{tikzpicture}%
}
\caption{A dyadic-square example in $\mathbb R^2$ with the Euclidean
metric and $A_{\rm n}\geq4$.
The third panel shows the balancing refinement of $Q_1$ under
\eqref{eq:volume-balancing-operation}.
The last arrow includes the remaining balancing refinements of
$Q_2$ and $Q_3$.
The final leaves have radii $1/2$ or $1/4$, so the
partition is balanced. The figure illustrates the additional refinements caused by a single
mark. Neighboring cubes in the sense of
\eqref{eq:dyadic-neighbors} need not share a face.}
\label{fig:volume-refinement-process}
\end{figure}

The following lemma estimates
$\sum_{K\in\mathcal T_J}\mu(K)\phi(K)^{p\gamma}$
in terms of the initial cubes and the marked leaves.
The proof uses the lower envelope $L_{\mathcal P}$ defined in
\eqref{eq:volume-closure-envelope} below, which is unchanged by
balancing refinements and can decrease only when a marked leaf is
refined. For a fixed marked leaf $P_j$, the weights in \eqref{eq:w_j}
are nonzero for only boundedly many final leaves $K$ at a given
scale. Summing the factors $\frac{r_{P_j}}{r_K}$ over the scales
gives a convergent geometric series that bounds
$\sum_{K\in\mathcal T_J}w_j(K)$.

\begin{lemma}
\label{lem:volume-weighted-linear-closure}
Let $p$, $(X,\rho,\mu)$, $\phi$, and $\gamma$ be as in
Proposition~\ref{prop:volume-negative-poincare}.
For the preceding construction and any choices of the marked leaves
$P_j$ and balancing refinements,
\begin{align}\label{eq:volume-weighted-linear-closure}
 \sum_{K\in\mathcal T_J}\mu(K)\phi(K)^{p\gamma}
 \lesssim
 \sum_{I\in\mathcal T_0}\mu(I)\phi(I)^{p\gamma}
 +\sum_{j=1}^J\mu(P_j)\phi(P_j)^{p\gamma},
\end{align}
where the implicit positive constant depends only on $p$, $\gamma$,
and the geometric constants of $X$.
\end{lemma}

\begin{proof}
Choose $\varepsilon\in(0,\delta)$, depending only on the geometric
constants of $X$, sufficiently small such that
\begin{equation}\label{eq:volume-closure-epsilon}
 \delta(1+\varepsilon)\leq1\ \text{and}\
 \delta^2+\varepsilon[A_{\rm n}(1+\delta^2)+1]\leq\delta,
\end{equation}
and let $\kappa:=1-\delta+\varepsilon\in(0,1)$.
Let $\mathcal P$ be any finite dyadic leaf partition of
$\bigcup_{I\in\mathcal T_0}I$ as in \eqref{eq:leaf-union}.
For $Q\in\mathcal D$, define
\begin{equation}\label{eq:volume-closure-envelope}
 b_Q(z):=r_Q+\varepsilon\rho(z,x_Q)\ \text{and}\
 L_{\mathcal P}(z):=\min_{Q\in\mathcal P}b_Q(z)
 \quad(z\in X).
\end{equation}
Fix $Q\in\mathcal P$, and let $S_*\in\mathcal D_1(Q)$ contain $x_Q$.
Clearly,
$\rho(x_S,x_Q)<r_Q$ for any $S\in\mathcal D_1(Q)$ and
$\rho(x_{S_*},x_Q)<r_{S_*}$.
From \eqref{eq:dyadic-descendants},
\eqref{eq:volume-closure-epsilon}, and
\eqref{eq:volume-closure-envelope}, we infer that, for any
$S\in\mathcal D_1(Q)$ and $z\in X$,
\begin{align}\label{eq:ref-child}
 b_S(z)\geq b_Q(z)-\kappa r_Q\ \text{and}\
 b_{S_*}(z)\leq
 \delta(1+\varepsilon)r_Q+\varepsilon\rho(z,x_Q)\leq b_Q(z).
\end{align}
Since $\mathcal P\setminus\{Q\}\subset\mathsf{Ref}_Q(\mathcal P)$
and $b_{S_*}\leq b_Q$, it follows from
\eqref{eq:volume-closure-envelope} and \eqref{eq:ref-child} that
$L_{\mathsf{Ref}_Q(\mathcal P)}\leq L_{\mathcal P}$.
By \eqref{eq:ref-child} and \eqref{eq:volume-closure-envelope},
we obtain $b_S\geq L_{\mathcal P}-\kappa r_Q$ for any
$S\in\mathcal D_1(Q)$, while
$b_R\geq L_{\mathcal P}$ for any
$R\in\mathcal P\setminus\{Q\}$. Thus,
\begin{equation}\label{eq:volume-closure-decrement}
 0\leq L_{\mathcal P}-L_{\mathsf{Ref}_Q(\mathcal P)}
 \leq\kappa r_Q.
\end{equation}

Now assume that $\mathsf{Ref}_Q(\mathcal P)$ is a balancing
refinement, and let $R\in\mathcal P$ satisfy
\eqref{eq:volume-balancing-operation}.
From \eqref{eq:dyadic-neighbors} and
\eqref{eq:volume-closure-epsilon}, we deduce that, for any
$S\in\mathcal D_1(Q)$,
\[
 r_R+\varepsilon\rho(x_R,x_S)
 \leq
 \left\{\delta^2+\varepsilon[A_{\rm n}(1+\delta^2)+1]\right\}r_Q
 \leq\delta r_Q=r_S.
\]
By \eqref{eq:volume-closure-envelope}, we find that $b_R\leq b_S$
for any such $S$, and hence $b_R\leq b_{S_*}\leq b_Q$.
Since $R\in\mathcal P\setminus\{Q\}\subset\mathsf{Ref}_Q(\mathcal P)$,
it follows from \eqref{eq:volume-closure-envelope} that
\begin{equation}\label{eq:volume-closure-balance-invariance}
 L_{\mathsf{Ref}_Q(\mathcal P)}=L_{\mathcal P}.
\end{equation}
Let $L_j:=L_{\mathcal T_j}$ for $j\in\{0,\ldots,J\}$.
Using \eqref{eq:volume-closure-decrement} and
\eqref{eq:volume-closure-balance-invariance} in the construction of
$\mathcal T_j$, we obtain
\begin{equation}\label{eq:volume-closure-mark-decrement}
 0\leq L_{j-1}(z)-L_j(z)\leq\kappa r_{P_j},
 \quad j\in\{1,\ldots,J\},\quad z\in X.
\end{equation}

Let $C_{\rm r}:=\frac{2}{\delta}$ and
$A_{\rm r}:=2(\varepsilon^{-1}+\delta^{-1})$.
For $K\in\mathcal T_J$, $j\in\{1,\ldots,J\}$, and
$I\in\mathcal T_0$, define
\begin{align}\label{eq:w_j}
 w_j(K):=
 \begin{cases}\displaystyle
 \frac{r_{P_j}}{r_K}&\text{if }r_{P_j}\leq C_{\rm r}r_K
 \text{ and }\rho(x_{P_j},x_K)\leq A_{\rm r}r_K,\\
 0&\text{otherwise}
 \end{cases}
\end{align}
and
\begin{align}\label{eq:w_0}
 w_0(K,I):=
 \mathbf1_{\{r_K\leq r_I\leq2r_K,\
 \rho(x_I,x_K)\leq2r_K\}}.
\end{align}
We claim that, for any $K\in\mathcal T_J$,
\begin{equation}\label{eq:volume-closure-row}
 \sum_{I\in\mathcal T_0}w_0(K,I)+\sum_{j=1}^Jw_j(K)\geq1.
\end{equation}
Indeed,
let $I\in\mathcal T_0$ be the initial ancestor of $K$.
From \eqref{eq:dyadic-ball} and \eqref{eq:dyadic-radii}, it follows that
$\rho(x_I,x_K)<r_I$ and $r_K\leq r_I$.
If $r_I\leq2r_K$, then \eqref{eq:w_0} gives $w_0(K,I)=1$,
which proves \eqref{eq:volume-closure-row} in this case.
Assume that $r_I>2r_K$.
Since the cubes in $\mathcal T_0$ have the same radius, it follows from
\eqref{eq:volume-closure-envelope} that $L_0(x_K)\geq r_I>2r_K$
and $L_J(x_K)\leq r_K$.
By \eqref{eq:volume-closure-mark-decrement}, we find that
$t:=\min\{j\in\{1,\ldots,J\}:L_j(x_K)\leq2r_K\}$ is well defined
and $L_j(x_K)\leq2r_K$ for $j\in\{t,\ldots,J\}$.
For any $j\in\{t,\ldots,J\}$ with $L_j(x_K)<L_{j-1}(x_K)$,
\eqref{eq:volume-closure-balance-invariance} and
\eqref{eq:volume-closure-envelope} give
$L_{\mathsf{Ref}_{P_j}(\mathcal T_{j-1})}(x_K)
=L_j(x_K)=b_S(x_K)\leq2r_K$
for some $S\in\mathcal D_1(P_j)$.
From this and
\eqref{eq:volume-closure-envelope}, we infer that
$r_{P_j}=\delta^{-1}r_S\leq C_{\rm r}r_K$ and
$\rho(x_{P_j},x_K)
 \leq r_{P_j}+\rho(x_S,x_K)\leq A_{\rm r}r_K.$
Thus, \eqref{eq:w_j} gives $w_j(K)=\frac{r_{P_j}}{r_K}$.
Using \eqref{eq:volume-closure-mark-decrement}, we obtain
$L_{j-1}(x_K)-L_j(x_K)\leq\kappa r_Kw_j(K)$ for any
$j\in\{t,\ldots,J\}$.
By summing over $j\in\{t,\ldots,J\}$ and using
$L_{t-1}(x_K)>2r_K$ and $L_J(x_K)\leq r_K$, we conclude that
\[
 r_K<L_{t-1}(x_K)-L_J(x_K)
 =\sum_{j=t}^J[L_{j-1}(x_K)-L_j(x_K)]
 \leq\kappa r_K\sum_{j=t}^Jw_j(K).
\]
Since $\kappa\in(0,1)$, it follows that
$\sum_{j=1}^Jw_j(K)\geq\sum_{j=t}^Jw_j(K)>1$,
which proves \eqref{eq:volume-closure-row} in this case.

Choose $m_0\in\mathbb Z_+$ such that
$C_{\rm r}\leq\delta^{-m_0}$.
Fix $j\in\{1,\ldots,J\}$.
By \eqref{eq:w_j} and \eqref{eq:dyadic-radii}, we find that any
$K\in\mathcal T_J$ with $w_j(K)\neq0$ satisfies
$r_K=\delta^{-m}r_{P_j}$ for a unique
$m\in\mathbb Z\cap[-m_0,\infty)$.
For fixed $m\in\mathbb Z\cap[-m_0,\infty)$,
\eqref{eq:w_j} gives
$\rho(x_K,x_{P_j})\leq A_{\rm r}\delta^{-m}r_{P_j}$
for any $K\in\mathcal T_J$ with $w_j(K)\neq0$ and
$r_K=\delta^{-m}r_{P_j}$.
From the pairwise disjoint inner balls in \eqref{eq:dyadic-ball}
and the geometric doubling condition, we deduce that
\[
 \sharp\left\{K\in\mathcal T_J:w_j(K)\neq0,\
 r_K=\delta^{-m}r_{P_j}\right\}\lesssim1.
\]
This further implies that
\[
 \sum_{K\in\mathcal T_J}w_j(K)
 =\sum_{m=-m_0}^{\infty}\delta^m
 \sharp\left\{K\in\mathcal T_J:w_j(K)\neq0,\
 r_K=\delta^{-m}r_{P_j}\right\}.
\]
Thus,
\begin{align}\label{eq:volume-closure-column}
 \sum_{K\in\mathcal T_J}w_j(K)
 \lesssim\sum_{m=-m_0}^{\infty}\delta^m\lesssim1.
\end{align}
For fixed $I\in\mathcal T_0$, \eqref{eq:w_0} gives
$r_K\in[\frac{r_I}{2},r_I]$ and
$\rho(x_K,x_I)\leq2r_I$ whenever $w_0(K,I)\neq0$.
Using this and the geometric doubling condition again, we obtain
\begin{align}\label{eq:volume-2}
 \sum_{K\in\mathcal T_J}w_0(K,I)\lesssim1,
\end{align}
where the implicit positive constants in
\eqref{eq:volume-closure-column} and \eqref{eq:volume-2}
depend only on the geometric constants of $X$.

When $\phi=V$, from \eqref{eq:w_j}, \eqref{eq:w_0},
\eqref{eq:dyadic-ball}, and the doubling condition, it follows that
$\mu(P_j)\lesssim\mu(K)$ when $w_j(K)\neq0$ and
$\mu(I)\lesssim\mu(K)$ when $w_0(K,I)\neq0$.
Since $1+p\gamma\leq0$, it follows that
$\mu(K)\phi(K)^{p\gamma}\lesssim\mu(P_j)\phi(P_j)^{p\gamma}$
and $\mu(K)\phi(K)^{p\gamma}\lesssim\mu(I)\phi(I)^{p\gamma}$,
respectively.
When $\phi=\rho$, using \eqref{eq:volume-upper-growth} and
\eqref{eq:w_j}, we obtain, if $w_j(K)\neq0$,
\[
 \frac{\mu(K)\phi(K)^{p\gamma}}
 {\mu(P_j)\phi(P_j)^{p\gamma}}
 \lesssim\left(\frac{r_K}{r_{P_j}}\right)^{D+p\gamma}
 \lesssim1,
\]
where the last inequality follows from $D+p\gamma\leq0$ and
$r_K\geq C_{\rm r}^{-1}r_{P_j}$.
From \eqref{eq:w_0} and the doubling condition, it also follows that,
if $w_0(K,I)\neq0$,
\[
 \frac{\mu(K)\phi(K)^{p\gamma}}
 {\mu(I)\phi(I)^{p\gamma}}
 \sim\left(\frac{r_K}{r_I}\right)^{p\gamma}\lesssim1.
\]
Using these and \eqref{eq:volume-closure-row}, we obtain,
for any $K\in\mathcal T_J$,
\begin{align*}
 \mu(K)\phi(K)^{p\gamma}
 &\leq\mu(K)\phi(K)^{p\gamma}
 \left[\sum_{I\in\mathcal T_0}w_0(K,I)+\sum_{j=1}^Jw_j(K)\right]\\
 &\lesssim
 \sum_{I\in\mathcal T_0}w_0(K,I)\mu(I)\phi(I)^{p\gamma}+\sum_{j=1}^Jw_j(K)\mu(P_j)\phi(P_j)^{p\gamma}.
\end{align*}
By summing over $K\in\mathcal T_J$, interchanging the finite sums,
and using \eqref{eq:volume-closure-column} and \eqref{eq:volume-2},
we conclude that \eqref{eq:volume-weighted-linear-closure} holds.
This completes the proof of Lemma~\ref{lem:volume-weighted-linear-closure}.
\end{proof}

\subsection{Poincar\'e Inequalities for Metric and Volume Functionals}
\label{subsec:volume-negative}

In this subsection, we use the estimates from
Subsection~\ref{subsec:dyadic-refinement} to prove
Proposition~\ref{prop:volume-negative-poincare}.
We first treat bounded functions and then use truncations to obtain
the inequality for arbitrary measurable functions.

\begin{lemma}\label{lem:volume-two-region}
Let $p$, $(X,\rho,\mu)$, $\phi$, and $\gamma$ be as in
Proposition~\ref{prop:volume-negative-poincare}, and let
$\kappa\in(0,\frac12)$. Then there exists $\Lambda_0\in[1,\infty)$
such that, for any ball $B$ with its radius $R\in(0,3\operatorname{diam} X)$ and
for any measurable function $g:X\to[0,t]$ with $t\in(0,\infty)$,
if
\begin{equation}\label{eq:volume-two-region-measure-assumption}
 \min\left\{\mu\left(B\cap\left\{g\leq\frac{t}{2}\right\}\right),
             \mu\left(B\cap\left\{g=t\right\}\right)\right\}
 \geq\kappa\mu(B),
\end{equation}
then
\begin{equation}\label{eq:volume-two-region}
 t^p\mu(B)
 \lesssim R^p\mu(B)\phi(B)^{p\gamma}
 +R^p\mathcal E^\phi_{1,\gamma,p}(g;\Lambda_0B),
\end{equation}
where $\Lambda_0$ and the implicit positive constant depend only on
$\kappa$, $p$, $\gamma$, and the geometric constants of $X$.
\end{lemma}

For the proof of Lemma~\ref{lem:volume-two-region}, we choose
$\mathcal T_0(B)$ and $U(B)$ so that any finite balanced leaf
partition of $U(B)$ satisfies the assumptions of
Lemma~\ref{lem:volume-balanced-cover}.
Let $C$ be the constant in Lemma~\ref{lem:volume-balanced-cover}.
For any ball $B$ with its
radius $R\in(0,3\operatorname{diam} X)$, choose $k_0\in\mathbb Z$ such that
\begin{equation}\label{eq:volume-fixed-root-scale}
 \frac{\delta\min\left\{R,\operatorname{diam} X\right\}}{10A_{\rm n}}
 \leq r_0:=C_1\delta^{k_0}
 <\frac{\min\left\{R,\operatorname{diam} X\right\}}{10A_{\rm n}},
\end{equation}
and define
\begin{equation}\label{eq:volume-fixed-roots}
 \mathcal T_0(B)
 :=\left\{L\in\mathcal D_{k_0}:
 L\cap(C+4A_{\rm n}+4)B\neq\emptyset\right\}
\end{equation}
and
\begin{equation}\label{eq:volume-fixed-root-union}
 U(B):=\bigcup_{L\in\mathcal T_0(B)}L.
\end{equation}
By \eqref{eq:volume-fixed-root-scale} and \eqref{eq:dyadic-ball},
we obtain
\begin{equation}\label{eq:volume-fixed-root-geometry}
 B\subset U(B)\subset(C+4A_{\rm n}+6)B.
\end{equation}
From \eqref{eq:volume-fixed-root-scale} and the doubling condition,
we deduce that, for any $L\in\mathcal T_0(B),$
\begin{equation}\label{eq:volume-fixed-root-measures}
 \mu(L)\sim\mu(B),
\end{equation}
where the implicit positive constant depends only on the geometric
constants of $X$.
By \eqref{eq:volume-fixed-root-scale}, \eqref{eq:volume-fixed-roots},
and the geometric doubling condition, we find that
$\mathcal T_0(B)$ is a finite balanced leaf partition of $U(B)$ with
$\sharp\mathcal T_0(B)\lesssim1$, where the implicit positive constant
depends only on the geometric constants of $X$.
For any finite balanced leaf partition $\mathcal T$ of $U(B)$,
by \eqref{eq:volume-fixed-root-scale}, \eqref{eq:volume-fixed-roots},
and \eqref{eq:volume-fixed-root-union}, we obtain $CB\subset U(B)$,
$r_P\leq r_0<R$ for any $P\in\mathcal T$, and
$A_{\rm n}B(P)\subset(C+2A_{\rm n})B\subset U(B)$ whenever
$P\in\mathcal T$ and $A_{\rm n}B(P)\cap CB\neq\emptyset$.
Thus, $\mathcal T$ satisfies the assumptions of
Lemma~\ref{lem:volume-balanced-cover}.

\begin{proof}[Proof of Lemma~\ref{lem:volume-two-region}]
Let $B$ be a ball with its radius $R\in(0,3\operatorname{diam} X)$, and let $t$
and $g$ satisfy the hypotheses of this lemma.
Let $\mathcal T_0:=\mathcal T_0(B)$ and $U:=U(B)$ be as in
\eqref{eq:volume-fixed-roots} and \eqref{eq:volume-fixed-root-union}.
It suffices to consider $\mathcal E^\phi_{1,\gamma,p}(g;U)<\infty$.
Fix $\varepsilon\in(0,\frac14)$, to be chosen below.
Let $A_0\in[1,\infty)$ be sufficiently large, depending only on
$\varepsilon$, $p$, $\gamma$, and the geometric constants of $X$.
Starting with $\mathcal T_0$, we construct the partitions
$\mathcal T_j$ as before Lemma~\ref{lem:volume-weighted-linear-closure},
with the following rule for choosing the marked leaves. In the
current partition $\mathcal T$, a leaf $P$ is marked if
\begin{equation}\label{eq:volume-internal-marking}
 m_P^\varepsilon(g)-m_P^{1-\varepsilon}(g)>A_0r_P\phi(P)^\gamma,
\end{equation}
or if there exists $P'\in\mathcal T$ with $P'\sim P$ such that
\begin{equation}\label{eq:volume-neighbor-marking}
 |a_P(g)-a_{P'}(g)|>A_0r_P\phi(P)^\gamma.
\end{equation}
The construction stops when no leaf satisfies either condition.
Let $J\in\mathbb Z_+$ be such that
$\mathcal T_0,\ldots,\mathcal T_J$ have been constructed, and fix
$j\in\{1,\ldots,J\}$. If $P_j$ satisfies
\eqref{eq:volume-internal-marking}, let
\[
 \mathcal C_j:=
 \left(P_j\cap\left\{g\leq m_{P_j}^{1-\varepsilon}(g)\right\}\right)
 \times
 \left(P_j\cap\left\{g\geq m_{P_j}^{\varepsilon}(g)\right\}\right).
\]
By \eqref{eq:volume-quantile-mass}, we obtain
$(\mu\times\mu)(\mathcal C_j)\geq\varepsilon^2\mu(P_j)^2.$
Otherwise, choose $P'_j\in\mathcal T_{j-1}$ with $P'_j\sim P_j$
satisfying \eqref{eq:volume-neighbor-marking}.
We may assume $a_{P_j}(g)<a_{P'_j}(g)$ because the other order is
handled similarly, and let
\[
 \mathcal C_j:=
 \left(P_j\cap\left\{g\leq a_{P_j}(g)\right\}\right)
 \times\left(P'_j\cap\left\{g\geq a_{P'_j}(g)\right\}\right).
\]
From \eqref{eq:volume-median-convention} with $\theta=\frac12$, it follows that
$\mu(P_j\cap\{g\leq a_{P_j}(g)\})\geq\frac12\mu(P_j)$ and
$\mu(P'_j\cap\{g\geq a_{P'_j}(g)\})\geq\frac12\mu(P'_j)$.
Using \eqref{eq:neighbor-r} and the doubling condition, we obtain
\[
 (\mu\times\mu)(\mathcal C_j)
 \geq\frac14\mu(P_j)\mu(P'_j)\gtrsim\mu(P_j)^2.
\]
From \eqref{eq:volume-internal-marking} in the first case and
\eqref{eq:volume-neighbor-marking} in the second case and
from the definitions
of $\mathcal C_j$ in either case, we infer that
\begin{equation}\label{eq:volume-two-region-difference}
 |g(x)-g(y)|>A_0r_{P_j}\phi(P_j)^\gamma,
 \quad (x,y)\in\mathcal C_j.
\end{equation}
By \eqref{eq:volume-reverse-growth} and the doubling condition, we obtain,
for any $\eta\in(0,1)$,
\[
 \int_{P_j}V(x,\eta r_{P_j})\,d\mu(x)
 \lesssim\eta^d\mu(P_j)^2.
\]
Choosing $\eta$ sufficiently small, depending only on $\varepsilon$
and the geometric constants of $X$, we find that
$\mathcal A_j:=\{(x,y)\in\mathcal C_j:
\rho(x,y)\geq\eta r_{P_j}\}\subset P_j\times U$ satisfies
\begin{equation}\label{eq:volume-two-region-witness}
 (\mu\times\mu)(\mathcal A_j)\gtrsim\mu(P_j)^2,
\end{equation}
\begin{equation}\label{eq:volume-two-region-distances}
 \eta r_{P_j}\leq\rho(x,y)\lesssim r_{P_j},
 \quad (x,y)\in\mathcal A_j,
\end{equation}
and
\begin{equation}\label{eq:volume-two-region-volumes}
 V(x,y)\sim\mu(P_j),\quad (x,y)\in\mathcal A_j.
\end{equation}
Using \eqref{eq:volume-two-region-distances},
\eqref{eq:volume-two-region-volumes}, and
\eqref{eq:volume-two-region-difference},
we conclude that, for $A_0$ sufficiently large,
\[
 |g(x)-g(y)|>\rho(x,y)\phi(x,y)^\gamma,
 \quad (x,y)\in\mathcal A_j.
\]
Let $x,y\in U$ be distinct.
If $(x,y)\in\mathcal A_j$, then $x\in P_j$ and, by
\eqref{eq:volume-two-region-distances}, $r_{P_j}\sim\rho(x,y)$.
Thus, only boundedly many generations of $P_j$ are possible.
Since the marked leaves $P_j$ are pairwise distinct and the cubes
of a fixed generation are pairwise disjoint, it follows that
$\sum_{j=1}^J\mathbf 1_{\mathcal A_j}(x,y)\lesssim1$.
Using \eqref{eq:energy}, \eqref{eq:volume-two-region-witness},
\eqref{eq:volume-two-region-distances}, and
\eqref{eq:volume-two-region-volumes}, we obtain
\begin{align}\label{eq:volume-two-region-marked-sum}
 \sum_{j=1}^J\mu(P_j)\phi(P_j)^{p\gamma}
 \lesssim\sum_{j=1}^J\iint_{\mathcal A_j}
 \frac{\phi(x,y)^{p\gamma}}{V(x,y)}\,d\mu(x)\,d\mu(y)
 \lesssim\mathcal E^\phi_{1,\gamma,p}(g;U)
\end{align}
for any such $J$, where the implicit positive constants depend only
on $\varepsilon$, $p$, $\gamma$, and the geometric constants of $X$.
For any marked leaf $P_j$, let $I\in\mathcal T_0$ be its initial
ancestor. We have
$\mu(P_j)\phi(P_j)^{p\gamma}\gtrsim\mu(I)\phi(I)^{p\gamma}$
if $\phi=V$, because $\mu(P_j)\leq\mu(I)$ and $1+p\gamma\leq0$.
When $\phi=\rho$, by \eqref{eq:volume-upper-growth}, we obtain
\[
 \frac{\mu(P_j)\phi(P_j)^{p\gamma}}
 {\mu(I)\phi(I)^{p\gamma}}
 \gtrsim\left(\frac{r_{P_j}}{r_I}\right)^{D+p\gamma}\geq1,
\]
where the last inequality follows from $r_{P_j}\leq r_I$ and
$D+p\gamma\leq0$.
From \eqref{eq:volume-fixed-root-scale} and
\eqref{eq:volume-fixed-root-measures}, it follows that $r_I=r_0\sim R$
and $\mu(I)\sim\mu(B)$ for any $I\in\mathcal T_0$, and hence
$\phi(I)\sim\phi(B)$.
Thus, $\mu(P_j)\phi(P_j)^{p\gamma}
\gtrsim\mu(B)\phi(B)^{p\gamma}$ for any $j\in\{1,\ldots,J\}$.
This, together with \eqref{eq:volume-two-region-marked-sum}, further
implies that
\[
 J\mu(B)\phi(B)^{p\gamma}
 \lesssim\sum_{j=1}^J\mu(P_j)\phi(P_j)^{p\gamma}
 \lesssim\mathcal E^\phi_{1,\gamma,p}(g;U)<\infty.
\]
Thus, the construction stops after finitely many steps, at a
balanced partition $\mathcal T$.
By \eqref{eq:volume-fixed-root-scale},
\eqref{eq:volume-fixed-root-measures}, and
$\sharp\mathcal T_0\lesssim1$, we obtain
\[
 \sum_{I\in\mathcal T_0}\mu(I)\phi(I)^{p\gamma}
 \lesssim\mu(B)\phi(B)^{p\gamma}.
\]
This, together with Lemma~\ref{lem:volume-weighted-linear-closure}
and \eqref{eq:volume-two-region-marked-sum}, further implies that
\begin{align}\label{eq:volume-two-region-terminal-weight}
 \sum_{P\in\mathcal T}\mu(P)\phi(P)^{p\gamma}
 &\lesssim\sum_{I\in\mathcal T_0}\mu(I)\phi(I)^{p\gamma}
 +\sum_{j=1}^J\mu(P_j)\phi(P_j)^{p\gamma}\notag\\
 &\lesssim\mu(B)\phi(B)^{p\gamma}
 +\mathcal E^\phi_{1,\gamma,p}(g;U).
\end{align}
For any $P\in\mathcal T$, by the failure of
\eqref{eq:volume-internal-marking} and
$m_P^{1-\varepsilon}(g)\leq a_P(g)\leq m_P^\varepsilon(g)$, we obtain
$|g-a_P(g)|\leq A_0r_P\phi(P)^\gamma$ on
$P\cap\{m_P^{1-\varepsilon}(g)\leq g\leq m_P^\varepsilon(g)\}$.
Since $g,a_P(g)\in[0,t]$, it follows that $|g-a_P(g)|\leq t$
elsewhere in $P$.
From these and \eqref{eq:volume-quantile-concentration}, we infer that
\begin{equation}\label{eq:volume-two-region-residual}
 \int_P|g-a_P(g)|^p\,d\mu
 \leq A_0^pr_P^p\mu(P)\phi(P)^{p\gamma}+2\varepsilon t^p\mu(P).
\end{equation}
Using Lemma~\ref{lem:volume-balanced-cover} with coefficients
$a_P(g)$ and the failure of \eqref{eq:volume-neighbor-marking},
we obtain
\begin{equation}\label{eq:volume-two-region-median-oscillation}
 \inf_{c\in\mathbb R}
 \sum_{\genfrac{}{}{0pt}{}{P\in\mathcal T}{P\cap B\neq\emptyset}}
 \mu(P)|a_P(g)-c|^p
 \lesssim A_0^pR^p\sum_{P\in\mathcal T}\mu(P)\phi(P)^{p\gamma}.
\end{equation}
From \eqref{eq:volume-two-region-residual},
\eqref{eq:volume-two-region-median-oscillation}, and
\eqref{eq:volume-two-region-terminal-weight}, we deduce that
\begin{align}\label{eq:volume-two-region-before-absorption}
 \inf_{c\in\mathbb R}\int_B|g-c|^p\,d\mu
 &\leq2^{p-1}\sum_{P\in\mathcal T}
 \int_P|g-a_P(g)|^p\,d\mu\notag\\
 &\quad+2^{p-1}\inf_{c\in\mathbb R}
 \sum_{\genfrac{}{}{0pt}{}{P\in\mathcal T}{P\cap B\neq\emptyset}}
 \mu(P)|a_P(g)-c|^p\notag\\
 &\leq C_\varepsilon R^p
 \left[\mu(B)\phi(B)^{p\gamma}+\mathcal E^\phi_{1,\gamma,p}(g;U)\right]
 +2^p\varepsilon t^p\mu(U),
\end{align}
where $C_\varepsilon\in(0,\infty)$ depends only on $\varepsilon$,
$p$, $\gamma$, and the geometric constants of $X$.
For any $c\in\mathbb R$, we have $|g-c|\geq\frac{t}{4}$ on
$B\cap\{g=t\}$ when $c\in(-\infty,\frac{3t}{4}]$, and on
$B\cap\{g\leq\frac{t}{2}\}$ when $c\in(\frac{3t}{4},\infty)$.
Using \eqref{eq:volume-two-region-measure-assumption} and taking the
infimum over $c$, we conclude that
\begin{equation}\label{eq:volume-two-region-lower-bound}
 \inf_{c\in\mathbb R}\int_B|g-c|^p\,d\mu
 \geq\frac{\kappa}{4^p}t^p\mu(B).
\end{equation}
We choose $\varepsilon\in(0,\frac14)$, depending only on
$\kappa$, $p$, and the geometric constants of $X$, such that
$2^p\varepsilon\mu(U)\leq\frac{\kappa}{2\cdot4^p}\mu(B)$,
which is possible by \eqref{eq:volume-fixed-root-geometry} and the doubling condition.
From \eqref{eq:volume-two-region-before-absorption} and
\eqref{eq:volume-two-region-lower-bound}, we infer that \eqref{eq:volume-two-region} holds, with
$\Lambda_0:=C+4A_{\rm n}+6$ as in
\eqref{eq:volume-fixed-root-geometry}.
This completes the proof of Lemma~\ref{lem:volume-two-region}.
\end{proof}

The following lemma uses Lemma~\ref{lem:volume-two-region} to
estimate $t^p\mu(U\cap\{g=t\})$ without a fixed positive lower bound
for $\frac{\mu(U\cap\{g=t\})}{\mu(U)}$.

\begin{lemma}\label{lem:volume-tail}
Let $p$, $(X,\rho,\mu)$, $\phi$, and $\gamma$ be as in
Proposition~\ref{prop:volume-negative-poincare}. Let $B$ be a ball
with its radius $r\in(0,\operatorname{diam} X)$, and let $U:=U(B)$ be as in
\eqref{eq:volume-fixed-root-union}.
Then, for any $\vartheta\in(0,\frac{1}{4})$, there exist
$A,C_\vartheta\in(0,\infty)$ and $\Lambda\in[1,\infty)$ such
that, for any $t\in(Ar\phi(B)^\gamma,\infty)$ and any measurable
function $g:X\to[0,t]$ satisfying
\begin{equation}\label{eq:volume-tail-zero}
 \mu\left(U\cap\left\{g=0\right\}\right)\geq\frac{\mu(U)}{2},
\end{equation}
one has
\begin{align}\label{eq:volume-tail-conclusion}
 t^p\mu\left(U\cap\left\{g=t\right\}\right)
 \leq C_\vartheta r^p\mathcal E^\phi_{1,\gamma,p}(g;\Lambda B)+\frac{4\vartheta}{3}t^p
 \mu\left(U\cap\left\{g>\frac{t}{2}\right\}\right),
\end{align}
where $A$, $C_\vartheta$, and $\Lambda$ depend only on
$\vartheta$, $p$, $\gamma$, and the geometric constants of $X$.
\end{lemma}

\begin{proof}
Let $\mathcal T_0:=\mathcal T_0(B)$ be as in
\eqref{eq:volume-fixed-roots}, and let
\begin{equation*}
 T:=U\cap\left\{g=t\right\}\ \text{and}\
 S:=U\cap\left\{g>\frac{t}{2}\right\}.
\end{equation*}
Let $\Lambda_0$ be the dilation chosen in the proof of
Lemma~\ref{lem:volume-two-region}, and let
$\Lambda:=(C+4A_{\rm n}+7)\Lambda_0$.
By \eqref{eq:volume-fixed-root-geometry} and
\eqref{eq:volume-fixed-root-scale}, we obtain
$\Lambda_0(C+4A_{\rm n}+6)B\subset\Lambda B$ and
$\Lambda_0B(Q)\subset\Lambda B$ for any $Q\in\mathcal D$
with $Q\subset I\in\mathcal T_0$.
We may assume $\mathcal E^\phi_{1,\gamma,p}(g;\Lambda B)<\infty$.
Let $c\in(0,1)$ be sufficiently small, depending only on
$\vartheta$, $p$, $\gamma$, and the geometric constants of $X$.
From \eqref{eq:volume-fixed-root-scale} and
\eqref{eq:volume-fixed-root-measures}, it follows that
$r_I\phi(I)^\gamma\lesssim r\phi(B)^\gamma$ for any $I\in\mathcal T_0$.
We choose $A$ sufficiently large so that $r_I\phi(I)^\gamma\leq ct$ for any
$I\in\mathcal T_0$.

Let $B_*$ have the same center as $B$ and radius
$\min\{(C+4A_{\rm n}+6)r,2\operatorname{diam} X\}$.
Thus, $U\subset B_*$, $r_{B_*}\sim r$, and
$\mu(B_*)\sim\mu(B)$, and Lemma~\ref{lem:volume-two-region}
can be used for $B_*$.
If some $I\in\mathcal T_0$ satisfies
$\mu(T\cap I)\in(\vartheta\mu(I),\mu(I)]$, then, by
\eqref{eq:volume-tail-zero}, \eqref{eq:volume-fixed-root-geometry},
\eqref{eq:volume-fixed-root-measures}, and the doubling condition,
we obtain
\[
 \mu\left(B_*\cap\left\{g\leq\frac{t}{2}\right\}\right)
 \geq\mu\left(U\cap\left\{g=0\right\}\right)
 \gtrsim\mu\left(B_*\right)
\]
and
\[
 \mu\left(B_*\cap\left\{g=t\right\}\right)
 \geq\mu(T\cap I)
 \gtrsim\vartheta\mu\left(B_*\right).
\]
Applying Lemma~\ref{lem:volume-two-region} with the ball
$B_*$, we find that
\[
 t^p\mu(B)\lesssim r^p\mu(B)\phi(B)^{p\gamma}
 +r^p\mathcal E^\phi_{1,\gamma,p}(g;\Lambda B).
\]
Since $t\in(Ar\phi(B)^\gamma,\infty)$, it follows that
$
 r^p\mu(B)\phi(B)^{p\gamma}
 \leq\frac{1}{A^p}t^p\mu(B).
$
By choosing $A$ sufficiently large, we absorb
$r^p\mu(B)\phi(B)^{p\gamma}$ into $t^p\mu(B)$ and obtain
\[
 t^p\mu(B)\lesssim r^p\mathcal E^\phi_{1,\gamma,p}(g;\Lambda B).
\]
From \eqref{eq:volume-fixed-root-geometry} and the doubling condition,
we infer that $\mu(T)\lesssim\mu(B)$.
Thus, $t^p\mu(T)\lesssim r^p\mathcal E^\phi_{1,\gamma,p}(g;\Lambda B)$,
which proves \eqref{eq:volume-tail-conclusion} in this case.
Thus, we only need to consider $\mu(T\cap I)\leq\vartheta\mu(I)$ for any
$I\in\mathcal T_0$.

Let $\mathcal P$ be the family of maximal cubes
$P\in\bigcup_{I\in\mathcal T_0}\mathcal D(I)$
satisfying $\mu(T\cap P)\in(\vartheta\mu(P),\mu(P)]$.
Here, $\mathcal D(I)$ is as in \eqref{eq:dyadic-descendants},
and maximality is with respect to inclusion.
These cubes are pairwise disjoint, and, by the Lebesgue
differentiation theorem and \eqref{eq:dyadic-ball}, we obtain
\begin{equation}\label{eq:volume-tail-cover}
 \mu\left(T\setminus\bigcup_{P\in\mathcal P}P\right)=0.
\end{equation}
Any proper ancestor $R$ of $P\in\mathcal P$ within its initial
cube satisfies $\mu(T\cap R)\leq\vartheta\mu(R)$ by maximality.
Let $\mathcal R$ consist of the cubes $R\in\mathcal D$
that are maximal with respect to inclusion among those satisfying
$\mu(S\cap R)\in(\frac34\mu(R),\mu(R)]$ and
$P\subsetneqq R\subset I$ for some $P\in\mathcal P$ and
$I\in\mathcal T_0$.
Since these cubes are pairwise disjoint, it follows that
\begin{align}\label{eq:volume-tail-dense-ancestors}
 t^p\mu\left(T\cap\bigcup_{R\in\mathcal R}R\right)
 &=t^p\sum_{R\in\mathcal R}\mu(T\cap R)\notag\\
 &\leq\frac{4\vartheta}{3}t^p
 \sum_{R\in\mathcal R}\mu(S\cap R)
 \leq\frac{4\vartheta}{3}t^p\mu(S).
\end{align}

Fix $P\in\mathcal P$ not contained in any $R\in\mathcal R$,
and let $Q$ be its parent. Since no initial cube belongs to
$\mathcal P$ in this case, $P$ is a proper descendant of its initial
cube $I$, and $Q\subset I$. By the definition of $\mathcal R$,
we obtain $\mu(Q\setminus S)\geq\frac{\mu(Q)}{4}$.
Using \eqref{eq:dyadic-ball} and the doubling condition, we find that
$\mu(P)\sim\mu(Q)\sim\mu(B(Q))$, and hence
\[
 \mu\left(B(Q)\cap\left\{g\leq\frac{t}{2}\right\}\right)
 \geq\mu(Q\setminus S)\gtrsim\mu(B(Q))
\]
and
\[
 \mu\left(B(Q)\cap\left\{g=t\right\}\right)
 \geq\mu(T\cap P)\gtrsim\vartheta\mu(B(Q)).
\]
Using Lemma~\ref{lem:volume-two-region} with $B(Q)$, we obtain
\[
 t^p\mu(Q)\lesssim r_Q^p\mu(Q)\phi(Q)^{p\gamma}
 +r_Q^p\mathcal E^\phi_{1,\gamma,p}(g;\Lambda_0B(Q)).
\]

If $r_Q\phi(Q)^\gamma\leq ct$, then
\[
 r_Q^p\mu(Q)\phi(Q)^{p\gamma}
 =\left[r_Q\phi(Q)^\gamma\right]^p\mu(Q)
 \leq c^pt^p\mu(Q).
\]
By choosing $c$ sufficiently small, we conclude that
$t^p\mu(Q)\lesssim r_Q^p\mathcal E^\phi_{1,\gamma,p}(g;\Lambda_0B(Q))$.
Let $\mathcal Q$ be the family of cubes $Q\in\mathcal D$ such
that $r_Q\phi(Q)^\gamma\leq ct$ and there exists a child
$P\in\mathcal P$ satisfying $P\not\subset R$ for any
$R\in\mathcal R$.
Using the bounded overlap of the balls $\Lambda_0B(Q)$
at a fixed generation and Tonelli's theorem, we obtain, for any
$j\in\mathbb Z_+$,
\[
 \sum_{I\in\mathcal T_0}
 \sum_{Q\in\mathcal Q\cap\mathcal D_j(I)}
 \mathcal E^\phi_{1,\gamma,p}(g;\Lambda_0B(Q))
 \lesssim\mathcal E^\phi_{1,\gamma,p}(g;\Lambda B).
\]
This, together with $r_Q\lesssim\delta^jr$ for
$Q\in\mathcal D_j(I)$ and $I\in\mathcal T_0$, further implies that
\begin{align}\label{eq:volume-tail-small-threshold}
 t^p\mu\left(T\cap\bigcup_{Q\in\mathcal Q}Q\right)
 &\leq t^p\sum_{Q\in\mathcal Q}\mu(Q)
 \lesssim\sum_{Q\in\mathcal Q}r_Q^p
 \mathcal E^\phi_{1,\gamma,p}(g;\Lambda_0B(Q))\notag\\
 &\lesssim\sum_{j=0}^{\infty}\delta^{jp}r^p
 \mathcal E^\phi_{1,\gamma,p}(g;\Lambda B)
 \lesssim r^p\mathcal E^\phi_{1,\gamma,p}(g;\Lambda B).
\end{align}

It remains to consider $P\in\mathcal P$ not contained in any
$R\in\mathcal R$, with its parent $Q$ satisfying
$r_Q\phi(Q)^\gamma\in(ct,\infty)$.
Let $I\in\mathcal T_0$ satisfy $Q\subset I$.
Using $r_I\phi(I)^\gamma\leq ct$, \eqref{eq:dyadic-radii}, and
the doubling condition, we obtain a cube $R\in\mathcal D$,
minimal with respect to inclusion among those satisfying
$Q\subsetneqq R\subset I$ and $r_R\phi(R)^\gamma\leq ct$,
such that $r_R\phi(R)^\gamma\sim ct\sim t$.
Indeed, let $R'$ be the child of $R$ containing $Q$.
From the minimality of $R$, it follows that
$r_{R'}\phi(R')^\gamma\in(ct,\infty)$ when $R'\neq Q$;
when $R'=Q$, this follows from the assumption
$r_Q\phi(Q)^\gamma\in(ct,\infty)$.
By the doubling condition and $r_{R'}=\delta r_R$, we find that
$r_R\phi(R)^\gamma\sim r_{R'}\phi(R')^\gamma$.
Since $P$ is not contained in any cube of $\mathcal R$, it follows
that $\mu(R\setminus S)\geq\frac{\mu(R)}{4}$.
Fix $x\in T\cap P$. By \eqref{eq:volume-reverse-growth},
\eqref{eq:dyadic-ball}, and the doubling condition, we obtain
\[
 \mu\left((R\setminus S)\setminus B(x,\eta r_R)\right)
 \gtrsim\mu(R),
\]
where $\eta\in(0,1)$ depends only on the geometric constants of $X$.
For any $y\in(R\setminus S)\setminus B(x,\eta r_R)$, we have
$\eta r_R\leq\rho(x,y)\lesssim r_R$,
$V(x,y)\sim\mu(R)$, $\phi(x,y)\sim\phi(R)$, and
$|g(x)-g(y)|\geq\frac{t}{2}$.
By choosing $c$ sufficiently small and using
$r_R\phi(R)^\gamma\leq ct$, we obtain
$|g(x)-g(y)|>\rho(x,y)\phi(x,y)^\gamma$.
We use $G_g^\phi$ from \eqref{eq:G} with $f=g$, $\lambda=1$, and
$E=\Lambda B$. From $r_R\phi(R)^\gamma\sim t$ and
\eqref{eq:volume-fixed-root-scale}, we deduce that
\begin{align}\label{eq:volume-tail-density}
 [G_g^\phi(x)]^p
 \geq\int_{(R\setminus S)\setminus B(x,\eta r_R)}
 \frac{\phi(x,y)^{p\gamma}}{V(x,y)}\,d\mu(y)
 \gtrsim\phi(R)^{p\gamma}
 \sim\frac{t^p}{r_R^p}\gtrsim\frac{t^p}{r^p}.
\end{align}
By \eqref{eq:volume-tail-cover}, we find that $\mu$-almost every
$x\in T\setminus(\bigcup_{R\in\mathcal R}R\cup\bigcup_{Q\in\mathcal Q}Q)$
belongs to some $P\in\mathcal P$.
Since $x\notin\bigcup_{R\in\mathcal R}R$, it follows that
$P\not\subset R$ for any $R\in\mathcal R$.
Let $Q$ be the parent of $P$. If $r_Q\phi(Q)^\gamma\leq ct$, then
$Q\in\mathcal Q$ and $x\in P\subset Q$, contradicting
$x\notin\bigcup_{Q\in\mathcal Q}Q$.
Thus, $r_Q\phi(Q)^\gamma\in(ct,\infty)$ and
\eqref{eq:volume-tail-density} holds at $x$.
From \eqref{eq:volume-tail-density} and \eqref{eq:G-mass}, we infer that
\begin{align*}
 t^p\mu\left(T\setminus
 \left[\bigcup_{R\in\mathcal R}R\cup\bigcup_{Q\in\mathcal Q}Q\right]\right)
 \lesssim r^p\int_{\Lambda B}[G_g^\phi(x)]^p\,d\mu(x)=r^p\mathcal E^\phi_{1,\gamma,p}(g;\Lambda B).
\end{align*}
This, together with \eqref{eq:volume-tail-dense-ancestors} and
\eqref{eq:volume-tail-small-threshold}, further implies that
\eqref{eq:volume-tail-conclusion} holds.
This completes the proof of Lemma~\ref{lem:volume-tail}.
\end{proof}

Now, we are in a position to prove Proposition \ref{prop:volume-negative-poincare}.

\begin{proof}[Proof of Proposition~\ref{prop:volume-negative-poincare}]
Fix a ball $B$ with radius $r\in(0,\operatorname{diam} X)$, and let
$U:=U(B)$ be as in \eqref{eq:volume-fixed-root-union}.
First, let $h:X\to[0,\infty)$ be a bounded measurable function
satisfying
\[
 \mu(U\cap\{h=0\})\geq\frac{\mu(U)}{2}.
\]
For any $k\in\mathbb Z$, define
\[
 h_k:=\min\left\{(h-2^k)_+,2^k,(5\cdot2^k-h)_+\right\}.
\]
Since $\{h_k=2^k\}=\{h\in[2^{k+1},2^{k+2}]\}$ and
$\{h_k>2^{k-1}\}=\{h\in(\frac{3}{2}\,2^k,\frac{9}{2}\,2^k)\}$,
it follows that
\begin{equation}\label{eq:volume-tail-band-comparison}
 \sum_{k\in\mathbb Z}2^{kp}\mu(U\cap\{h_k=2^k\})
 \geq4^{-p}\int_Uh^p\,d\mu,
\end{equation}
and
\begin{equation}\label{eq:volume-tail-band-S-sum}
 \sum_{k\in\mathbb Z}2^{kp}\mu(U\cap\{h_k>2^{k-1}\})
 \leq2\int_Uh^p\,d\mu.
\end{equation}
Let $\vartheta\in(0,\frac{1}{4})$ satisfy
$\frac{8\vartheta}{3}\leq\frac{1}{2\cdot4^p}$, and let $A$, $C_\vartheta$,
and $\Lambda$ be as in Lemma~\ref{lem:volume-tail}.
Note that, for any $k\in\mathbb Z$ and $x,y\in X$,
\[
 |h_k(x)-h_k(y)|\leq|h(x)-h(y)|.
\]
Since $h_k(x)\ne0$ precisely when
$2^k\in(\frac{h(x)}{5},h(x))$, it follows that
$\sharp\{k\in\mathbb Z:h_k(x)\ne0\}\leq3$ for any $x\in X$.
Thus, for any $x,y\in X$,
 \begin{align*}
 \sharp\left\{k\in\mathbb Z:h_k(x)\ne h_k(y)\right\}
 \leq\sharp\left\{k\in\mathbb Z:h_k(x)\ne0\right\}
 +\sharp\left\{k\in\mathbb Z:h_k(y)\ne0\right\}\leq6.
 \end{align*}
By this and $|h_k(x)-h_k(y)|\leq|h(x)-h(y)|$, we obtain
\begin{equation}\label{eq:volume-tail-band-energy}
 \sum_{k\in\mathbb Z}\mathcal E^\phi_{1,\gamma,p}(h_k;\Lambda B)
 \leq6\mathcal E^\phi_{1,\gamma,p}(h;\Lambda B).
\end{equation}
Since $\{h=0\}\subset\{h_k=0\}$, it follows that
$\mu(U\cap\{h_k=0\})\geq\frac{\mu(U)}{2}$.
Using Lemma~\ref{lem:volume-tail} for $h_k$ with
$2^k\in(Ar\phi(B)^\gamma,\infty)$, together with
\eqref{eq:volume-tail-band-comparison},
\eqref{eq:volume-tail-band-S-sum}, and
\eqref{eq:volume-tail-band-energy}, we obtain
 \begin{align*}
 4^{-p}\int_Uh^p\,d\mu
 \leq\mu(U)
 \sum_{\genfrac{}{}{0pt}{}{k\in\mathbb Z}{2^k\leq Ar\phi(B)^\gamma}}2^{kp}
 +6C_\vartheta r^p\mathcal E^\phi_{1,\gamma,p}(h;\Lambda B)
 +\frac{8\vartheta}{3}\int_Uh^p\,d\mu.
 \end{align*}
From \eqref{eq:volume-fixed-root-geometry} and the doubling condition, we infer that
\[
 \mu(U)
 \sum_{\genfrac{}{}{0pt}{}{k\in\mathbb Z}{2^k\leq Ar\phi(B)^\gamma}}2^{kp}
 \lesssim r^p\mu(B)\phi(B)^{p\gamma}.
\]
Since $\int_Uh^p\,d\mu<\infty$ and
$\frac{8\vartheta}{3}\leq\frac{1}{2\cdot4^p}$, it follows from
these estimates that
\begin{equation}\label{eq:volume-tail-bounded-Lp}
 \int_Uh^p\,d\mu
 \lesssim r^p\mu(B)\phi(B)^{p\gamma}
 +r^p\mathcal E^\phi_{1,\gamma,p}(h;\Lambda B),
\end{equation}
where the implicit positive constant depends only on $p$, $\gamma$,
and the geometric constants of $X$.

By \eqref{eq:volume-median-convention}, we find that both
$\min\{(f-a_U(f))_+,M\}$ and $\min\{(a_U(f)-f)_+,M\}$ satisfy
the hypotheses imposed on $h$ for any $M\in(0,\infty)$.
From this,
\eqref{eq:volume-tail-bounded-Lp}, and Fatou's lemma as $M\to\infty$,
we deduce that
\[
 \int_U|f-a_U(f)|^p\,d\mu
 \lesssim r^p\mu(B)\phi(B)^{p\gamma}
 +r^p\mathcal E^\phi_{1,\gamma,p}(f;\Lambda B).
\]
This, together with $B\subset U$, further implies that
\[
 \fint_B\fint_B|f(x)-f(y)|^p\,d\mu(x)\,d\mu(y)
 \lesssim r^p\phi(B)^{p\gamma}
 +\frac{r^p}{\mu(B)}\mathcal E^\phi_{1,\gamma,p}(f;\Lambda B).
\]
This completes the proof of Proposition~\ref{prop:volume-negative-poincare}.
\end{proof}

\begin{proof}[Proof of Theorem~\ref{thm:NPI}]
Using Lemma~\ref{lem:curve-local-estimates}(i) with $f$ replaced by
$\frac f\lambda$ and using $s\leq a+(s-a)_+$ for $s,a\in[0,\infty)$,
we obtain \eqref{eq:NPI} for $\gamma\in(-\infty,-\frac qp]$
if $\phi=\rho$, and for $\gamma\in(-\infty,-\frac{q}{dp}]$
if $\phi=V$.
Applying Proposition~\ref{prop:volume-negative-poincare} with $f$
replaced by $\frac f\lambda$ and multiplying its estimate by
$\lambda^p$, we conclude that \eqref{eq:NPI} holds for
$\gamma\in(-\infty,-\frac{D}{p}]$ if $\phi=\rho$, and for
$\gamma\in(-\infty,-\frac1p]$ if $\phi=V$.
The unions of these ranges are precisely the ranges stated in
Theorem~\ref{thm:NPI}.
This completes the proof of Theorem~\ref{thm:NPI}.
\end{proof}

\begin{proof}[Proof of Theorem~\ref{thm:asymptotic}]
When $\gamma\in(-\infty,0)$, Theorem~\ref{thm:NPI} gives
\eqref{eq:NPI} throughout the stated parameter ranges.
Applying Proposition~\ref{prop:curve-asymptotic}(i), we obtain
\eqref{eq:asymptotic-lower}.
When $\gamma\in(0,\infty)$, Proposition~\ref{prop:curve-asymptotic}(ii)
gives
\eqref{eq:asymptotic-lower}.
This completes the proof of Theorem~\ref{thm:asymptotic}.
\end{proof}

\section{Sharpness of the Parameter Ranges}\label{sec:optimality}

In this section, we construct counterexamples showing that the
negative parameter ranges in Theorem~\ref{thm:asymptotic} are sharp
for given lower and upper dimensions.
The sharpness of the ranges in Theorem~\ref{thm:NPI} then follows
from Proposition~\ref{prop:curve-asymptotic}(i).

\begin{proposition}\label{prop:counterexample}

Let $p\in[1,\infty)$, $q\in[1,p]$, $D\in[1,\infty)$, and
$d\in(0,D]$. Then the following assertions hold.

\begin{enumerate}[{\rm (i)}]
\item
For any $\gamma\in(-\frac{\min\{q,D\}}{p},0)$, there exist
a compact $q$-Poincar\'e space $(X,\rho,\mu)$, with lower dimension
$d$ and upper dimension $D$, and $f\in L^1(\mu)$ such that
$\mathcal E_p(f;X)\in(0,\infty]$ and, with $\phi=\rho$,
\begin{equation}\label{eq:energy-zero}
 \lim_{\lambda\to0^+}\mathcal E^\phi_{\lambda,\gamma,p}(f;X)=0.
\end{equation}

\item
For any $\gamma\in(-\frac{q}{p\max\{q,d\}},0)$, there exist
a compact $q$-Poincar\'e space $(X,\rho,\mu)$, with lower dimension
$d$ and upper dimension $D$, and $f\in L^1(\mu)$ such that
$\mathcal E_p(f;X)\in(0,\infty]$ and \eqref{eq:energy-zero} holds
with $\phi=V$.
\end{enumerate}
\end{proposition}

\begin{corollary}\label{cor:NPI-sharpness}
Under the parameter conditions of either assertion in
Proposition~\ref{prop:counterexample}, there exists a compact
$q$-Poincar\'e space $(X,\rho,\mu)$ with lower dimension $d$ and
upper dimension $D$ on which \eqref{eq:NPI} fails.
\end{corollary}

\begin{proof}
Let $(X,\rho,\mu)$, $f\in L^1(\mu)$, and $\phi$ be as in the
corresponding assertion of Proposition~\ref{prop:counterexample}.
If \eqref{eq:NPI} held on $X$, then, by
Proposition~\ref{prop:curve-asymptotic}(i) and \eqref{eq:energy-zero},
we would obtain $\mathcal E_p(f;X)=0$, which contradicts
$\mathcal E_p(f;X)\in(0,\infty]$.
This completes the proof of Corollary~\ref{cor:NPI-sharpness}.
\end{proof}

Now, we show that Proposition \ref{prop:counterexample} holds.
\begin{proof}[Proof of Proposition~\ref{prop:counterexample}]
We first consider (i) when $D=1$ and (ii) when $d\in(0,1]$.
In both cases, $\gamma\in(-\frac1p,0)$.
Let $(X,\rho,\mu):=([-1,1],|\cdot-\cdot|,\,d x)$ and
$f:=\mathbf 1_{(0,1]}$.
This compact space supports a weak $(1,1)$-Poincar\'e inequality.
Since $V(x,r)\sim r$
for any $x\in[-1,1]$ and $r\in(0,2)$, it has lower dimension $d$
and upper dimension $D$ under the stated parameter conditions.
For $\phi\in\{\rho,V\}$, we obtain
\[
 \mathcal E^\phi_{\lambda,\gamma,p}(f;X)
 \lesssim\lambda^p\int_0^1\int_0^1
 (u+v)^{p\gamma-1}\,d u\,d v
 \lesssim\lambda^p\int_0^2r^{p\gamma}\,d r
 \lesssim\lambda^p,
\]
where the implicit positive constant depends only on $p$ and $\gamma$.
Thus, \eqref{eq:energy-zero} holds. When $p\in(1,\infty)$,
suppose that $f$ admits a Haj\l asz upper gradient $g\in L^p(\mu)$,
and let $r\in(0,\frac12)$. By \eqref{eq:H-gradient}, we obtain
$1\leq4r[g(x)+g(y)]$ for almost every
$(x,y)\in(r,2r)\times(-2r,-r)$.
By integrating the $p$-th power and using that both intervals
have measure $r$, we find that
$r^{1-p}\lesssim\int_{\{r<|x|<2r\}}g(x)^p\,d x$,
where the implicit positive constant is independent of $r$.
Letting $r\to0^+$ contradicts $g\in L^p(\mu)$.
Thus, $f\notin\dot M^{1,p}(\mu)$.
Finally, when $p=1$, it is clear that $|Df|(X)=1$.
Therefore, $\mathcal E_p(f;X)\in(0,\infty]$.

It remains to consider (i) when $D\in(1,\infty)$ and (ii) when
$d\in(1,\infty)$. In the case (i), we take
$N:=D$ and $\kappa\in(-p\gamma,\min\{q,D\})$.
In the case (ii), we take
$N:=d$ and $\kappa\in(-dp\gamma,\min\{q,d\})$.
The ranges of $\gamma$ ensure that these choices are possible.
In both cases, $N\in(1,\infty)$, $d\leq N\leq D$, and
$\kappa\in(0,\min\{q,N\})$.
Let $s:=N-\kappa\in(0,N)$.
The following gluing construction borrows some ideas from
\cite[Remark~6.19(b)]{HeinonenKoskela98}.
By \cite[Theorem~2.6]{Laakso00}, we obtain a compact Ahlfors
$N$-regular space $(Y,\rho_Y,\nu)$ admitting pencils of curves
with uniform constants, which, combined with
\cite[Theorem~1.4]{FasslerOrponen19} (see also \cite{Semmes96}),
further implies that $Y$ supports
a weak $(1,1)$-Poincar\'e inequality.
Using \cite[Theorem~5.1]{MattilaSaaranen09}, we obtain a closed
Ahlfors $s$-regular subset $A\subset Y$ with
$\operatorname{diam} A\sim\operatorname{diam} Y$.

Let $Y_+$ and $Y_-$ be two disjoint copies of $Y$, and, for
$x\in Y$, denote its copies in $Y_+$ and $Y_-$ by $x_+$ and
$x_-$, respectively.
Let $X:=Y_+\cup_A Y_-$ be obtained by gluing the two copies along
$A$; that is, $a_+$ and $a_-$ are identified as one point for any
$a\in A$, while $x_+$ and $x_-$ remain distinct for any
$x\in Y\setminus A$. We regard $Y_+$ and $Y_-$ as subsets of $X$.
The metric $\rho$ on $X$ agrees with $\rho_Y$ on $Y_+$ and $Y_-$ and satisfies
\begin{equation}\label{eq:glued-distance}
 \rho(x_+,y_-)
 :=\inf_{a\in A}\left[\rho_Y(x,a)+\rho_Y(a,y)\right].
\end{equation}
We define the Borel measure $\mu$ on $X$ by
\[
 \mu(E):=
 \nu\left(\left\{y\in Y:y_+\in E\right\}\right)
 +\nu\left(\left\{y\in Y:y_-\in E\right\}\right)
\]
for any Borel set $E\subset X$.
Since $s\in(0,N)$, it follows that $\nu(A)=0$.
Clearly, $X$ is compact.
By \cite[Section~6.14, p.~43]{HeinonenKoskela98}, we find that
$(X,\rho,\mu)$ is Ahlfors $N$-regular.
By rescaling $\rho_Y$ and $\rho$, we may assume
$\operatorname{diam} X=1$, so that
\begin{equation}\label{eq:glued-volume}
 V(x,r)\sim r^N
\end{equation}
for any $x\in X$ and $r\in(0,1]$.
Since $d\leq N\leq D$, it follows that $\mu$ has lower dimension
$d$ and upper dimension $D$.
Let $\eta$ be an $s$-regular Borel measure on $A$, extended by zero
to $X$; see \cite[Definition~1.1]{MattilaSaaranen09}.
Fix $a\in A$ and $r\in(0,\operatorname{diam} Y)$.
For any countable ball cover
$\{B(x_i,r_i)\}_{i=1}^{\infty}$ of $A\cap B(a,r)$, the
$s$-regularity of $\eta$ and $\operatorname{diam} A\sim\operatorname{diam} Y$ give
 \begin{align*}
 r^s\lesssim\min\left\{r,\operatorname{diam} A\right\}^{s}
 \lesssim\eta(A\cap B(a,r))\leq\sum_{i=1}^{\infty}\eta(B(x_i,r_i))
 \lesssim\sum_{i=1}^{\infty}r_i^s.
 \end{align*}
By taking the infimum over these covers, we conclude that
$\mathcal H_\infty^s(A\cap B(a,r))\gtrsim r^s$, where
$\mathcal H_\infty^s$ denotes the Hausdorff $s$-content (see, for instance,
\cite[p.~31]{HeinonenKoskela98} for its precise definition).
Let $q':=\min\{q,N\}$.
Since the two copies of $Y$ are compact, it follows that they are
complete and locally compact. They are also Ahlfors $N$-regular.
By H\"older's inequality, we obtain a weak $(1,q')$-Poincar\'e
inequality on $Y_+$ and $Y_-$ from their weak
$(1,1)$-Poincar\'e inequality.
Since $A$ is closed in both copies, the identification along $A$
is isometric, and the preceding content estimate holds with
$N-q'<s<N$, it follows from
\cite[Theorem~6.15]{HeinonenKoskela98} that $X$ supports a weak
$(1,q')$-Poincar\'e inequality in the formulation for continuous
functions. Since $X$ is complete and satisfies the doubling
condition, it follows from \cite[Theorem~2]{Keith03} that
\eqref{PS} holds with exponent $q'$, also when $q'=1$.
Finally, by H\"older's inequality, we obtain a weak
$(1,q)$-Poincar\'e inequality on $X$.

Let $f:=\mathbf 1_{Y_+\setminus A}\in L^1(\mu)$, and, for $r\in(0,1]$, write
$A_r:=\{x\in Y:\inf_{a\in A}\rho_Y(x,a)\leq r\}$.
Let $\{a_i\}_{i=1}^m\subset A$ be a maximal $r$-separated set.
From the $s$-regularity of $\eta$, $\operatorname{diam} A\sim\operatorname{diam} Y\sim1$, and
the disjointness of the balls $B(a_i,\frac r2)$, we deduce that
$mr^s\lesssim\eta(A)\lesssim1$.
Since $A_r\subset\bigcup_{i=1}^m\{x\in Y:\rho_Y(x,a_i)<2r\}$,
the $N$-regularity of $\nu$ gives
\[
 \nu(A_r)\lesssim mr^N\lesssim r^{N-s}=r^\kappa.
\]
By \eqref{eq:glued-distance}, we find that $x\in A_r$ and
$\rho_Y(x,y)\leq r$ whenever $\rho(x_+,y_-)\leq r$.
The $N$-regularity of $Y$ then gives
\[
 \iint_{\genfrac{}{}{0pt}{}{x,y\in Y}{\rho(x_+,y_-)\leq r}}
 \,d\nu(x)\,d\nu(y)
 \lesssim r^N\nu(A_r)\lesssim r^{N+\kappa}.
\]
Since $\nu(A)=0$, it follows that only pairs from different
copies contribute to the energies. By splitting these pairs
according to $2^{-j-1}<\rho(x_+,y_-)\leq2^{-j}$ with
$j\in\mathbb Z_+$ and using \eqref{eq:glued-volume}, we obtain
\begin{align*}
 \mathcal E^\rho_{\lambda,\gamma,p}(f;X)
 &\lesssim\lambda^p\sum_{j\in\mathbb Z_+}
 (2^{-j})^{p\gamma-N}
 \iint_{\genfrac{}{}{0pt}{}{x,y\in Y}{
 2^{-j-1}<\rho(x_+,y_-)\leq2^{-j}}}
 \,d\nu(x)\,d\nu(y)\\
 &\lesssim\lambda^p\sum_{j\in\mathbb Z_+}
 (2^{-j})^{p\gamma-N}(2^{-j})^{N+\kappa}
 =\lambda^p\sum_{j\in\mathbb Z_+}2^{-j(p\gamma+\kappa)}
\end{align*}
and
\begin{align*}
 \mathcal E^V_{\lambda,\gamma,p}(f;X)
 &\lesssim\lambda^p\sum_{j\in\mathbb Z_+}
 (2^{-j})^{Np\gamma-N}
 \iint_{\genfrac{}{}{0pt}{}{x,y\in Y}{
 2^{-j-1}<\rho(x_+,y_-)\leq2^{-j}}}
 \,d\nu(x)\,d\nu(y)\\
 &\lesssim\lambda^p\sum_{j\in\mathbb Z_+}
 (2^{-j})^{Np\gamma-N}(2^{-j})^{N+\kappa}
 =\lambda^p\sum_{j\in\mathbb Z_+}2^{-j(Np\gamma+\kappa)}.
\end{align*}
From the choice of $\kappa$, it follows that $p\gamma+\kappa\in(0,\infty)$
in the case (i) and $Np\gamma+\kappa\in(0,\infty)$ in the case (ii).
Thus, in the corresponding case,
$\mathcal E^\phi_{\lambda,\gamma,p}(f;X)\lesssim\lambda^p\to0$ as
$\lambda\to0^+$, where the implicit positive constant depends only on
$p$, $\gamma$, $N$, $\kappa$, and the regularity constants of $Y$ and $A$.

Fix a ball $B=B(a,r_B)$ with $a\in A$ and $r_B\in(0,1)$.
The map exchanging $x_+$ and $x_-$ for any $x\in Y$
preserves distances and $\mu$. Since it fixes $a$, it follows that
$B$ is mapped onto itself. Using this and $\mu(A)=0$, we obtain
$\mu(B\cap Y_+)=\mu(B\cap Y_-)=\frac{\mu(B)}{2}$. Thus,
\[
 \fint_B\fint_B|f(x)-f(y)|^p\,d\mu(x)\,d\mu(y)=\frac12.
\]
This, together with \eqref{eq:H-gradient}, further implies that
$\int_Xg^p\,d\mu\gtrsim\frac{\mu(B)}{r_B^p}$ for any
Haj\l asz upper gradient $g\in L^p(\mu)$ when $p\in(1,\infty)$.
When $p=1$, we have $q=1$. If $f\in\dot{\mathrm{BV}}(\mu)$, then the
Poincar\'e inequality in \cite[Remark~3.5]{Miranda03} gives
$|Df|(X)\gtrsim\frac{\mu(B)}{r_B}$.
Thus, $\mathcal E_p(f;X)\in(0,\infty]$.
This completes the proof of Proposition~\ref{prop:counterexample}.
\end{proof}

\smallskip
\noindent\textbf{Acknowledgements}\quad
The authors acknowledge the use of AI tools during the exploratory stage
of this project. All mathematical arguments and proofs in the final manuscript were checked and
written by the authors.

Yirui Zhao would like to thank Yinqin Li for some helpful discussions.

\bigskip
\noindent\begin{minipage}{\linewidth}
\noindent Feng Dai
\smallskip

\noindent Department of Mathematical and
Statistical Sciences, University of Alberta
Edmonton, Alberta T6G 2G1, Canada

\smallskip

\noindent\textit{E-mail:} \texttt{fdai@ualberta.ca}
\end{minipage}

\bigskip

\noindent\begin{minipage}{\linewidth}
\noindent Tuomas P.\ Hyt\"onen

\smallskip

\noindent Department of Mathematics and Systems Analysis, Aalto University,
P.O. Box 11100 (Otakaari~1), \mbox{FI-00076} Aalto, Finland

\smallskip

\noindent\textit{E-mail:} \texttt{tuomas.hytonen@aalto.fi}
\end{minipage}

\bigskip

\noindent\begin{minipage}{\linewidth}
\noindent Dachun Yang,
Wen Yuan, and Yirui Zhao

\medskip

\noindent  Laboratory of Mathematics and Complex Systems
(Ministry of Education of China),
School of Mathematical Sciences, Institute for Advanced Study,
Beijing Normal University,
Beijing 100875, The People's Republic of China

\smallskip

\noindent\textit{E-mails:} \texttt{dcyang@bnu.edu.cn} (D. Yang)

\noindent\phantom{\textit{E-mails:}} \texttt{wenyuan@bnu.edu.cn} (W. Yuan)

\noindent\phantom{\textit{E-mails:}} \texttt{yiruizhao@mail.bnu.edu.cn} (Y. Zhao)
\end{minipage}


\begin{thebibliography}{99}
\setlength{\itemsep}{-3.6pt}\setlength{\parsep}{0pt}

\bibitem{Barbieri2011}
D. Barbieri,
Approximations of Sobolev norms in Carnot groups,
Commun. Contemp. Math. 13 (2011), 765--794.

% Bibliography spacing is set uniformly at the start of the list.

\bibitem{bb11}
A. Bj\"orn and J. Bj\"orn,
Nonlinear Potential Theory on Metric Spaces,
EMS Tracts in Mathematics 17, European Mathematical Society (EMS), Z\"urich, 2011.

% Bibliography spacing is set uniformly at the start of the list.

\bibitem{BBM2001}
J. Bourgain, H. Brezis and P. Mironescu,
Another look at Sobolev spaces,
in: Optimal Control and Partial Differential Equations,
pp. 439--455, IOS, Amsterdam, 2001.

% Bibliography spacing is set uniformly at the start of the list.

\bibitem{BBM2002}
J. Bourgain, H. Brezis and P. Mironescu,
Limiting embedding theorems for $W^{s,p}$ when $s\uparrow1$ and applications,
J. Anal. Math. 87 (2002), 77--101.

% Bibliography spacing is set uniformly at the start of the list.

\bibitem{BourgainNguyen2006}
J. Bourgain and H.-M. Nguyen,
A new characterization of Sobolev spaces,
C. R. Math. Acad. Sci. Paris 343 (2006), 75--80.

% Bibliography spacing is set uniformly at the start of the list.

\bibitem{BrezisNguyenVMO}
H. Brezis and H.-M. Nguyen,
On a new class of functions related to VMO,
C. R. Math. Acad. Sci. Paris 349 (2011), 157--160.

% Bibliography spacing is set uniformly at the start of the list.

\bibitem{BrezisNguyenImages}
H. Brezis and H.-M. Nguyen,
Non-local functionals related to the total variation and connections with
image processing,
Ann. PDE 4 (2018), Paper No. 9, 77 pp.

% Bibliography spacing is set uniformly at the start of the list.

\bibitem{BrezisNguyen2020}
H. Brezis and H.-M. Nguyen,
Non-local, non-convex functionals converging to Sobolev norms,
Nonlinear Anal. 191 (2020), Paper No. 111626, 9 pp.

% Bibliography spacing is set uniformly at the start of the list.

\bibitem{BSVY}
H. Brezis, A. Seeger, J. Van Schaftingen and P.-L. Yung,
Families of functionals representing Sobolev norms,
Anal. PDE 17 (2024), 943--979.

% Bibliography spacing is set uniformly at the start of the list.

\bibitem{BVY2021}
H. Brezis, J. Van Schaftingen and P.-L. Yung,
A surprising formula for Sobolev norms,
Proc. Natl. Acad. Sci. USA 118 (2021), Paper No. e2025254118, 6 pp.

% Bibliography spacing is set uniformly at the start of the list.

\bibitem{BVYLorentz2021}
H. Brezis, J. Van Schaftingen and P.-L. Yung,
Going to Lorentz when fractional Sobolev, Gagliardo and Nirenberg estimates fail,
Calc. Var. Partial Differential Equations 60 (2021), Paper No. 129, 12 pp.

% Bibliography spacing is set uniformly at the start of the list.

\bibitem{BursteddeWilcoxGhattas11}
C. Burstedde, L. C. Wilcox and O. Ghattas,
p4est: Scalable algorithms for parallel adaptive mesh refinement
on forests of octrees,
SIAM J. Sci. Comput. 33 (2011), 1103--1133.

% Bibliography spacing is set uniformly at the start of the list.

\bibitem{CaputoCavallucci25}
E. Caputo and N. Cavallucci,
A geometric approach to Poincar\'e inequality and Minkowski content
of separating sets,
Int. Math. Res. Not. IMRN 2025 (2025), Paper No. rnae276, 30 pp.

% Bibliography spacing is set uniformly at the start of the list.

\bibitem{CuiLamLu2014}
X. Cui, N. Lam and G. Lu,
New characterizations of Sobolev spaces on the Heisenberg group,
J. Funct. Anal. 267 (2014), 2962--2994.

% Bibliography spacing is set uniformly at the start of the list.

\bibitem{DaiEtAlOperators2026}
F. Dai, Y. Li, D. Yang, W. Yuan and Y. Zhao,
Sharp weak-type estimate for local lifted Hardy--Littlewood
maximal operators with applications to generators of linear
operator families and Hardy(--Sobolev) spaces,
Adv. Math. 490 (2026), Paper No. 110822, 55 pp.

% Bibliography spacing is set uniformly at the start of the list.

\bibitem{dlyyz26b}
F. Dai, Y. Li, D. Yang, W. Yuan and Y. Zhao,
Characterizations and their applications of Sobolev seminorms via
lifted sharp maximal functions on complete Poincar\'e spaces,
Adv. Math. 502 (2026), part B, Paper No. 111153, 71 pp.

% Bibliography spacing is set uniformly at the start of the list.

\bibitem{DaiEtAlMetric2022}
F. Dai, X. Lin, D. Yang, W. Yuan and Y. Zhang,
Poincar\'e inequality meets Brezis--Van Schaftingen--Yung
formula on metric measure spaces,
J. Funct. Anal. 283 (2022), Paper No. 109645, 52 pp.

% Bibliography spacing is set uniformly at the start of the list.

\bibitem{DaiEtAlBall2023}
F. Dai, X. Lin, D. Yang, W. Yuan and Y. Zhang,
Brezis--Van Schaftingen--Yung formulae in ball Banach function spaces
with applications to fractional Sobolev and Gagliardo--Nirenberg inequalities,
Calc. Var. Partial Differential Equations 62 (2023), Paper No. 56, 73 pp.

% Bibliography spacing is set uniformly at the start of the list.

\bibitem{DYYZTruncated26}
F. Dai, D. Yang, W. Yuan and Y. Zhao,
Sharp truncated difference estimates for nonlocal nonconvex functionals
defined by superlevel sets,
arXiv:2610.04259v2.

% Bibliography spacing is set uniformly at the start of the list.

\bibitem{DiMarinoSquassina19}
S. Di Marino and M. Squassina,
New characterizations of Sobolev metric spaces,
J. Funct. Anal. 276 (2019), 1853--1874.

% Bibliography spacing is set uniformly at the start of the list.

\bibitem{DominguezMilman}
O. Dom\'{i}nguez and M. Milman,
New Brezis--Van Schaftingen--Yung--Sobolev type inequalities connected with
maximal inequalities and one parameter families of operators,
Adv. Math. 411 (2022), Paper No. 108774, 76 pp.

% Bibliography spacing is set uniformly at the start of the list.

\bibitem{DominguezEtAlBesov2023}
O. Dom\'inguez, A. Seeger, B. Street, J. Van Schaftingen and P.-L. Yung,
Spaces of Besov-Sobolev type and a problem on nonlinear approximation,
J. Funct. Anal. 284 (2023), Paper No. 109775, 50 pp.

% Bibliography spacing is set uniformly at the start of the list.

\bibitem{FasslerOrponen19}
K. F\"assler and T. Orponen,
Metric currents and the Poincar\'e inequality,
Calc. Var. Partial Differential Equations 58 (2019), Paper No. 69, 20 pp.

% Bibliography spacing is set uniformly at the start of the list.

\bibitem{Frank2024}
R. L. Frank,
A characterization of $\dot W^{1,p}(\mathbb R^d)$,
Pure Appl. Funct. Anal. 9 (2024), 53--68.

% Bibliography spacing is set uniformly at the start of the list.

\bibitem{gp25}
M. Gobbino and N. Picenni,
Gamma-liminf estimate for a class of non-local approximations of Sobolev
and BV norms,
J. Funct. Anal. 289 (2025), Paper No. 111106, 25 pp.

% Bibliography spacing is set uniformly at the start of the list.

\bibitem{HanXuZhu2025}
B.-X. Han, Z.-F. Xu and Z.-N. Zhu,
A new characterization of Sobolev spaces on Lipschitz differentiability spaces,
arXiv:2504.16657.

% Bibliography spacing is set uniformly at the start of the list.

\bibitem{HKT17}
T. Heikkinen, P. Koskela and H. Tuominen,
Approximation and quasicontinuity of Besov and Triebel--Lizorkin functions,
Trans. Amer. Math. Soc. 369 (2017), 3547--3573.

% Bibliography spacing is set uniformly at the start of the list.

\bibitem{HeinonenKoskela98}
J. Heinonen and P. Koskela,
Quasiconformal maps in metric spaces with controlled geometry,
Acta Math. 181 (1998), 1--61.

% Bibliography spacing is set uniformly at the start of the list.

\bibitem{HKST}
J. Heinonen, P. Koskela, N. Shanmugalingam and J. T. Tyson,
Sobolev Spaces on Metric Measure Spaces:
An Approach Based on Upper Gradients,
New Mathematical Monographs 27, Cambridge University Press, Cambridge, 2015.

% Bibliography spacing is set uniformly at the start of the list.

\bibitem{HuEtAlHigherOrder2025}
P. Hu, Y. Li, D. Yang, W. Yuan and Y. Zhang,
Sharp Brezis--Seeger--Van Schaftingen--Yung formulae for higher-order gradients
in ball Banach function spaces,
J. Geom. Anal. 35 (2025), Paper No. 388, 76 pp.

% Bibliography spacing is set uniformly at the start of the list.

\bibitem{HytKairema}
T. Hyt\"onen and A. Kairema,
Systems of dyadic cubes in a doubling metric space,
Colloq. Math. 126 (2012), 1--33.

% Bibliography spacing is set uniformly at the start of the list.

\bibitem{HK:W1p}
T. Hyt\"onen and R. Korte,
Characterisations of Sobolev spaces and constant functions over metric spaces,
arXiv:2508.07801v3.

% Bibliography spacing is set uniformly at the start of the list.

\bibitem{HytonenYangYuanZhao2026}
T. P. Hyt\"onen, D. Yang, W. Yuan and Y. Zhao,
Weak-type characterizations of Sobolev and bounded variation spaces on
metric measure spaces,
arXiv:2608.24106.

% Bibliography spacing is set uniformly at the start of the list.

\bibitem{Keith03}
S. Keith,
Modulus and the Poincar\'e inequality on metric measure spaces,
Math. Z. 245 (2003), 255--292.

% Bibliography spacing is set uniformly at the start of the list.

\bibitem{Laakso00}
T. J. Laakso,
Ahlfors $Q$-regular spaces with arbitrary $Q>1$ admitting weak
Poincar\'e inequality,
Geom. Funct. Anal. 10 (2000), 111--123.

% Bibliography spacing is set uniformly at the start of the list.

\bibitem{LahtiPinamontiZhou2024}
P. Lahti, A. Pinamonti and X. Zhou,
BV functions and nonlocal functionals in metric measure spaces,
J. Geom. Anal. 34 (2024), Paper No. 318, 34 pp.

% Bibliography spacing is set uniformly at the start of the list.

\bibitem{LiEtAlWeighted2026}
Y. Li, D. Yang, W. Yuan,
Y. Zhang and Y. Zhao,
Sharp weighted
Cohen--Dahmen--Daubechies--DeVore inequality
with applications to (weighted) critical Sobolev spaces,
Gagliardo--Nirenberg inequalities,
and Muckenhoupt weights, Sci. China
Math. (2026), DOI: 10.1007/s11425-025-2527-6.

% Bibliography spacing is set uniformly at the start of the list.

\bibitem{LinYangYangYuanZhang}
X. Lin, D. Yang, S. Yang, W. Yuan and Y. Zhang,
$\Gamma$-convergence of weak-type nonlocal functionals on bounded domains,
arXiv:2608.18414.

% Bibliography spacing is set uniformly at the start of the list.

\bibitem{LinEtAlFullRange2026}
X. Lin, D. Yang, S. Yang, W. Yuan and Y. Zhang,
Full-range $\Gamma$-convergence of weak-type nonlocal functionals,
arXiv:2610.07001.

% Bibliography spacing is set uniformly at the start of the list.

\bibitem{LinEtAlCompactness2026}
X. Lin, D. Yang, W. Yuan and Y. Zhang,
Compactness and its applications of Sobolev spaces associated with
ball Banach function spaces,
arXiv:2606.23137.

% Bibliography spacing is set uniformly at the start of the list.

\bibitem{MallickNguyen2020}
A. Mallick and H.-M. Nguyen,
Exponential integrability in the spirit of Moser--Trudinger's inequalities
of functions with finite non-local, non-convex energy,
Pure Appl. Funct. Anal. 5 (2020), 1153--1165.

% Bibliography spacing is set uniformly at the start of the list.

\bibitem{MattilaSaaranen09}
P. Mattila and P. Saaranen,
Ahlfors--David regular sets and bilipschitz maps,
Ann. Acad. Sci. Fenn. Math. 34 (2009), 487--502.

% Bibliography spacing is set uniformly at the start of the list.

\bibitem{Miranda03}
M. Miranda Jr.,
Functions of bounded variation on ``good'' metric spaces,
J. Math. Pures Appl. (9) 82 (2003), 975--1004.

% Bibliography spacing is set uniformly at the start of the list.

\bibitem{Nguyen2006}
H.-M. Nguyen,
Some new characterizations of Sobolev spaces,
J. Funct. Anal. 237 (2006), 689--720.

% Bibliography spacing is set uniformly at the start of the list.

\bibitem{Nguyen2008}
H.-M. Nguyen,
Further characterizations of Sobolev spaces,
J. Eur. Math. Soc. (JEMS) 10 (2008), 191--229.

% Bibliography spacing is set uniformly at the start of the list.

\bibitem{Nguyen2011}
H.-M. Nguyen,
Some inequalities related to Sobolev norms,
Calc. Var. Partial Differential Equations 41 (2011), 483--509.

% Bibliography spacing is set uniformly at the start of the list.

\bibitem{NguyenDegree2017}
H.-M. Nguyen,
A refined estimate for the topological degree,
C. R. Math. Acad. Sci. Paris 355 (2017), 1046--1049.

% Bibliography spacing is set uniformly at the start of the list.

\bibitem{NguyenSurvey}
H.-M. Nguyen,
Characterizations of the Sobolev norms and the total variation via
nonlocal functionals, and related problems,
C. R. Math. Acad. Sci. Paris 363 (2025), 1429--1455.

% Bibliography spacing is set uniformly at the start of the list.

\bibitem{NguyenEtAlMagnetic2018}
H.-M. Nguyen, A. Pinamonti, M. Squassina and E. Vecchi,
New characterizations of magnetic Sobolev spaces,
Adv. Nonlinear Anal. 7 (2018), 227--245.

% Bibliography spacing is set uniformly at the start of the list.

\bibitem{NguyenSquassinaAnisotropic2019}
H.-M. Nguyen and M. Squassina,
On anisotropic Sobolev spaces,
Commun. Contemp. Math. 21 (2019), Paper No. 1850017, 13 pp.

% Bibliography spacing is set uniformly at the start of the list.

\bibitem{NguyenSquassina2019}
H.-M. Nguyen and M. Squassina,
On Hardy and Caffarelli--Kohn--Nirenberg inequalities,
J. Anal. Math. 139 (2019), 773--797.

% Bibliography spacing is set uniformly at the start of the list.

\bibitem{Picenni2026}
N. Picenni,
Lower bounds and Poincar\'e inequality for the
Brezis--Seeger--Van Schaftingen--Yung functionals,
arXiv:2610.05929.

% Bibliography spacing is set uniformly at the start of the list.

\bibitem{Poliakovsky2022}
A. Poliakovsky,
Some remarks on a formula for
Sobolev norms due to Brezis, Van Schaftingen and Yung,
J. Funct. Anal. 282 (2022), Paper No. 109312, 47 pp.

% Bibliography spacing is set uniformly at the start of the list.

\bibitem{Ponce2004}
A. C. Ponce,
An estimate in the spirit of Poincar\'e's inequality,
J. Eur. Math. Soc. (JEMS) 6 (2004), 1--15.

% Bibliography spacing is set uniformly at the start of the list.

\bibitem{Semmes96}
S. Semmes,
Finding curves on general spaces through quantitative topology,
with applications to Sobolev and Poincar\'e inequalities,
Selecta Math. (N.S.) 2 (1996), 155--295.

% Bibliography spacing is set uniformly at the start of the list.

\bibitem{SundarSampathBiros08}
H. Sundar, R. S. Sampath and G. Biros,
Bottom-up construction and $2{:}1$ balance refinement of linear octrees
in parallel,
SIAM J. Sci. Comput. 30 (2008), 2675--2708.

% Bibliography spacing is set uniformly at the start of the list.

\bibitem{yz11}
D. Yang and Y. Zhou,
New properties of Besov and Triebel--Lizorkin spaces on RD-spaces,
Manuscripta Math. 134 (2011), 59--90.

% Bibliography spacing is set uniformly at the start of the list.

\bibitem{ZhaoEtAlFrank2024}
Y. Zhao, Y. Li, D. Yang, W. Yuan and Y. Zhang,
Generalized Frank characterizations of Muckenhoupt weights and
homogeneous ball Banach Sobolev spaces,
Adv. Math. 458 (2024), part A, Paper No. 109957, 63 pp.

\end{thebibliography}
\end{document}